\documentclass[12pt]{amsart}
\usepackage{amsmath,amssymb}
\usepackage{amscd}
\usepackage{mathrsfs}
\usepackage[all]{xy}
\usepackage{pstricks}
\UseComputerModernTips
\usepackage{graphicx}
\DeclareGraphicsExtensions{.eps} {\bf }

\theoremstyle{plain}
\newtheorem{theorem}{Theorem}[section]

\newtheorem{lemma}{Lemma}[section]

\newcommand{\edge}[1]{\ar@{-}[#1]}
\theoremstyle{definition}
\newtheorem{definition}{Definition}[section]
\newtheorem{example}{Example}[section]
\theoremstyle{remark}
\newtheorem{remark}{Remark}[section]

\def\Spec{\operatorname{Spec}}

\numberwithin{equation}{section} \numberwithin{equation}{section}
\textheight=8.75in \theoremstyle{plain}
\begin{document}
\title[ Zero-divisor graphs with seven vertices]{Zero-divisor graph with seven vertices }
\author{Xinyun Zhu}
\address{Department of Mathematics\\  University of Texas of Permian Basin\\ Odessa, TX 79762}
\email{zhu\_x@utpb.edu} \date{\today}
\begin{abstract} 

 Inspired by the work in \cite{sauer} regarding the classification of all the zero-divisor graphs with six vertices, we  obtain all the zero-divisor graphs with seven vertices. Hence we classify all the zero-divisor commutative semigroups with 8 elements. We also obtain all the connected graphs with seven vertices which satisfies the necessary condition $\star$ of zero-divisor graphs given in \cite{fl} but are not the zero-divisor graphs. 
\end{abstract}
\maketitle
\section{Introduction}

Let $S$ be a commutative semigroup with zero.  The {\it zero-divisor graph}  of $S,$ denoted $\Gamma(S),$  is the graph with vertices corresponding to the nonzero zero-divisors of $S,$ and distinct zero-divisors $x$ and $y$ are adjacent if and only if $xy=0.$ A semigroup is called a {\it zero-divisor semigroup} if it consists solely of zero-divisors.

Given a connected graph $G,$ let $V(G)$ denote the vertices set of $G$ and $E(G)$ the edges set of $G$. 
 For any two distinct vertices $a$ and $b$, define $ab=0$ if $a$ and $b$ are adjacent  otherwise define $ab=ba\in V(G)$.  Using this way, we can get a family of multiplication tables corresponding to $G$.  If there exists  a multiplication table which defines a semigroup, denoted $S(G)$, then $G$ is a zero-divisor graph.

In \cite{leck}, all the non-isomorphic connected graphs with six vertices are given. Based on \cite{leck}, Sauer classified all the zero-divisor graphs with six vertices in \cite{sauer}.  

The purpose of this paper is to extend Sauer's work \cite{sauer} to zero-divisor graphs with seven vertices.

This paper is organized as follows. In section~\ref{section1}, we compare the necessary conditions for the zero-divisor graphs in \cite{fl} and conclude that condition (4) in Theorem~\ref{condition} is the most important condition. In section~\ref{section2}, we give a method of filling the multiplication table of zero-divisor graphs and find some properties of zero-divisor graphs with seven vertices. In section~\ref{section3}, we get all the  zero-divisor graphs with seven vertices in \cite{RW}. Hence we classify all the zero-divisor commutative semigroups with 8 elements. In section~\ref{section4}, we get all the  non-zero divisor graphs in \cite{RW} which satisfy the necessary condition (4) in Theorem~\ref{condition}.
\section{Relations among the  Necessary conditions regarding the zero-divisor graphs given in Theorem 1 of \cite{fl}.}\label{section1}
\begin{definition}\label{neighborhood} Given a connected graph $G$. Let $a$ be a vertex of $G.$ We define $N(a)$ be a set of all vertices which is adjacent to $a$ and $\overline{N(a)}=N(a)\cup \{a\}$.
\end{definition}
In \cite[Theorem 1]{fl}, the following necessary conditions for a zero-divisor graph were given. 
\begin{theorem}\cite[Theorem 1]{fl}\label{condition} If $G$ is the graph of a semigroup then $G$ satisfies all of the following conditions. 
\begin{enumerate}
\item
$G$ is connected.
\item
Any two vertices of $G$ are connected by a path with $\le$ 3 edges.
\item
If $G$ contains a cycle then the core of $G$ is a union of quadrilaterals and triangles, and any vertex not in the core of $G$ is an end.
\item
For any pair $x,\, y$ of nonadjacent vertices of $G$, there is a vertex $z$ with $N(x)\cup N(y)\subset \overline{N(z)}.$
\end{enumerate}
\end{theorem}
In this section, we find the relations among those necessary conditions in Theorem~\ref{condition}. 

\begin{lemma} $N(a)\cup N(b)\subset \overline{N(c)}$ implies $d(a,b)\le 3.$
\end{lemma}
\begin{proof}

If $N(a)\cup N(b)\subset \overline{N(c)}-N(c)$, then $c\in N(a)$ or $c\in N(b).$  Assume $c\in N(b)$ and $f\in N(a)$, then we have path $a-f-c-b$.  Hence $d(a,b)\le 3.$
  
  If $N(a)\cup N(b)\subset N(c)$ and $e\in N(a)\cap N(b)$, then we have path $a-e-b.$  Hence $d(a,b)\le 3.$
  
  Suppose $N(a)\cup N(b)\subset N(c)$ and $N(a)\cap N(b)=\emptyset$.  If $d\in N(a)$ and $e\in N(b),$ then $\{a,c\}\subset N(d)$ and $\{b,c\}\subset N(e).$  Hence there exists  a vertex $f$ such that 
  \[\{a,b,c\}\subset N(d)\cup N(e)\subset \overline{N(f)}\] 
  Hence we have path $a-f-b$ and $d(a,b)\le 3.$ 
  \end{proof}   
\begin{lemma} $d(a,b)\le 3$  does not imply $N(a)\cup N(b)\subset \overline{N(c)}$.
\end{lemma}
\begin{proof}
The connected graph $G$  with $V(G)=\{a,b, c,d, e,f,g\}$ and $E(G)$ defined by $N(a)=\{b, e,g\},$ $N(b)=\{a,c\},\, N(c)=\{b, d\}, N(d)=\{c,f\},N(e)=\{a,f\}, N(f)=\{d,e,g\}, N(g)=\{a,f\}$ is  a counter example. 
\end{proof}
\begin{lemma} $d(a,b)\le 3$ does not imply a union of quadrilaterals and triangles.
\end{lemma}
\begin{proof} The circuit $a-b-c-d-e-f-a$ gives a counter-example.
\end{proof}
\begin{lemma}a union of quadrilaterals and triangles does not imply $d(a,b)\le 3$.  
\end{lemma}
\begin{proof} \[N(a)=\{b,c\},\, N(b)=\{a,c,d,e\},\, N(c)=\{a,b,f,g\},\, N(d)=\{b,e,h\},\]
\[N(e)=\{b,d,f,i\},\, N(f)=\{c,e,g,i\},\, N(g)=\{c,f,j\},\,N(h)=\{d,i,k\},\]  
\[ N(i)=\{e,f,h,j, k\},\, N(j)=\{g,i,k\},\, N(k)=\{h,i,j\}\]
gives a counter-example.
\end{proof}
\begin{lemma}$N(a)\cup N(b)\subset \overline{N(c)}$ implies a union of quadrilaterals and triangles.  
\end{lemma}
\begin{proof}
Suppose the induced subgraph is a circuit $a_1a_2a_3\cdots a_ka_1.$ Then $\{a_2,a_k\}\subset N(a_1),$ and $\{a_2,a_4\}\subset N(a_3).$  It follows that there exists a vertex $f$ such that $N(a_1)\cup N(a_3)\subset \overline{N(f)}.$ Since the induced subgraph is a circuit $a_1a_2a_3\cdots a_ka_1,$  we get $f\ne a_1,\, f\ne a_2,\, f\ne a_3.$  Hence $a_1a_2$ is an edge of quadrilateral $a_1a_2fa_ka_1 $ and $a_2a_3$ is an edge of quadrilateral $a_2a_3fa_ka_2 $. We have a short circuit $a_kfa_4\cdots a_{k-1}a_k.$  Continue this process we get $a_1a_2a_3\cdots a_ka_1$ is a union of quadrilaterals and triangles.
\end{proof}

Hence condition (4) in Theorem~\ref{condition} is the most important necessary condition. Henceforth, we call condition (4) in Theorem~\ref{condition} the condition $\star.$
\section{a way to fill the multiplication table of zero-divisor graphs and some properties of zero-divisor graphs with seven vertices}\label{section2}
\subsection{a way to fill the multiplication table of zero-divisor graphs}
\begin{remark}  For any $\{a,b\}\in V(G),$ if $a\ne b$, then $ab=0$ if and only if $a-b$ is an edge of $G$ and $ab\in V(G)$ if and only if $a-b$ is not an edge of $G.$  For any $a\in V(G),$ $a^2\in V(G)\cup\{0\}.$
\end{remark}
\begin{definition} Let $G$ be a connected graph $G$  with $V(G)=\{a,b, c,d, e,f,g\}$. Let $x\in V(G).$  The spectrum of $x$ is defined  to be $\Spec(a)=(xa,xb,xc, xd, xe, xf, xg).$  Given $y,z\in V(G),$  define $yz=x$ if $\Spec(yz)=\Spec(x).$
\end{definition}
\begin{definition} Let $a$ and $b$ be non-adjacent vertices of $G.$ Define $D(ab)$ to be a set of vertices such that $c\in D(ab)$ if and only if $N(a)\cup N(b)\subset \overline{N(c)}.$ For  vertex $a\in V(G)$, define $D(a^2)$ to be a set such that $c\in D(a^2)$ if and only if $c=0$ or $N(a)\subset \overline{N(c)}.$
\end{definition}
\begin{example}
The graph $G$ defined by
$N(a)=N(b)=\{x,y\},$  $N(c)=\{y\},$ $N(x)=\{a,b,y,z,w\},$ $N(y)=\{a,b,c,x\},$ $N(z)=\{x\}$, $N(w)=\{x\}$
is a zero-divisor graph because by Light's associativity test, one can show the multiplication given by following table is associative. 
\begin{center}
\begin{tabular}{|l|c|r|r|r|r|r|r|}
\hline
$\bullet$&$a$&$b$&$c$&$x$&$y$&$z$&$w$  \\ \hline
$a$&$a$&$a$&$a$&0&0&$a$&$a$  \\ \hline

$b$&$a$&$b$&$a$&0&$0$&$a$&$a$  \\ \hline
$c$&$a$&$a$&$c$&$x$&$0$&$a$&$a$  \\ \hline
$x$&0&$0$&$x$&$x$&$0$&$0$&0  \\ \hline
$y$&0&$0$&$0$&$0$&$y$&$y$&$y$  \\ \hline
$z$&$a$&$a$&$a$&$0$&$y$&$z$&$z$  \\ \hline
$w$&$a$&$a$&$a$&$0$&$y$&$z$&$w$\\ \hline

\end{tabular}
\end{center}

$\Spec(a)=(a, a,a, 0, 0, a,a),$ $\Spec(b)=(a,b,a,0, 0, a,a),$ $\Spec(c)=(a,a,c,x,0, a,a),$ $\Spec(x)=(0,0,x,x,0,0,0),$
$\Spec(y)=(0,0,0,0,y,y,y),$ $\Spec(z)=(a,a,a,0,y,z,z),$ $\Spec(w)=(a,a,a,0,y,z,w).$

$D(ab)=\{a,b,y\}$ because
\[N(a)\cup N(b)=\{x,y\}\subset N(a)\cap N(b)\cap \overline{N(y)}\] 
We define $ab=a$ because 
 \begin{enumerate}\item if  $ab=b,$ then 
$(ba)c=bc=a\ne b=ba=b(ac);$ \item
if $ab=y,$ then $y^2=0.$ It follows that $a^2b=ab=y\ne 0=ay=a(ab).$
\end{enumerate}
\end{example}
\subsection{some properties of zero-divisor graphs with seven vertices}
\begin{lemma}Let $G$ be a zero-divisor graph and $d(x)=\Delta,$  where $d(x)=|N(x)|$ and  $\Delta$ is the maximum degree. Then for any vertex $y\in V(G),$  $d(x,y)\le 2.$
\end{lemma}
\begin{proof} By  \cite[Lemma 4.1]{DGSW}, we get $N(y)\subseteq N(x)$ for all $y\notin N(x).$ 
\end{proof}
\begin{lemma} Let $G$ be a zero-divisor graph with seven vertices.  If $d(x)=\Delta,$ then $d(x)\ge 4.$
\end{lemma}
\begin{proof} Suppose $\Delta=3.$ Let $N(a)=\{x,y,z\}.$  Since $|N(a)|=3, $  one gets that $N(b),\,N(c)$ and $N(w)$ are subsets of $N(a).$  Suppose $N(x)=\{a,b,c\}.$  Since $|N(x)=3,$ one gets $w\in N(y)$ or $w\in N(z)$ and $x\notin N(y)$ or $x\notin N(z).$  It follows that $N(y)$ is not a subset of  $N(x)$  or $N(z)$ is not  a subset of $N(x).$ This is a contradiction.  Hence one gets $N(x)\cap\{a,b,c\}=N(y)\cap\{a,b,c\}=N(z)\cap\{a,b,c\}=2.$ If $d(x)=3,$ then one get $xy$ or $xz$ is an edge. Hence $d(y)=3$ or $d(z)=3.$ If $xy$ is an edge, then $xz$ is not an edge.But $N(z)$ is not an subset of $N(x).$  Contradiction. Hence $|N(x)|=|N(y)|=|N(z)|=2.$ Let $N(b)=\{z\},$ $N(c)=\{y\},$ $N(w)=\{x\}.$  Then there is no $v$ such that $N(y)\cup N(z)\subseteq \overline {N(v)}.$ The proof is completed. 
  
\end{proof}     
\begin{lemma}\label{highest}Let $G$ be a connected graph satisfies the necessary condition $\star$. Let $d(x)=\Delta.$ Then emanating from $x$ an end $w$ results a graph satisfies the necessary condition $\star$.
\end{lemma}
\begin{proof} If $x$ and $y$ are not adjacent, then $N(y)\subseteq N(x)$ by \cite{DGSW}. Hence $N(y)\cup N(w)\subset N(x).$
If $y\in N(x),$ then $N(y)\cup N(w)=N(y)\cup \{x\}=N(y).$

\end{proof} 
\section{zero divisor graph with seven vertices}\label{section3}

\subsection{zero-divisor graphs  with seven vertices produced by applying Theorems in \cite{fl} and \cite{sauer} }
\begin{theorem}\cite[Theorem 3]{fl}\label{fl} The following graphs are the graph of a semigroup.
\begin{enumerate} 
\item A complete graph or a complete graph together with one end.
\item A complete bipartite graph or a complete bipartite graph together with an end.
\item A refinement of a star graph.
\item A graph which is the union of two star graphs whose centers are connected by a single edge.
\end{enumerate}
\end{theorem}

\begin{theorem}
\begin{enumerate} \item
\cite[Theorem IV.1]{sauer} A complete graph together with any number of ends, each of which emanates from one of two vertices, is the graph of a commutative semigroup.
\item
\cite[Theorem IV.2]{sauer} The complete graph on three vertices, from each of which emanates at least one end, is the graph of a commutative semigroup.
\item 
\cite[Theorem IV.3]{sauer} The complete graph on four or more vertices together with any number of ends emanating from at least three different vertices is never the graph of a commutative semigroup.
\item \cite[Theorem IV.4]{sauer} A complete bipartite graph together with any number of ends emanating from the same vertex is the graph of a commutative semigroup.
\item \cite[Theorem IV.5]{sauer}A complete bipartite graph together with two or more ends emanating from at least two distinct vertices is never the graph of a commutative semigroup.
  
\end{enumerate}
\end{theorem}
\begin{theorem}\cite{sauer} Up to isomorphism, there are 67 zero-divisor graphs with six vertices.  They are one star graph, 33 refinements of  a star graph, two double star graphs, two complete bipartite graphs, four complete bipartite graphs with ends emanating from at most two different vertices,  a complete graphs with three vertices and three ends emanating from three different vertices, and twenty-four exceptional cases.
 \end{theorem}


\subsection{a seven vertices zero-divisor graph produced by applying \cite[Lemma 3.14]{DGSW}}
\begin{theorem}\cite[Lemma 3.14]{DGSW}\label{add vertex} Suppose $G$ is a zero divisor graph of a semigroup, $x\in V(G),$ $y\notin V(G)$.  Let $G'$ be a graph defined by $V(G')=V(G)\cup \{y\},$  $N(y)=N(x)$ if $x^2\not=0,$ or $N(y)=\overline{N(x)}$ if $x^2=0.$   Then $G'$ is a zero-divisor graph of a semigroup.
\end{theorem}
\begin{proof} Define $xy=x^2=y^2$ and  $xz=yz$ for all $z\in V(g)-\{y\}.$
\end{proof}

\subsection{list of all the zero-divisor graphs with seven vertices}
\begin{theorem} The following graphs are all the zero-divisor graphs in \cite{RW}. Notice there exists a one-to-one correspondence between zero-divisor graphs and zero-divisor semigroups. Hence we give a classification of all the zero-divisor commutative semigroups with 8 elements.
\begin{center}
 $G270-G272$,$G314-G317$,$G319$,$G379-G382,$ $G384$, $G388$,$G390$,$G392-G393$,$G411$,$G473-G474$,$G476-G481$,$G483$,$G485-G486$,$G493$,$G503$,$G507$,$G513$, $G522$,$G525$,$G551$,$G598-G599$,$G601$,$G603-G604$,$G606$,
$G612-G614$,$G616$,$G618-G620$, $G624$,$G626$, $G629$, $G631$,$G633$,$G639$,$G667-G668$,$G670$,$G671-G672$,$G678$,$G740-G741$,$G743$,$G746-G749$,$G751-G753$,$G755$,$G757-G759$, $G762$,$G764$, $G767$,$G775$,
$G780$,$G786$,$G790-G792$,$G794-G796$,$G798$,$G800-G801$,$G805$, $G812-G815$, $G832$, $G872$, $G884-G891$, $G894$, $G896-G898$, $G902$, $G906$, $G908-G909$, $G913-G916$, $G919-G925$, $G927$, $G929-G930$, $G932$, $G934$, $G939$, $G944$,$G948$,$G950-G952$,$G957$, $G972$, $G975$,$G1007-G1009$, $G1012-G1020$, $G1025-G1029$, $G1031-G1032$, $G1035-G1042$, $G1045-G1050$, $G1052-G1053$, $G1056-G1057$, $G1059$, $G1062$,$G1067$, $G1072$, $G1077-G1081$, $G1085$, $G1088$, $G1106$, $G1108-G1111$, $G1113-G1119$, $G1121-G1126$, $G1128-G1129$, $G1131-G1132$, $G1134-G1135$, $G1137-G1145$,
$G1147-G1152$, $G1157$,$G1163$, $G1169$, $G1173-G1176$, $G1178-G1200$, $G1202-G1203$, $G1205-G1208$, $G1210$, $G1213-G1252$,
\end{center}
\end{theorem}
\begin{proof} The proof is given in the following examples in this section.
\end{proof}
\begin{example}
 $G270-G272,$  are zero-divisor graphs. ($G270$ is a star graph. $G271$ and $G272$ are bi-star graphs).
 \end{example}
 \begin{example}$G314-G317$ are zero-divisor graphs (complete graphs with three vertices emanating ends).
 \end{example}
 \begin{example} Let $G319$ be a graph with $V(G)=\{1,2,3,4,5,6,7\}$, and $E(G)$ is defined by $N(1)=\{2,3\},$  $N(2)=N(3)=\{1,4\},$ $N(4)=\{2,3,5,6,7\},$ $N(5)=N(6)=N(7)=\{4\}$.  Then $G$ is a zero-divisor graph since $G$ is a complete bipartite graph together with ends emanating from the same vertex.

 \begin{center}G319
\begin{tabular}{|l|c|r|r|r|r|r|r|}
\hline
$\bullet$&$1$&$2$&$3$&$4$&$5$&$6$&$7$  \\ \hline
$1$&1&0&$0$&$4$&4&4&4  \\ \hline

$2$&0&2&$2$&$0$&2&2&2  \\ \hline
$3$&0&$2$&2&$0$&$2$&2&2  \\ \hline
$4$&$4$&$0$&$0$&0&$0$&0&$0$  \\ \hline
$5$&4&2&$2$&$0$&2&2&2  \\ \hline
$6$&4&2&$2$&$0$&2&2&2  \\ \hline
$7$&4&2&$2$&$0$&2&2&2  \\ \hline

\end{tabular}
\end{center}
\end{example}
\begin{example}$G379$ is a  zero-divisor graph because it is the refinement of a star graph with seven vertices..
 \end{example}
 \begin{example}Let $G380$ be a connected graph with six vertices $V(G)=\{1,2,3,4,5,6,7\}$   and the edges set  $E(G)$ is defined by the following way,
$N(1)=\{2\},$  $N(2)=\{1,3,6,7\},$ $N(3)=\{2,4,5,6,7\},$ $N(4)=N(5)=\{3\},$ $N(6)=N(7)=\{2,3\}$.
$G380$ is a zero-divisor graph.
  \begin{center}G380
\begin{tabular}{|l|c|r|r|r|r|r|r|}
\hline
$\bullet$&$1$&$2$&$3$&$4$&$5$&$6$&$7$  \\ \hline
$1$&1&0&$3$&$7$&7&7&7  \\ \hline

$2$&0&2&$0$&$2$&2&0&0  \\ \hline
$3$&3&$0$&3&$0$&$0$&0&0  \\ \hline
$4$&$7$&$2$&$0$&4&$4$&7&$7$  \\ \hline
$5$&$7$&$2$&$0$&4&$4$&7&$7$  \\ \hline

$6$&7&0&$0$&$7$&7&7&7  \\ \hline
$7$&7&0&$0$&$7$&7&7&7  \\ \hline

\end{tabular}
\end{center}
 \end{example}
 
 \begin{example} $G381$ is a zero-divisor graph.
 \begin{center} G381
\begin{tabular}{|l|c|r|r|r|r|r|r|}
\hline
$\bullet$&$a$&$b$&$c$&$x$&$y$&$z$&$w$  \\ \hline
$a$&$a$&$a$&$a$&0&0&$0$&$a$  \\ \hline

$b$&$a$&$a$&$a$&0&$0$&$x$&$a$  \\ \hline
$c$&$a$&$a$&$c$&$0$&$y$&$y$&$c$  \\ \hline
$x$&0&$0$&$0$&0&$0$&$x$&0  \\ \hline
$y$&0&$0$&$y$&$0$&$y$&$y$&$y$  \\ \hline
$z$&$0$&$x$&$y$&$x$&$y$&$z$&$y$  \\ \hline
$w$&0&0&$y$&0&$y$&$y$&$w$\\ \hline

\end{tabular}
\end{center}
 \end{example}
  \begin{example} $G382$ is a zero-divisor graph.
 \begin{center} G382
\begin{tabular}{|l|c|r|r|r|r|r|r|}
\hline
$\bullet$&$a$&$b$&$c$&$x$&$y$&$z$&$w$  \\ \hline
$a$&$0$&$z$&$z$&0&0&$0$&$z$  \\ \hline

$b$&$z$&$b$&$b$&0&$z$&$z$&$b$  \\ \hline
$c$&$z$&$b$&$b$&$0$&$z$&$z$&$b$  \\ \hline
$x$&0&$0$&$0$&$x$&$x$&$0$&0  \\ \hline
$y$&0&$z$&$z$&$x$&$x$&$0$&$z$  \\ \hline
$z$&$0$&$z$&$z$&$0$&$0$&$0$&$z$  \\ \hline
 $w$&$z$&$b$&$b$&$0$&$z$&$z$&$b$\\ \hline

\end{tabular}
\end{center}
 \end{example}
 \begin{example}$G384$ is a zero-divisor graph by applying \cite[Lemma 3.14]{DGSW} to $G113$ since $5^2\not=0$.
 \begin{center}G384
\begin{tabular}{|l|c|r|r|r|r|r|r|}
\hline
$\bullet$&$1$&$2$&$3$&$4$&$5$&$6$&$7$  \\ \hline
$1$&4&0&$0$&$4$&2&2&4  \\ \hline

$2$&0&0&$0$&$0$&2&2&0  \\ \hline
$3$&0&$0$&3&$0$&$3$&3&3  \\ \hline
$4$&$4$&$0$&$0$&4&$0$&0&$4$  \\ \hline
$5$&2&2&$3$&$0$&5&5&3  \\ \hline
$6$&2&2&3&0&5&5&$3$  \\ \hline
$7$&4&0&3&$4$&3&$3$&7  \\ \hline

\end{tabular}
\end{center}
\end{example}
  \begin{example}
 $G388$ is a refinement of a star graph and hence a zero divisor graph.
 \end{example}
 \begin{example} $G390$ is a zero-divisor graph.
 \begin{center} G390
\begin{tabular}{|l|c|r|r|r|r|r|r|}
\hline
$\bullet$&$a$&$b$&$c$&$x$&$y$&$z$&$w$  \\ \hline
$a$&$0$&$a$&$a$&0&0&$0$&$a$  \\ \hline

$b$&$a$&$b$&$b$&0&$0$&$$a$$&$b$  \\ \hline
$c$&$a$&$b$&$b$&$0$&$x$&$a$&$b$  \\ \hline
$x$&0&$0$&$0$&0&$x$&$0$&$0$  \\ \hline
$y$&0&$0$&$x$&$x$&$y$&$x$&$x$  \\ \hline
$z$&$0$&$a$&$a$&$0$&$x$&$0$&$a$  \\ \hline
$w$&$a$&$b$&$b$&$0$&$x$&$a$&$b$\\ \hline

\end{tabular}
\end{center}
 \end{example}
 \begin{example}The connected graph $G392$  with $V(G)=\{1,2,3,4,5,6,7\}$  and $E(G)$ defined by $N(1)=N(2)=\{3\},\, N(3)=\{1,2,4,5,6\},$ $N(4)=N(5)=N(6)=\{3,7\},$ $N(7)=\{4,5,6\}$. 
 $G392$ is a zero-divisor graph by \cite[Theorem IV.4]{sauer}.
 \begin{center}G392
\begin{tabular}{|l|c|r|r|r|r|r|r|}
\hline
$\bullet$&$1$&$2$&$3$&$4$&$5$&$6$&$7$  \\ \hline
$1$&4&4&$0$&$4$&4&4&3  \\ \hline
$2$&4&4&$0$&$4$&4&4&3  \\ \hline

$3$&0&$0$&0&$0$&$0$&0&3  \\ \hline
$4$&$4$&$4$&$0$&4&$4$&4&$0$  \\ \hline
$5$&$4$&$4$&$0$&4&$4$&4&$0$  \\ \hline
$6$&$4$&$4$&$0$&4&$4$&4&$0$  \\ \hline

$7$&3&3&3&$0$&0&$0$&7  \\ \hline

\end{tabular}
\end{center}
 \end{example}
 
\begin{example}
 $G393$ is a zero-divisor graph.  $G393$ can be defined by the following way: $N(1)=\{2\},$ $N(2)=\{1,3,5,7\},$ $N(3)=\{2,4\},$ $N(4)=\{3,5\},$ $N(4)=\{3,5\},$ $N(5)=\{2,4,6,7\},$ $N(6)=\{5\},$ $N(7)=\{2,5\}.$ 
 
 We define the multiplication table in the following way. One can check the associativity by Light's test.
 
  \begin{center}G393
\begin{tabular}{|l|c|r|r|r|r|r|r|}
\hline
$\bullet$&$1$&$2$&$3$&$4$&$5$&$6$&$7$  \\ \hline
$1$&3&0&$3$&$2$&5&7&5  \\ \hline

$2$&0&0&$0$&$2$&0&2&0  \\ \hline
$3$&3&$0$&3&$0$&$5$&5&5  \\ \hline
$4$&$2$&$2$&$0$&4&$0$&4&$2$  \\ \hline
$5$&5&0&$5$&$0$&0&0&0  \\ \hline
$6$&7&2&5&4&0&4&$2$  \\ \hline
$7$&5&0&5&$2$&0&$2$&0  \\ \hline

\end{tabular}
\end{center}
 
 \end{example}
 
\begin{example} Let $G411$ be a graph with $V(G)=\{a,b,c,x,y,z,w\}$, and $E(G)$ is defined by $N(a)=\{x,y,z,w\},$  $N(b)=N(c)=\{x,y\},$ $N(x)=N(y)=\{a,b,c\},$ $N(z)=N(w)=\{a\}$.  Then $G$ is a zero-divisor graph since $G$ is a complete bipartite graph together with ends emanating from the same vertex.
$G411$ is a zero-divisor graph by \cite[Theorem IV.4]{sauer} (bi-partite graph).
 \end{example}

 \begin{example}The connected graph $G473$ which is defined by  $N(a)=\{b,c, x,y, z,w\},$ $N(b)=\{a,c,x\},$ $N(c)=\{a,b, x\},$ $N(x)=\{a,b,c\},$ $N(y)=N(z)=N(w)=\{a\}$ is a zero-divisor graph because it is a refinement of a star graph.
\end{example}
\begin{example} Let $G474$ be a graph with $V(G)=\{a,b,c,x,y,z,w\}$, and $E(G)$ is defined by $N(a)=\{c,x,y,z,\},$  $N(b)=N(w)=\{x\},$ $N(c)=\{a,x,y\},$ $N(x)=\{a,b,c,w,y\},$ $N(y)=\{a,c,x\},$ $N(z)=\{a\}$  Then $G$ is a zero-divisor graph since $G$ is a complete  graph together with ends emanating from one of the  vertices.
\end{example}
\begin{example} $G476$ is a refinement of a star graph and hence a zero divisor graph.
 \end{example}
 \begin{example} $G477$ is a zero-divisor graph.
 \begin{center}G477
\begin{tabular}{|l|c|r|r|r|r|r|r|}
\hline
$\bullet$&$1$&$2$&$3$&$4$&$5$&$6$&$7$  \\ \hline
$1$&1&0&$3$&$4$&3&3&3  \\ \hline

$2$&0&2&$0$&$0$&2&0&0  \\ \hline
$3$&3&$0$&3&$0$&$3$&3&3  \\ \hline
$4$&$4$&$0$&$0$&4&$0$&0&$0$  \\ \hline
$5$&3&2&$3$&$0$&5&3&3  \\ \hline
$6$&3&0&$3$&$0$&3&3&3  \\ \hline
$7$&3&0&$3$&$0$&3&3&3  \\ \hline

\end{tabular}
\end{center}
\end{example}
\begin{example}
$G478$ is a refinement of a star graph and hence a zero divisor graph.
\end{example}
\begin{example}
$G479$ is a zero-divisor graph.
\begin{center} G479
\begin{tabular}{|l|c|r|r|r|r|r|r|}
\hline
$\bullet$&$a$&$b$&$c$&$x$&$y$&$z$&$w$  \\ \hline
$a$&$0$&$a$&$a$&0&0&$0$&$a$  \\ \hline

$b$&$a$&$b$&$b$&0&$0$&$$a$$&$b$  \\ \hline
$c$&$a$&$b$&$c$&$x$&$0$&$a$&$b$  \\ \hline
$x$&0&$0$&$x$&$x$&$0$&$0$&$0$  \\ \hline
$y$&0&$0$&$0$&$0$&$y$&$y$&$y$  \\ \hline
$z$&$0$&$a$&$a$&$0$&$y$&$y$&$z$  \\ \hline
$w$&$a$&$b$&$b$&$0$&$y$&$z$&$w$\\ \hline

\end{tabular}
\end{center}
\end{example}
\begin{example}Let $G480$ be a connected graph with six vertices $V(G)=\{a,b,c,x,y,z,w\}$   and the edges set  $E(G)$ is defined by the following way,
$N(a)=\{x,y,z\},$ $N(b)=N(c)=\{x,y\},$  $N(x)=\{a,b,c,w,y\},$ $N(y)=\{a,b,c,x\},$  $N(z)=\{a\}, N(w)=\{x\}$.
 $G480$ is a zero-divisor graph (combining $G113$ and \cite[Lemma 3.14]{DGSW} since $b^2\not=0$).
\begin{center} G480
\begin{tabular}{|l|c|r|r|r|r|r|r|}
\hline
$\bullet$&$a$&$b$&$c$&$x$&$y$&$z$&$w$  \\ \hline
$a$&$a$&$a$&$a$&0&0&$0$&$a$  \\ \hline

$b$&$a$&$a$&$a$&0&$0$&$x$&$a$  \\ \hline
$c$&$a$&$a$&$a$&$0$&$0$&$x$&$a$  \\ \hline
$x$&0&$0$&$0$&$0$&$0$&$x$&$0$  \\ \hline
$y$&0&$0$&$0$&$0$&$y$&$y$&$y$  \\ \hline
$z$&$0$&$x$&$x$&$x$&$y$&$z$&$y$  \\ \hline
$w$&$a$&$a$&$a$&$0$&$y$&$y$&$w$\\ \hline

\end{tabular}
\end{center}
\end{example}
\begin{example}Let $G481$ be a connected graph with six vertices $V(G)=\{a,b,c,x,y,z\}$   and the edges set  $E(G)$ is defined by the following way,
$N(a)=\{x,y,z\},$ $N(b)=\{x,y\},$ $N(c)=\{y\},$  $N(x)=\{a,b,y,z\},$ $N(y)=\{a,b,c,w,x\},$  $N(z)=\{a,x\}$.

  $G481$ is a zero-divisor graph $G481$ by checking the following multiplication table.
\begin{center}G481
\begin{tabular}{|l|c|r|r|r|r|r|r|}
\hline
$\bullet$&$a$&$b$&$c$&$x$&$y$&$z$&$w$  \\ \hline
$a$&$0$&$a$&$a$&0&0&$0$&$a$  \\ \hline
$b$&$a$&$b$&$b$&0&0&$a$&$b$  \\ \hline

$c$&$a$&$b$&$c$&$x$&$0$&$a$&$c$  \\ \hline
$x$&0&$0$&$x$&$x$&$0$&$0$&$x$  \\ \hline
$y$&0&$0$&$0$&$0$&$y$&$y$&$0$  \\ \hline
$z$&$0$&$a$&$a$&$0$&$y$&$y$&$a$  \\ \hline
$w$&$a$&$b$&$c$&$x$&$0$&$a$&$c$\\ \hline

\end{tabular}
\end{center}
\end{example}

 \begin{example}
 $G483$ is a zero-divisor graph (combining $G119$ and \cite[Lemma 3.14]{DGSW} because $y^2\not=0$).
\begin{center} G483
\begin{tabular}{|l|c|r|r|r|r|r|r|}
\hline
$\bullet$&$a$&$b$&$c$&$x$&$y$&$z$&$w$  \\ \hline
$a$&$0$&$z$&$z$&0&0&$0$&$0$  \\ \hline

$b$&$z$&$b$&$b$&0&$z$&$z$&$z$  \\ \hline
$c$&$z$&$b$&$b$&$0$&$z$&$z$&$z$  \\ \hline
$x$&0&$0$&$0$&$x$&$x$&$0$&$x$  \\ \hline
$y$&0&$z$&$z$&$x$&$x$&$0$&$x$  \\ \hline
$z$&$0$&$z$&$z$&$0$&$0$&$0$&$0$  \\ \hline
 $w$&0&$z$&$z$&$x$&$x$&$0$&$x$  \\ \hline

\end{tabular}
\end{center}
\end{example}
\begin{example}
 $G485$ is a zero-divisor graph.
 \begin{center} G485
\begin{tabular}{|l|c|r|r|r|r|r|r|}
\hline
$\bullet$&$a$&$b$&$c$&$x$&$y$&$z$&$w$  \\ \hline
$a$&$a$&$a$&$a$&0&0&$0$&$a$  \\ \hline

$b$&$a$&$b$&$b$&0&$0$&$z$&$b$  \\ \hline
$c$&$a$&$b$&$b$&$0$&$x$&$z$&$b$  \\ \hline
$x$&0&$0$&$0$&$0$&$x$&$0$&$0$  \\ \hline
$y$&0&$0$&$x$&$x$&$y$&$0$&$x$  \\ \hline
$z$&$0$&$z$&$z$&$0$&$0$&$z$&$z$  \\ \hline
 $w$&$a$&$b$&$b$&$0$&$x$&$z$&$b$\\ \hline

\end{tabular}
\end{center}
 \end{example}
 \begin{example}
 $G486$ is a zero-divisor graph.
 \begin{center} G486
\begin{tabular}{|l|c|r|r|r|r|r|r|}
\hline
$\bullet$&$a$&$b$&$c$&$x$&$y$&$z$&$w$  \\ \hline
$a$&$a$&$a$&$a$&0&0&$0$&$a$  \\ \hline

$b$&$a$&$b$&$a$&0&$0$&$$0$$&$b$  \\ \hline
$c$&$a$&$a$&$c$&$0$&$z$&$z$&$c$  \\ \hline
$x$&0&$0$&$0$&$x$&$x$&$0$&$0$  \\ \hline
$y$&0&$0$&$z$&$x$&$y$&$z$&$z$  \\ \hline
$z$&$0$&$0$&$z$&$0$&$z$&$z$&$z$  \\ \hline
$w$&$a$&$b$&$c$&$0$&$z$&$z$&$w$\\ \hline

\end{tabular}
\end{center}
 \end{example}
 \begin{example}
 $G493$ is a zero-divisor graph.
 \begin{center}G493
\begin{tabular}{|l|c|r|r|r|r|r|r|}
\hline
$\bullet$&$1$&$2$&$3$&$4$&$5$&$6$&$7$  \\ \hline
$1$&1&0&3&$5$&5&5&3  \\ \hline

$2$&0&2&$0$&2&0&0&2  \\ \hline
$3$&3&$0$&3&$0$&$0$&0&3  \\ \hline
$4$&$5$&2&$0$&4&5&5&2  \\ \hline
$5$&5&0&$0$&5&5&5&0  \\ \hline
$6$&5&0&$0$&5&5&5&0 \\ \hline
$7$&3&2&3&2&0&0& 7 \\ \hline

\end{tabular}
\end{center}
 
 \end{example}
\begin{example} $G503$ is a  refinement of star graph hence it is a zero divisor graph.
 \end{example}
 \begin{example}The connected graph $G507$  with $V(G)=\{a_1,a_2,b_1,b_2,b_3,b_4,x_1\}$  and $E(G)$ defined by $N(a_1)=\{b_1,b_2,b_3, b_4,x_1\},\, N(a_2)=\{b_1,b_2,b_3,b_4\},$ $N(b_1)=N(b_2)=N(b_3)=N(b_4)=\{a_1,a_2\},$ $N(x_1)=\{a_1\}$  is a zero-divisor graph by \cite[Theorem IV.4]{sauer}.
\end{example}
 \begin{example} $G513$ is a zero-divisor graph.
 \begin{center} G513
\begin{tabular}{|l|c|r|r|r|r|r|r|}
\hline
$\bullet$&$a$&$b$&$c$&$x$&$y$&$z$&$w$  \\ \hline
$a$&$0$&$a$&$a$&0&0&$0$&$a$  \\ \hline

$b$&$a$&$b$&$b$&0&$0$&$a$&$b$  \\ \hline
$c$&$a$&$b$&$c$&$0$&$0$&$a$&$b$  \\ \hline
$x$&0&$0$&$0$&0&$x$&$0$&$0$  \\ \hline
$y$&0&$0$&$0$&$x$&$y$&$x$&$x$  \\ \hline
$z$&$0$&$a$&$a$&$0$&$x$&$0$&$a$  \\ \hline
$w$&$a$&$b$&$b$&$0$&$x$&$a$&$b$\\ \hline

\end{tabular}
\end{center}
\end{example}
 \begin{example}
$G522$ is a zero-divisor graph.
\begin{center}G522
\begin{tabular}{|l|c|r|r|r|r|r|r|}
\hline
$\bullet$&$a$&$b$&$c$&$x$&$y$&$z$&$w$  \\ \hline
$a$&$0$&$a$&$a$&0&0&$0$&0  \\ \hline

$b$&$a$&$b$&$b$&0&$0$&$$a$$&$a$  \\ \hline
$c$&$a$&$b$&$c$&$0$&$0$&$a$&$a$  \\ \hline
$x$&0&$0$&$0$&0&$x$&$0$&$x$  \\ \hline
$y$&0&$0$&$0$&$x$&$y$&$x$&$y$  \\ \hline
$z$&$0$&$a$&$a$&$0$&$x$&$0$&$x$  \\ \hline
$w$&0&$a$&$a$&$x$&$y$&$x$&$y$\\ \hline

\end{tabular}
\end{center}
Another way to prove that $G522$ is a zero-divisor graph is by combining $G118$ and \cite[Lemma 3.14]{DGSW}.
\end{example}
 
 \begin{example}The connected graph $G525$  with $V(G)=\{a_1,a_2,b_1,b_2,b_3,b_4,x_1\}$  and $E(G)$ defined by $N(a_1)=N(a_2)=\{b_1,b_2,b_3, b_4\}$, $N(b_1)=\{a_1,a_2, x_1\}$, $N(b_2)=N(b_3)=N(b_4)=\{a_1,a_2\},$ $N(x_1)=\{b_1\}$  is a zero-divisor graph by \cite[Theorem IV.4]{sauer}.
\end{example}
\begin{example} $G551$ is a refinement of a star graph and hence a zero-divisor graph.
 \end{example}
 \begin{example} $G598$ is refinement of a star graph and hence a zero-divisor graph.
\end{example}
\begin{example}
$G599$ is a zero-divisor graph.
\begin{center}G599
\begin{tabular}{|l|c|r|r|r|r|r|r|}
\hline
$\bullet$&$1$&$2$&$3$&$4$&$5$&$6$&$7$  \\ \hline
$1$&1&0&3&$4$&3&6&6  \\ \hline

$2$&0&2&$0$&0&2&0&0  \\ \hline
$3$&3&$0$&3&$0$&$3$&6&6  \\ \hline
$4$&$4$&0&$0$&4&0&0&0  \\ \hline
$5$&3&2&$3$&0&5&6&6  \\ \hline
$6$&6&0&$6$&0&6&0&0 \\ \hline
$7$&6&0&6&0&6&0& 0 \\ \hline

\end{tabular}
\end{center}
 
\end{example}
\begin{example}
$G601$ is a zero-divisor graph.
\begin{center} G601
\begin{tabular}{|l|c|r|r|r|r|r|r|}
\hline
$\bullet$&$a$&$b$&$c$&$x$&$y$&$z$&$w$  \\ \hline
$a$&$a$&$a$&$a$&0&0&$0$&$a$  \\ \hline
$b$&$a$&$b$&$a$&$x$&$0$&$0$&$a$  \\ \hline

$c$&$a$&$a$&$c$&$0$&$y$&$z$&$c$  \\ \hline
$x$&0&$x$&$0$&$x$&$0$&$0$&$0$  \\ \hline
$y$&0&$0$&$y$&$0$&$y$&$0$&$y$  \\ \hline
$z$&$0$&$0$&$z$&$0$&$0$&$z$&$z$  \\ \hline
$w$&$a$&$a$&$c$&$0$&$y$&$z$&$c$\\ \hline

\end{tabular}
\end{center}
\end{example}
\begin{example}
$G603$ is a zero-divisor graph because it is a refinement of a star graph.
\end{example}
\begin{example}
Let $G604$ be a connected graph with seven vertices $V(G)=\{1,2,3,4,5,6,7\}$   and the edges set  $E(G)$ is defined by the following way,
$N(1)=\{2\},$ $N(2)=\{1,3,5,6\},$  $N(3)=N(5)=\{2,4,6\},$ $N(4)=\{3,5,6\}$, $N(6)=\{2,3,4,5,7\}$, and $N(7)=\{6\}.$
  $G604$ is a zero-divisor graph.

\begin{center}G604
\begin{tabular}{|l|c|r|r|r|r|r|r|}
\hline
$\bullet$&$1$&$2$&$3$&$4$&$5$&$6$&$7$  \\ \hline
$1$&1&$0$&$5$&6&5&6&$5$  \\ \hline

$2$&$0$&2&$0$&2&0&0&$2$  \\ \hline
$3$&5&$0$&5&$0$&$5$&0&5  \\ \hline
$4$&6&$2$&$0$&2&0&0&2  \\ \hline
$5$&$5$&$0$&$5$&$0$&5&$0$&5  \\ \hline
$6$&6&0&0&0&$0$&0&$0$ \\ \hline
$7$&$5$&$2$&$5$&$2$&$5$&$0$&7\\ \hline
\end{tabular}
\end{center}
\end{example}
\begin{example}
$G606$ is a zero-divisor graph (combining  $G159$ and \cite[Lemma 3.14]{DGSW}).
\begin{center} G606
\begin{tabular}{|l|c|r|r|r|r|r|r|}
\hline
$\bullet$&$a$&$b$&$c$&$x$&$y$&$z$&$w$  \\ \hline
$a$&$a$&$a$&$a$&0&0&$0$&$a$  \\ \hline
$b$&$a$&$b$&$a$&0&0&$0$&$a$  \\ \hline

$c$&$a$&$a$&$c$&$0$&$z$&$z$&$c$  \\ \hline
$x$&0&$0$&$0$&$x$&$x$&$0$&$0$  \\ \hline
$y$&0&$0$&$z$&$x$&$x$&$0$&$z$  \\ \hline
$z$&$0$&$0$&$z$&$0$&$0$&$0$&$z$  \\ \hline
$w$&$a$&$a$&$c$&$0$&$z$&$z$&$c$\\ \hline

\end{tabular}
\end{center}
\end{example}
\begin{example} $G612-G613$ are zero-divisor graphs because they are the refinement of star graphs.
\end{example}
\begin{example} $G614$ is a zero-divisor graph (combining  $G137$  and \cite[Lemma 3.14]{DGSW} because $b^2\not=0$.)
\begin{center} G614
\begin{tabular}{|l|c|r|r|r|r|r|r|}
\hline
$\bullet$&$a$&$b$&$c$&$x$&$y$&$z$&$w$  \\ \hline
$a$&$0$&$a$&$a$&0&0&$0$&$a$  \\ \hline

$b$&$a$&$b$&$b$&0&$0$&$$a$$&$b$  \\ \hline
$c$&$a$&$b$&$c$&$x$&$0$&$a$&$b$  \\ \hline
$x$&0&$0$&$x$&$x$&$0$&$0$&$0$  \\ \hline
$y$&0&$0$&$0$&$0$&$y$&$y$&$0$  \\ \hline
$z$&$0$&$a$&$a$&$0$&$y$&$y$&$a$  \\ \hline
$w$&$a$&$b$&$b$&$0$&$0$&$a$&$b$\\ \hline

\end{tabular}
\end{center}
\end{example}
\begin{example} $G616$ is a zero-divisor graph (combining $G137$ and \cite[Lemma 3.14]{DGSW} since $z\times z\ne 0.$.)
\begin{center} G616
\begin{tabular}{|l|c|r|r|r|r|r|r|}
\hline
$\bullet$&$a$&$b$&$c$&$x$&$y$&$z$&$w$  \\ \hline
$a$&$0$&$a$&$a$&0&0&$0$&$0$  \\ \hline

$b$&$a$&$b$&$b$&0&$0$&$$a$$&$a$  \\ \hline
$c$&$a$&$b$&$c$&$x$&$0$&$a$&$a$  \\ \hline
$x$&0&$0$&$x$&$x$&$0$&$0$&$0$  \\ \hline
$y$&0&$0$&$0$&$0$&$y$&$y$&$y$  \\ \hline
$z$&$0$&$a$&$a$&$0$&$y$&$y$&$y$  \\ \hline
$w$&$0$&$a$&$a$&$0$&$y$&$y$&$y$\\ \hline

\end{tabular}
\end{center}
\end{example}
\begin{example} $G618-G619$ are zero-divisor graphs because they are the refinement of star graphs.
\end{example}
\begin{example}$G620$ is a zero-divisor graph (combining  $G166$  and \cite[Lemma 3.14]{DGSW}.)
\begin{center} G620
\begin{tabular}{|l|c|r|r|r|r|r|r|}
\hline
$\bullet$&$a$&$b$&$c$&$x$&$y$&$z$&$w$  \\ \hline
$a$&$a$&$a$&$a$&0&0&$0$&$a$  \\ \hline
$b$&$a$&$a$&$a$&0&$x$&$0$&$a$  \\ \hline

$c$&$a$&$a$&$c$&$0$&$0$&$z$&$c$  \\ \hline
$x$&0&$0$&$0$&$0$&$x$&$0$&$0$  \\ \hline
$y$&0&$x$&$0$&$x$&$y$&$0$&$x$  \\ \hline
$z$&$0$&$0$&$z$&$0$&$0$&$z$&$z$  \\ \hline
$w$&$a$&$a$&$c$&$0$&$x$&$z$&$c$\\ \hline

\end{tabular}
\end{center}
\end{example}
\begin{example}
Let $G624$ be a connected graph with six vertices $V(G)=\{a,b,c,x,y,z,w\}$   and the edges set  $E(G)$ is defined by the following way,
$N(a)=\{x,y,z\},$ $N(b)=\{x,z\},$  $N(c)=\{x,y\},$  $N(x)=\{a,b,c,z\},$ $N(y)=\{a,c,z\}$, $N(z)=\{a,b,x,y,w\}$, and $N(w)=\{z\}.$
  $G624$ is a zero-divisor graph.
\begin{center}G624
\begin{tabular}{|l|c|r|r|r|r|r|r|}
\hline
$\bullet$&$a$&$b$&$c$&$x$&$y$&$z$&$w$  \\ \hline
$a$&$a$&$a$&$a$&0&0&$0$&$a$  \\ \hline
$b$&$a$&$a$&$a$&0&$x$&$0$&$b$  \\ \hline

$c$&$a$&$a$&$c$&$0$&$0$&$z$&$a$  \\ \hline
$x$&0&$0$&$0$&$0$&$x$&$0$&$x$  \\ \hline
$y$&0&$x$&$0$&$x$&$y$&$0$&$y$  \\ \hline
$z$&$0$&$0$&$z$&$0$&$0$&$z$&$0$  \\ \hline
$w$&$a$&$b$&$a$&$x$&$y$&$0$&$w$\\ \hline

\end{tabular}
\end{center}

\end{example}
\begin{example}Let $G626$ be a connected graph with six vertices $V(G)=\{a,b,c,x,y,z\}$   and the edges set  $E(G)$ is defined by the following way,
$N(a)=\{c,x,y\},$ $N(b)=\{y,z\},$ $N(c)=\{a,x,y\},$ $N(x)=\{a,c,y,z,w\},$ $N(y)=\{a,b,c,x\},$   $N(z)=\{b,x\}$, and $N(w)=\{x\}.$

$G626$ is a zero-divisor graph.
\begin{center} G626
\begin{tabular}{|l|c|r|r|r|r|r|r|}
\hline
$\bullet$&$a$&$b$&$c$&$x$&$y$&$z$&$w$  \\ \hline
$a$&$0$&$x$&$0$&0&0&$y$&$y$  \\ \hline

$b$&$x$&$b$&$x$&$x$&$0$&$0$&$x$  \\ \hline
$c$&$0$&$x$&$0$&$0$&$0$&$y$&$y$  \\ \hline
$x$&0&$x$&$0$&$0$&$0$&$0$&0  \\ \hline
$y$&0&$0$&$0$&$0$&$0$&$y$&$y$  \\ \hline
$z$&$y$&$0$&$y$&$0$&$y$&$z$&$z$  \\ \hline
$w$&$y$&$x$&$y$&$0$&$y$&$z$&$z$\\ \hline

\end{tabular}
\end{center}
\end{example}
\begin{example} $G629$ is a zero-divisor graph.
\begin{center} G629
\begin{tabular}{|l|c|r|r|r|r|r|r|}
\hline
$\bullet$&$a$&$b$&$c$&$x$&$y$&$z$&$w$  \\ \hline
$a$&$b$&$b$&$b$&0&0&$0$&$b$  \\ \hline
$b$&$b$&$b$&$b$&0&0&$0$&$b$  \\ \hline

$c$&$b$&$b$&$b$&$0$&$0$&$x$&$b$  \\ \hline
$x$&0&$0$&$0$&$0$&$0$&$x$&$0$  \\ \hline
$y$&0&$0$&$0$&$0$&$y$&$y$&$y$  \\ \hline
$z$&$0$&$0$&$x$&$x$&$y$&$z$&$y$  \\ \hline
$w$&$b$&$b$&$b$&$0$&$y$&$y$&$w$\\ \hline

\end{tabular}
\end{center}
\end{example}
\begin{example}$G631$ is a zero-divisor graph.
\begin{center} G631
\begin{tabular}{|l|c|r|r|r|r|r|r|}
\hline
$\bullet$&$a$&$b$&$c$&$x$&$y$&$z$&$w$  \\ \hline
$a$&$0$&$a$&$a$&0&0&$0$&$a$  \\ \hline

$b$&$a$&$b$&$b$&0&$0$&$$a$$&$b$  \\ \hline
$c$&$a$&$b$&$b$&$0$&$0$&$a$&$b$  \\ \hline
$x$&0&$0$&$0$&$0$&$x$&$0$&$0$  \\ \hline
$y$&0&$0$&$0$&$x$&$y$&$0$&$x$  \\ \hline
$z$&$0$&$a$&$a$&$0$&$0$&$0$&$a$  \\ \hline
$w$&$a$&$b$&$b$&$0$&$x$&$a$&$b$\\ \hline

\end{tabular}
\end{center}
\end{example}
\begin{example}
$G633$ is a zero-divisor graph.
\begin{center}G633
\begin{tabular}{|l|c|r|r|r|r|r|r|}
\hline
$\bullet$&$1$&$2$&$3$&$4$&$5$&$6$&$7$  \\ \hline
$1$&1&$0$&3&7&5&3&7  \\ \hline

$2$&$0$&2&$0$&2&0&2&$0$  \\ \hline
$3$&3&$0$&0&$0$&$3$&0&0  \\ \hline
$4$&7&$2$&$0$&4&0&2&7 \\ \hline
$5$&$5$&$0$&$3$&$0$&5&$3$&0  \\ \hline
$6$&3&2&0&2&$3$&2&0 \\ \hline
$7$&7&$0$&0&7&$0$&$0$&7\\ \hline
\end{tabular}
\end{center}
\end{example}
\begin{example}$G639$ is a zero-divisor graph by \cite[IV.1]{sauer} and \cite[Lemma 3.14]{DGSW}.
\begin{center}G639
\begin{tabular}{|l|c|r|r|r|r|r|r|}
\hline
$\bullet$&$1$&$2$&$3$&$4$&$5$&$6$&$7$  \\ \hline
$1$&1&$0$&0&1&1&1&0  \\ \hline

$2$&$0$&2&$0$&0&2&0&$2$  \\ \hline
$3$&0&$0$&3&$3$&$0$&3&0  \\ \hline
$4$&1&$0$&$3$&4&1&4&0 \\ \hline
$5$&$1$&$2$&$0$&$1$&5&$1$&2  \\ \hline
$6$&1&0&3&4&$1$&4&0 \\ \hline
$7$&0&$2$&0&0&$2$&$0$&2\\ \hline
\end{tabular}
\end{center}

\end{example}
\begin{example}The connected graph $G667$  with $V(G)=\{a_1,a_2,b_1,b_2,b_3,b_4,x_1\}$  and $E(G)$ defined by $N(a_1)=N(a_2)=N(a_3)=\{b_1,b_2,b_3\}$, $N(b_1)=\{a_1,a_2, a_3,x_1\}$, $N(b_2)=N(b_3)=\{a_1,a_2,a_3\},$ $N(x_1)=\{b_1\}$  is a zero-divisor graph by \cite[IV.4]{sauer}. 
\end{example}

\begin{example} $G668$ is a zero-divisor graph because it is a refinement of star graph.
\end{example}
\begin{example} The connected graph $G670$ which is defined by  $N(a)=N(b)=\{c,d,e,f,g\},$ $N(c)=N(d)=N(e)=N(f)=N(g)=\{a,b\}$ is a zero-divisor graph by \cite[IV.4]{sauer}.
\end{example}

\begin{example} $G671-G672$ are zero-divisor graphs because they are the refinement of star graphs.
\end{example}
\begin{example} $G678$ is a zero-divisor graph by applying \cite[Lemma 3.14]{DGSW} to $G119$ because $b^2\not=0$.
\begin{center} G678
\begin{tabular}{|l|c|r|r|r|r|r|r|}
\hline
$\bullet$&$a$&$b$&$c$&$x$&$y$&$z$&$w$  \\ \hline
$a$&$0$&$a$&$a$&0&0&$0$&$a$  \\ \hline

$b$&$a$&$b$&$b$&0&$0$&$a$&$b$  \\ \hline
$c$&$a$&$b$&$c$&$0$&$0$&$a$&$b$  \\ \hline
$x$&0&$0$&$0$&0&$x$&$0$&$0$  \\ \hline
$y$&0&$0$&$0$&$x$&$y$&$x$&$0$  \\ \hline
$z$&$0$&$a$&$a$&$0$&$x$&$0$&$a$  \\ \hline
$w$&$a$&$b$&$b$&$0$&$0$&$a$&$b$\\ \hline

\end{tabular}
\end{center}
\end{example}
\begin{example} $G740$ is a zero-divisor graph because it is a refinement of a star graph.
\end{example}
\begin{example}
$G741$ is a zero-divisor graph. The proof of \cite[Theorem IV.1]{sauer} is wrong.
\begin{center}G741
\begin{tabular}{|l|c|r|r|r|r|r|r|}
\hline
$\bullet$&$1$&$2$&$3$&$4$&$5$&$6$&$7$  \\ \hline
$1$&1&$0$&1&0&0&1&1  \\ \hline

$2$&$0$&0&$0$&0&0&5&$4$  \\ \hline
$3$&1&$0$&1&0&0&1&1  \\ \hline
$4$&0&$0$&$0$&0&0&0&4 \\ \hline
$5$&$0$&$0$&$0$&$0$&0&$5$&0  \\ \hline
$6$&1&5&1&0&$5$&6&1 \\ \hline
$7$&1&$4$&1&4&$0$&$1$&7\\ \hline
\end{tabular}
\end{center}
\end{example}
\begin{example} $G743$ is a zero-divisor graph.
\begin{center} G743
\begin{tabular}{|l|c|r|r|r|r|r|r|}
\hline
$\bullet$&$a$&$b$&$c$&$x$&$y$&$z$&$w$  \\ \hline
$a$&$0$&$a$&$0$&0&0&$0$&$a$  \\ \hline

$b$&$a$&$b$&$a$&$0$&$a$&$a$&$b$  \\ \hline
$c$&$0$&$a$&$0$&$0$&$0$&$0$&$a$  \\ \hline
$x$&0&$0$&$0$&$x$&$x$&$0$&0  \\ \hline
$y$&0&$a$&$0$&$x$&$x$&$0$&$a$  \\ \hline
$z$&$0$&$a$&$0$&$0$&$0$&$0$&$a$  \\ \hline
$w$&$a$&$b$&$a$&$0$&$a$&$a$&$b$\\ \hline

\end{tabular}
\end{center}
\end{example}
\begin{example} $G746-G747$ are zero-divisor graphs because they are the refinement of star graphs.
\end{example}
\begin{example} $G748$ is a zero-divisor graph.
\begin{center}G748
\begin{tabular}{|l|c|r|r|r|r|r|r|}
\hline
$\bullet$&$1$&$2$&$3$&$4$&$5$&$6$&$7$  \\ \hline
$1$&1&$0$&5&4&5&0&5  \\ \hline

$2$&$0$&0&$0$&0&0&0&$2$  \\ \hline
$3$&5&$0$&5&0&5&2&3  \\ \hline
$4$&4&$0$&$0$&0&0&0&0 \\ \hline
$5$&$5$&$0$&$5$&$0$&5&$0$&5  \\ \hline
$6$&0&0&2&0&$0$&2&6 \\ \hline
$7$&5&$2$&3&0&$5$&$6$&7\\ \hline
\end{tabular}
\end{center}
\end{example}

\begin{example}$G749$ is a zero-divisor graph ( applying  \cite[Lemma 3.14]{DGSW} to $G157$ since $b^2\not=0$).
\begin{center} G749
\begin{tabular}{|l|c|r|r|r|r|r|r|}
\hline
$\bullet$&$a$&$b$&$c$&$x$&$y$&$z$&$w$  \\ \hline
$a$&$a$&$a$&$a$&0&0&$0$&$a$  \\ \hline
$b$&$a$&$b$&$a$&$x$&$0$&$0$&$b$  \\ \hline

$c$&$a$&$a$&$c$&$0$&$y$&$z$&$a$  \\ \hline
$x$&0&$x$&$0$&$x$&$0$&$0$&$x$  \\ \hline
$y$&0&$0$&$y$&$0$&$y$&$0$&$0$  \\ \hline
$z$&$0$&$0$&$z$&$0$&$0$&$z$&$0$  \\ \hline
$w$&$a$&$b$&$a$&$x$&$0$&$0$&$a$\\ \hline

\end{tabular}
\end{center}
\end{example}
\begin{example} $G751$ is a zero-divisor graph.
\begin{center}G751
\begin{tabular}{|l|c|r|r|r|r|r|r|}
\hline
$\bullet$&$1$&$2$&$3$&$4$&$5$&$6$&$7$  \\ \hline
$1$&1&$0$&0&0&5&0&1  \\ \hline

$2$&$0$&2&$0$&4&6&6&$4$  \\ \hline
$3$&5&$0$&5&0&5&2&3  \\ \hline
$4$&4&$0$&$0$&0&0&0&0 \\ \hline
$5$&$5$&$0$&$5$&$0$&5&$0$&5  \\ \hline
$6$&0&0&2&0&$0$&2&6 \\ \hline
$7$&5&$2$&3&0&$5$&$6$&7\\ \hline
\end{tabular}
\end{center}

\end{example}
\begin{example}
$G752$ and $G753$ are zero-divisor graphs because they are the refinement of star graphs.
\end{example}
\begin{example}
$G755$ is a zero-divisor graph.
\begin{center}G755
\begin{tabular}{|l|c|r|r|r|r|r|r|}
\hline
$\bullet$&$1$&$2$&$3$&$4$&$5$&$6$&$7$  \\ \hline
$1$&1&$0$&3&4&6&6&7  \\ \hline

$2$&$0$&2&$0$&0&2&0&$0$  \\ \hline
$3$&3&$0$&4&0&4&4&0  \\ \hline
$4$&4&$0$&$0$&0&0&0&0 \\ \hline
$5$&$6$&$2$&$4$&$0$&2&$0$&4  \\ \hline
$6$&6&0&4&0&$0$&0&4 \\ \hline
$7$&7&$0$&0&0&$4$&$4$&4\\ \hline
\end{tabular}
\end{center}
\end{example}
\begin{example}
$G757$ is a zero-divisor graph.
\begin{center}G757
\begin{tabular}{|l|c|r|r|r|r|r|r|}
\hline
$\bullet$&$1$&$2$&$3$&$4$&$5$&$6$&$7$  \\ \hline
$1$&1&$0$&3&5&5&6&7  \\ \hline

$2$&$0$&2&$0$&2&0&0&$0$  \\ \hline
$3$&3&$0$&6&0&0&6&5  \\ \hline
$4$&5&$2$&$0$&2&0&0&5 \\ \hline
$5$&$5$&$0$&$0$&$0$&0&$0$&5  \\ \hline
$6$&6&0&6&0&$0$&6&0 \\ \hline
$7$&7&$0$&5&5&$5$&$0$&7\\ \hline
\end{tabular}
\end{center}
\end{example}

\begin{example}
$G758$ is a zero-divisor graph.
\begin{center}G758
\begin{tabular}{|l|c|r|r|r|r|r|r|}
\hline
$\bullet$&$1$&$2$&$3$&$4$&$5$&$6$&$7$  \\ \hline
$1$&1&$0$&3&7&5&5&7  \\ \hline

$2$&$0$&2&$0$&2&0&2&$0$  \\ \hline
$3$&3&$0$&3&0&5&5&0  \\ \hline
$4$&7&$2$&$0$&2&0&2&0 \\ \hline
$5$&$5$&$0$&$5$&$0$&0&$0$&0  \\ \hline
$6$&5&2&5&2&$0$&2&0 \\ \hline
$7$&7&$0$&0&0&$0$&$0$&0\\ \hline
\end{tabular}
\end{center}
\end{example}
\begin{example} $G759$ is a zero-divisor graph (proved in Example 4.32).
\begin{center}G759
\begin{tabular}{|l|c|r|r|r|r|r|r|}
\hline
$\bullet$&$a$&$b$&$c$&$x$&$y$&$z$&$w$  \\ \hline
$a$&0&$a$&0&0&0&0&0  \\ \hline
$b$&$a$&$b$&$c$&$0$&0&$a$&$c$  \\ \hline

$c$&0&$c$&$c$&0&$0$&0&$c$  \\ \hline
$x$&0&$0$&$0$&0&$x$&$0$&$x$  \\ \hline
$y$&$0$&$0$&$0$&$x$&$y$&$x$&$y$  \\ \hline
$z$&0&$a$&0&0&$x$&$0$&$x$ \\ \hline
$w$&0&$c$&$c$&$x$&$y$&$x$&$w$\\ \hline
\end{tabular}
\end{center}
\end{example}
\begin{example} $G762$ is a zero-divisor graph by applying \cite[Lemma 3.14]{DGSW} to $G141$  since $b^2\not=0$.
\begin{center}G762
\begin{tabular}{|l|c|r|r|r|r|r|r|}
\hline
$\bullet$&$a$&$b$&$c$&$x$&$y$&$z$&$w$  \\ \hline
$a$&$b$&$b$&$b$&0&0&$0$&$b$  \\ \hline
$b$&$b$&$b$&$b$&0&0&$0$&$b$  \\ \hline

$c$&$b$&$b$&$c$&$x$&$0$&$0$&$c$  \\ \hline
$x$&0&$0$&$x$&$x$&$0$&$0$&$x$  \\ \hline
$y$&0&$0$&$0$&$0$&$0$&$y$&$0$  \\ \hline
$z$&$0$&$0$&$0$&$0$&$y$&$z$&$y$  \\ \hline
$w$&$b$&$b$&$c$&$x$&$0$&$y$&$c$\\ \hline

\end{tabular}
\end{center}
\end{example}
\begin{example} $G764$ is  a zero-divisor graph by applying \cite[Lemma 3.14]{DGSW} to $G118$ since $a^2=0$.
\begin{center} G764
\begin{tabular}{|l|c|r|r|r|r|r|r|}
\hline
$\bullet$&$a$&$b$&$c$&$x$&$y$&$z$&$w$  \\ \hline
$a$&$0$&$a$&$0$&0&0&$0$&$a$  \\ \hline

$b$&$a$&$b$&$a$&$0$&$0$&$a$&$b$  \\ \hline
$c$&$0$&$a$&$0$&$0$&$0$&$0$&$a$  \\ \hline
$x$&0&$0$&$0$&$0$&$x$&$0$&0  \\ \hline
$y$&0&$0$&$0$&$x$&$y$&$x$&$x$  \\ \hline
$z$&$0$&$a$&$0$&$0$&$x$&$0$&$a$  \\ \hline
$w$&$a$&$b$&$a$&$0$&$x$&$a$&$b$\\ \hline

\end{tabular}
\end{center}

\end{example}
\begin{example}$G767$ is a zero-divisor graph.
\begin{center} G767
\begin{tabular}{|l|c|r|r|r|r|r|r|}
\hline
$\bullet$&$a$&$b$&$c$&$x$&$y$&$z$&$w$  \\ \hline
$a$&$a$&$a$&$a$&0&0&$0$&$a$  \\ \hline
$b$&$a$&$a$&$a$&0&0&$0$&$a$  \\ \hline
$c$&$a$&$a$&$c$&$x$&0&$0$&$a$  \\ \hline
$x$&0&$0$&$x$&$x$&$0$&$0$&$0$  \\ \hline
$y$&0&$0$&$0$&$0$&$y$&$y$&$y$  \\ \hline
$z$&$0$&$0$&$0$&$0$&$y$&$y$&$y$  \\ \hline
$w$&$a$&$a$&$a$&$0$&$y$&$y$&$w$\\ \hline

\end{tabular}
\end{center}
\end{example}
\begin{example} $G775$ is a zero-divisor graph because it is a refinement of a star graph.
\end{example}
\begin{example} $G780$ is a zero-divisor graph (combining $G140$ and \cite[Lemma 3.14]{DGSW} since $a^2\not=0$).
\begin{center} G780
\begin{tabular}{|l|c|r|r|r|r|r|r|}
\hline
$\bullet$&$a$&$b$&$c$&$x$&$y$&$z$&$w$  \\ \hline
$a$&$a$&$a$&$a$&0&0&$0$&$a$  \\ \hline
$b$&$a$&$a$&$a$&0&0&$0$&$a$  \\ \hline
$c$&$a$&$a$&$a$&0&0&$0$&$a$  \\ \hline
$x$&0&$0$&$0$&$0$&$x$&$0$&$0$  \\ \hline
$y$&0&$0$&$0$&$x$&$y$&$z$&$z$  \\ \hline
$z$&$0$&$0$&$0$&$0$&$z$&$z$&$z$  \\ \hline
$w$&$a$&$a$&$a$&$0$&$z$&$z$&$w$\\ \hline

\end{tabular}
\end{center}
\end{example}
\begin{example} $G786$ is a zero-divisor graph (combining $G141$ and \cite[Lemma 3.14]{DGSW} since $y^2\not=0$).
\begin{center} G786
\begin{tabular}{|l|c|r|r|r|r|r|r|}
\hline
$\bullet$&$a$&$b$&$c$&$x$&$y$&$z$&$w$  \\ \hline
$a$&$a$&$a$&$a$&0&0&$0$&$0$  \\ \hline

$b$&$a$&$b$&$b$&0&$0$&$z$&$0$  \\ \hline
$c$&$a$&$b$&$b$&$0$&$x$&$z$&$x$  \\ \hline
$x$&0&$0$&$0$&$0$&$x$&$0$&$x$  \\ \hline
$y$&0&$0$&$x$&$x$&$y$&$0$&$y$  \\ \hline
$z$&$0$&$z$&$z$&$0$&$0$&$z$&$0$  \\ \hline
 $w$&$0$&$0$&$x$&$x$&$y$&$0$&$y$\\ \hline

\end{tabular}
\end{center}
\end{example}
\begin{example} $G790-G792$ are  zero-divisor graphs because they are the refinement of star graphs.
\end{example}
\begin{example} $G794-G796$ are zero-divisor graphs because they are the refinement of star graphs.
\end{example}
\begin{example} $G798$ is a zero-divisor graph (combining $G166$  and \cite[Lemma 3.14]{DGSW} because $b^2\not=0$).
\begin{center} G798
\begin{tabular}{|l|c|r|r|r|r|r|r|}
\hline
$\bullet$&$a$&$b$&$c$&$x$&$y$&$z$&$w$  \\ \hline
$a$&$a$&$a$&$a$&0&0&$0$&$a$  \\ \hline
$b$&$a$&$a$&$a$&0&$x$&$0$&$a$  \\ \hline

$c$&$a$&$a$&$c$&$0$&$0$&$z$&$a$  \\ \hline
$x$&0&$0$&$0$&$0$&$x$&$0$&$0$  \\ \hline
$y$&0&$x$&$0$&$x$&$y$&$0$&$x$  \\ \hline
$z$&$0$&$0$&$z$&$0$&$0$&$z$&$0$  \\ \hline
$w$&$a$&$a$&$a$&$0$&$x$&$0$&$a$\\ \hline

\end{tabular}
\end{center}
\end{example}
\begin{example} $G800$ is a zero-divisor graph (combining $G168$ and \cite[Lemma 3.14]{DGSW} since $c^2\not=0$).
\begin{center} G800
\begin{tabular}{|l|c|r|r|r|r|r|r|}
\hline
$\bullet$&$a$&$b$&$c$&$x$&$y$&$z$&$w$  \\ \hline
$a$&$b$&$b$&$b$&0&0&$0$&$b$  \\ \hline
$b$&$b$&$b$&$b$&0&0&$0$&$b$  \\ \hline

$c$&$b$&$b$&$b$&$0$&$0$&$x$&$b$  \\ \hline
$x$&0&$0$&$0$&$0$&$0$&$x$&$0$  \\ \hline
$y$&0&$0$&$0$&$0$&$y$&$y$&$0$  \\ \hline
$z$&$0$&$0$&$x$&$x$&$y$&$z$&$x$  \\ \hline
$w$&$b$&$b$&$b$&$0$&$0$&$x$&$b$\\ \hline

\end{tabular}
\end{center}
\end{example}
\begin{example} $G801$ is a zero-divisor graph (combining $G170$  and \cite[Lemma 3.14]{DGSW} since $b^2\not=0$).
\begin{center} G801
\begin{tabular}{|l|c|r|r|r|r|r|r|}
\hline
$\bullet$&$a$&$b$&$c$&$x$&$y$&$z$&$w$  \\ \hline
$a$&$0$&$a$&$a$&0&0&$0$&$a$  \\ \hline

$b$&$a$&$b$&$b$&0&$0$&$$a$$&$b$  \\ \hline
$c$&$a$&$b$&$b$&$0$&$0$&$a$&$b$  \\ \hline
$x$&0&$0$&$0$&$0$&$x$&$0$&$0$  \\ \hline
$y$&0&$0$&$0$&$x$&$y$&$0$&$0$  \\ \hline
$z$&$0$&$a$&$a$&$0$&$0$&$0$&$a$  \\ \hline
$w$&$a$&$b$&$b$&$0$&$0$&$a$&$b$\\ \hline

\end{tabular}
\end{center}
\end{example}
\begin{example} $G805$ is a zero-divisor graph (combining $G166$  and \cite[Lemma 3.14]{DGSW} because $c^2\not=0$).
\begin{center} G805
\begin{tabular}{|l|c|r|r|r|r|r|r|}
\hline
$\bullet$&$a$&$b$&$c$&$x$&$y$&$z$&$w$  \\ \hline
$a$&$a$&$a$&$a$&0&0&$0$&$a$  \\ \hline
$b$&$a$&$a$&$a$&0&$x$&$0$&$a$  \\ \hline

$c$&$a$&$a$&$c$&$0$&$0$&$z$&$c$  \\ \hline
$x$&0&$0$&$0$&$0$&$x$&$0$&$0$  \\ \hline
$y$&0&$x$&$0$&$x$&$y$&$0$&$0$  \\ \hline
$z$&$0$&$0$&$z$&$0$&$0$&$z$&$z$  \\ \hline
$w$&$a$&$a$&$c$&$0$&$0$&$z$&$c$\\ \hline

\end{tabular}
\end{center}
\end{example}
\begin{example} $G812$ is a zero-divisor graph (applying \cite[Lemma 3.14]{DGSW} to $G168$ because $z^2\not=0$).
\begin{center} G812
\begin{tabular}{|l|c|r|r|r|r|r|r|}
\hline
$\bullet$&$a$&$b$&$c$&$x$&$y$&$z$&$w$  \\ \hline
$a$&$b$&$b$&$b$&0&0&$0$&$0$  \\ \hline
$b$&$b$&$b$&$b$&0&0&$0$&$0$  \\ \hline

$c$&$b$&$b$&$b$&$0$&$0$&$x$&$x$  \\ \hline
$x$&0&$0$&$0$&$0$&$0$&$x$&$x$  \\ \hline
$y$&0&$0$&$0$&$0$&$y$&$y$&$y$  \\ \hline
$z$&$0$&$0$&$x$&$x$&$y$&$z$&$z$  \\ \hline
$w$&$0$&$0$&$x$&$x$&$y$&$z$&$z$\\ \hline

\end{tabular}
\end{center}
\end{example}
\begin{example} $G813-G815$ are zero-divisor graphs because they are the refinement of star graphs.
\end{example}
\begin{example} $G832$ is a zero-divisor graph (combining $G169$  and \cite[Lemma 3.14]{DGSW} because $b^2\not=0$).
\begin{center} G832
\begin{tabular}{|l|c|r|r|r|r|r|r|}
\hline
$\bullet$&$a$&$b$&$c$&$x$&$y$&$z$&$w$  \\ \hline
$a$&$0$&$x$&$0$&0&0&$y$&$x$  \\ \hline

$b$&$x$&$b$&$x$&$x$&$0$&$0$&$b$  \\ \hline
$c$&$0$&$x$&$0$&$0$&$0$&$y$&$x$  \\ \hline
$x$&0&$x$&$0$&$0$&$0$&$0$&$x$  \\ \hline
$y$&0&$0$&$0$&$0$&$0$&$y$&$0$  \\ \hline
$z$&$y$&$0$&$y$&$0$&$y$&$z$&$0$  \\ \hline
$w$&$x$&$b$&$x$&$x$&$0$&$0$&$b$\\ \hline

\end{tabular}
\end{center}
\end{example}
\begin{example}
$G872$ is a zero divisor graph (combining $G119$  and \cite[Lemma 3.14]{DGSW} because $y^2\not=0$).
\begin{center} G872
\begin{tabular}{|l|c|r|r|r|r|r|r|}
\hline
$\bullet$&$a$&$b$&$c$&$x$&$y$&$z$&$w$  \\ \hline
$a$&$0$&$a$&$a$&0&0&$0$&$0$  \\ \hline

$b$&$a$&$b$&$b$&0&$0$&$a$&$0$  \\ \hline
$c$&$a$&$b$&$c$&$0$&$0$&$a$&$0$  \\ \hline
$x$&0&$0$&$0$&0&$x$&$0$&$x$  \\ \hline
$y$&0&$0$&$0$&$x$&$y$&$x$&$y$  \\ \hline
$z$&$0$&$a$&$a$&$0$&$x$&$0$&$x$  \\ \hline
$w$&$0$&$0$&$0$&$x$&$y$&$x$&$y$\\ \hline

\end{tabular}
\end{center}
\end{example}
\begin{example}The connected graph $G884$ which is defined by  $N(a)=\{b,c, x,y, z,w\},$ $N(b)=\{a,c,x,y\},$ $N(c)=\{a,b, x,y\},$ $N(x)=\{a,b,c, y\},$ $N(y)=\{a,b,c, x\},$ $N(z)=N(w)=\{a\}$ is a zero-divisor graph because it is a refinement of a star graph..
\end{example}
\begin{example}
$G885$ is a zero-divisor graph.
\begin{center}G885
\begin{tabular}{|l|c|r|r|r|r|r|r|}
\hline
$\bullet$&$1$&$2$&$3$&$4$&$5$&$6$&$7$  \\ \hline
$1$&0&$0$&0&0&0&3&4  \\ \hline
$2$&0&$2$&0&0&0&2&2  \\ \hline

$3$&0&$0$&0&0&0&3&0  \\ \hline
$4$&0&0&$0$&0&0&0&4 \\ \hline
$5$&0&$0$&0&0&0&3&4  \\ \hline

$6$&3&2&3&0&$3$&6&2 \\ \hline
$7$&4&2&0&4&4&$2$&7\\ \hline
\end{tabular}
\end{center}

\end{example}
\begin{example} $G886$ is a zero-divisor graph because it is a refinement of a star graph.
\end{example} 
\begin{example} $G887$ is a zero-divisor graph (combining $G157$  and \cite[Lemma 3.14]{DGSW} because $a^2\not=0$).
\begin{center} G887
\begin{tabular}{|l|c|r|r|r|r|r|r|}
\hline
$\bullet$&$a$&$b$&$c$&$x$&$y$&$z$&$w$  \\ \hline
$a$&$a$&$a$&$a$&0&0&$0$&$a$  \\ \hline
$b$&$a$&$b$&$a$&$x$&$0$&$0$&$a$  \\ \hline

$c$&$a$&$a$&$c$&$0$&$y$&$z$&$a$  \\ \hline
$x$&0&$x$&$0$&$x$&$0$&$0$&$0$  \\ \hline
$y$&0&$0$&$y$&$0$&$y$&$0$&$0$  \\ \hline
$z$&$0$&$0$&$z$&$0$&$0$&$z$&$0$  \\ \hline
$w$&$a$&$a$&$a$&$0$&$0$&$0$&$a$\\ \hline

\end{tabular}
\end{center}

\end{example}
\begin{example}$G888$ is a zero-divisor graph because it is a refinement of a star graph.
\end{example}
\begin{example} $G889$ is a zero-divisor graph.
\begin{center}G889
\begin{tabular}{|l|c|r|r|r|r|r|r|}
\hline
$\bullet$&$1$&$2$&$3$&$4$&$5$&$6$&$7$  \\ \hline
$1$&1&$0$&7&4&5&4&7  \\ \hline
$2$&0&$0$&0&0&0&2&0  \\ \hline

$3$&7&$0$&5&0&0&2&5  \\ \hline
$4$&4&0&$0$&4&0&4&0 \\ \hline
$5$&5&$0$&0&0&0&0&0  \\ \hline

$6$&4&2&2&4&$0$&6&0 \\ \hline
$7$&7&0&5&0&0&$0$&5\\ \hline
\end{tabular}
\end{center}
\end{example} 
\begin{example}$G890$ is a zero-divisor graph (combining $G113$  and \cite[Lemma 3.14]{DGSW} because $x^2=0$).
\begin{center} G890
\begin{tabular}{|l|c|r|r|r|r|r|r|}
\hline
$\bullet$&$a$&$b$&$c$&$x$&$y$&$z$&$w$  \\ \hline
$a$&$a$&$a$&$a$&0&0&$0$&$0$  \\ \hline

$b$&$a$&$a$&$a$&0&$0$&$x$&$0$  \\ \hline
$c$&$a$&$a$&$c$&$0$&$y$&$y$&$0$  \\ \hline
$x$&0&$0$&$0$&0&$0$&$x$&0  \\ \hline
$y$&0&$0$&$y$&$0$&$y$&$y$&$0$  \\ \hline
$z$&$0$&$x$&$y$&$x$&$y$&$z$&$x$  \\ \hline
$w$&0&0&$0$&0&$0$&$x$&$0$\\ \hline

\end{tabular}
\end{center}
\end{example}
\begin{example}$G891$ is a zero-divisor graph (combining $G137$  and \cite[Lemma 3.14]{DGSW} because $a^2=0$).
\begin{center} G891
\begin{tabular}{|l|c|r|r|r|r|r|r|}
\hline
$\bullet$&$a$&$b$&$c$&$x$&$y$&$z$&$w$  \\ \hline
$a$&$0$&$a$&$a$&0&0&$0$&$0$  \\ \hline

$b$&$a$&$b$&$b$&0&$0$&$$a$$&$a$  \\ \hline
$c$&$a$&$b$&$c$&$x$&$0$&$a$&$a$  \\ \hline
$x$&0&$0$&$x$&$x$&$0$&$0$&$0$  \\ \hline
$y$&0&$0$&$0$&$0$&$y$&$y$&$0$  \\ \hline
$z$&$0$&$a$&$a$&$0$&$y$&$y$&$0$  \\ \hline
$w$&$0$&$a$&$a$&$0$&$0$&$0$&$0$\\ \hline

\end{tabular}
\end{center}
\end{example}
\begin{example}$G894$ is a zero-divisor graph.
\begin{center}G894
\begin{tabular}{|l|c|r|r|r|r|r|r|}
\hline
$\bullet$&$1$&$2$&$3$&$4$&$5$&$6$&$7$  \\ \hline
$1$&7&$0$&3&4&5&2&7  \\ \hline
$2$&0&$0$&0&0&0&2&0  \\ \hline

$3$&3&$0$&3&0&0&0&3  \\ \hline
$4$&4&0&$0$&4&0&0&4 \\ \hline
$5$&5&$0$&0&0&5&0&5  \\ \hline

$6$&2&2&0&0&$0$&6&0 \\ \hline
$7$&7&0&3&4&5&$0$&7\\ \hline
\end{tabular}
\end{center}
\end{example} 
\begin{example} $G896-G898$ are zero-divisor graph because they are the refinement of the star graph.
\end{example}
\begin{example}$G902$ is a zero-divisor graph (combining $G159$  and \cite[Lemma 3.14]{DGSW} because $a\times a\not=0$).
\begin{center} G902
\begin{tabular}{|l|c|r|r|r|r|r|r|}
\hline
$\bullet$&$a$&$b$&$c$&$x$&$y$&$z$&$w$  \\ \hline
$a$&$a$&$a$&$a$&0&0&$0$&$a$  \\ \hline
$b$&$a$&$b$&$a$&0&0&$0$&$a$  \\ \hline

$c$&$a$&$a$&$c$&$0$&$z$&$z$&$a$  \\ \hline
$x$&0&$0$&$0$&$x$&$x$&$0$&$0$  \\ \hline
$y$&0&$0$&$z$&$x$&$x$&$0$&$0$  \\ \hline
$z$&$0$&$0$&$z$&$0$&$0$&$0$&$0$  \\ \hline
$w$&$a$&$a$&$a$&$0$&$0$&$0$&$a$\\ \hline

\end{tabular}
\end{center}
\end{example}
\begin{example} $G906$ is a zero-divisor graph.
\begin{center} G906
\begin{tabular}{|l|c|r|r|r|r|r|r|}
\hline
$\bullet$&$a$&$b$&$c$&$x$&$y$&$z$&$w$  \\ \hline
$a$&$a$&$a$&$0$&0&0&$0$&$a$  \\ \hline
$b$&$a$&$b$&$0$&$x$&0&$0$&$a$  \\ \hline

$c$&$0$&$0$&$c$&$0$&$0$&$c$&$c$  \\ \hline
$x$&0&$x$&$0$&$x$&$0$&$0$&0  \\ \hline
$y$&0&$0$&$0$&$0$&$y$&$y$&$y$  \\ \hline
$z$&$0$&$0$&$c$&$0$&$y$&$z$&$z$  \\ \hline
 $w$&$a$&$a$&$c$&$0$&$y$&$z$&$w$\\ \hline

\end{tabular}
\end{center}
\end{example}
\begin{example} $G908$ is a zero-divisor graph.
\begin{center} G908
\begin{tabular}{|l|c|r|r|r|r|r|r|}
\hline
$\bullet$&$a$&$b$&$c$&$x$&$y$&$z$&$w$  \\ \hline
$a$&$a$&$a$&$0$&0&0&$0$&$a$  \\ \hline
$b$&$a$&$a$&$0$&0&0&$0$&$a$  \\ \hline

$c$&$0$&$0$&$0$&$0$&$0$&$c$&$c$  \\ \hline
$x$&0&$0$&$0$&$x$&$x$&$0$&0  \\ \hline
$y$&0&$0$&$0$&$x$&$x$&$c$&$c$  \\ \hline
$z$&$0$&$0$&$c$&$0$&$c$&$z$&$z$  \\ \hline
 $w$&$a$&$a$&$c$&$0$&$c$&$z$&$w$\\ \hline

\end{tabular}
\end{center}
\end{example}
\begin{example} $G909$ is a zero-divisor graph (combining $G159$  and \cite[Lemma 3.14]{DGSW} because $y^2\not=0$).
\begin{center} G909
\begin{tabular}{|l|c|r|r|r|r|r|r|}
\hline
$\bullet$&$a$&$b$&$c$&$x$&$y$&$z$&$w$  \\ \hline
$a$&$a$&$a$&$a$&0&0&$0$&$0$  \\ \hline
$b$&$a$&$b$&$a$&0&0&$0$&$0$  \\ \hline

$c$&$a$&$a$&$c$&$0$&$z$&$z$&$z$  \\ \hline
$x$&0&$0$&$0$&$x$&$x$&$0$&$x$  \\ \hline
$y$&0&$0$&$z$&$x$&$x$&$0$&$x$  \\ \hline
$z$&$0$&$0$&$z$&$0$&$0$&$0$&$0$  \\ \hline
$w$&$0$&$0$&$z$&$x$&$x$&$0$&$x$\\ \hline

\end{tabular}
\end{center}

\end{example}
\begin{example} $G913-G914$ are  zero-divisor graphs because they are  the refinement of  star graphs.
\end{example}
\begin{example} $G915$ is a zero-divisor graph.
\begin{center}G915
\begin{tabular}{|l|c|r|r|r|r|r|r|}
\hline
$\bullet$&$1$&$2$&$3$&$4$&$5$&$6$&$7$  \\ \hline
$1$&1&$0$&7&7&0&6&7  \\ \hline
$2$&0&2&0&2&0&0&0  \\ \hline

$3$&7&$0$&3&7&5&0&7  \\ \hline
$4$&7&2&$7$&4&0&0&7 \\ \hline
$5$&0&0&5&0&5&0&0  \\ \hline

$6$&6&0&0&0&$0$&6&0 \\ \hline
$7$&7&0&7&7&0&$0$&7\\ \hline
\end{tabular}
\end{center}
\end{example}
\begin{example} $G916$ is a zero-divisor graph because it is a refinement of a star graph.
\end{example}
\begin{example}$G919-G925$ are zero-divisor graphs because they are the refinement of star graphs.
\end{example}
\begin{example} $G927$ is a zero-divisor graph.
\begin{center}G927
\begin{tabular}{|l|c|r|r|r|r|r|r|}
\hline
$\bullet$&$1$&$2$&$3$&$4$&$5$&$6$&$7$  \\ \hline
$1$&1&$0$&4&4&0&1&0  \\ \hline
$2$&0&0&0&0&2&0&0  \\ \hline

$3$&4&$0$&7&0&2&4&7  \\ \hline
$4$&4&0&0&0&0&4&0 \\ \hline
$5$&0&2&2&0&5&2&0  \\ \hline

$6$&1&0&4&4&$2$&1&0 \\ \hline
$7$&0&0&7&0&0&$0$&7\\ \hline
\end{tabular}
\end{center}

\end{example}
\begin{example} $G929$ is a zero-divisor graph (combining $G166$  and \cite[Lemma 3.14]{DGSW} since $a^2\not=0$).
\begin{center} G929
\begin{tabular}{|l|c|r|r|r|r|r|r|}
\hline
$\bullet$&$a$&$b$&$c$&$x$&$y$&$z$&$w$  \\ \hline
$a$&$a$&$a$&$a$&0&0&$0$&$a$  \\ \hline
$b$&$a$&$a$&$a$&0&$x$&$0$&$a$  \\ \hline

$c$&$a$&$a$&$c$&$0$&$0$&$z$&$a$  \\ \hline
$x$&0&$0$&$0$&$0$&$x$&$0$&$0$  \\ \hline
$y$&0&$x$&$0$&$x$&$y$&$0$&$0$  \\ \hline
$z$&$0$&$0$&$z$&$0$&$0$&$z$&$0$  \\ \hline
$w$&$a$&$a$&$a$&$0$&$0$&$0$&$a$\\ \hline

\end{tabular}
\end{center}
\end{example}
\begin{example} $G930$ is a zero-divisor graph (combining $G118$  and \cite[Lemma 3.14]{DGSW} since $x^2=0$).
\begin{center} G930
\begin{tabular}{|l|c|r|r|r|r|r|r|}
\hline
$\bullet$&$a$&$b$&$c$&$x$&$y$&$z$&$w$  \\ \hline
$a$&$0$&$a$&$a$&0&0&$0$&$0$  \\ \hline

$b$&$a$&$b$&$b$&0&$0$&$$a$$&$0$  \\ \hline
$c$&$a$&$b$&$b$&$0$&$x$&$a$&$0$  \\ \hline
$x$&0&$0$&$0$&0&$x$&$0$&$0$  \\ \hline
$y$&0&$0$&$x$&$x$&$y$&$x$&$x$  \\ \hline
$z$&$0$&$a$&$a$&$0$&$x$&$0$&$0$  \\ \hline
$w$&$0$&$0$&$0$&$0$&$x$&$0$&$0$\\ \hline

\end{tabular}
\end{center}
\end{example}
\begin{example} $G932$ is a zero-divisor graph (combining $G185$  and \cite[Lemma 3.14]{DGSW} since $c^2\not=0$).
\begin{center} G932
\begin{tabular}{|l|c|r|r|r|r|r|r|}
\hline
$\bullet$&$a$&$b$&$c$&$x$&$y$&$z$&$w$  \\ \hline
$a$&$a$&$a$&$0$&0&0&$0$&$0$  \\ \hline

$b$&$a$&$a$&$x$&$0$&$0$&$0$&$x$  \\ \hline
$c$&$0$&$x$&$c$&$x$&$0$&$x$&$c$  \\ \hline
$x$&0&$0$&$x$&$0$&$0$&$0$&$x$  \\ \hline
$y$&0&$0$&$0$&$0$&$y$&$y$&$0$  \\ \hline
$z$&$0$&$0$&$x$&$0$&$y$&$y$&$x$  \\ \hline
$w$&$0$&$x$&$c$&$x$&$0$&$x$&$c$\\ \hline

\end{tabular}
\end{center}
\end{example}
\begin{example}$G934$ is a zero-divisor graph (combining $G186$  and \cite[Lemma 3.14]{DGSW} because $c^2\not=0.$).
\begin{center} G934
\begin{tabular}{|l|c|r|r|r|r|r|r|}
\hline
$\bullet$&$a$&$b$&$c$&$x$&$y$&$z$&$w$  \\ \hline
$a$&$a$&$a$&$a$&0&0&$0$&$a$  \\ \hline
$b$&$a$&$a$&$a$&0&0&$0$&$a$  \\ \hline
$c$&$a$&$a$&$c$&$x$&0&$0$&$c$  \\ \hline
$x$&0&$0$&$x$&$x$&$0$&$0$&$x$  \\ \hline
$y$&0&$0$&$0$&$0$&$y$&$y$&$0$  \\ \hline
$z$&$0$&$0$&$0$&$0$&$y$&$y$&$0$  \\ \hline
$w$&$a$&$a$&$c$&$x$&$0$&$0$&$c$\\ \hline

\end{tabular}
\end{center}
\end{example}
\begin{example} $G939$ is a zero-divisor graph (combining $G137$  and \cite[Lemma 3.14]{DGSW} because $y^2\not=0$).
\begin{center} G939
\begin{tabular}{|l|c|r|r|r|r|r|r|}
\hline
$\bullet$&$a$&$b$&$c$&$x$&$y$&$z$&$w$  \\ \hline
$a$&$0$&$a$&$a$&0&0&$0$&$0$  \\ \hline

$b$&$a$&$b$&$b$&0&$0$&$$a$$&$0$  \\ \hline
$c$&$a$&$b$&$c$&$x$&$0$&$a$&$0$  \\ \hline
$x$&0&$0$&$x$&$x$&$0$&$0$&$0$  \\ \hline
$y$&0&$0$&$0$&$0$&$y$&$y$&$y$  \\ \hline
$z$&$0$&$a$&$a$&$0$&$y$&$y$&$y$  \\ \hline
$w$&$0$&$0$&$0$&$0$&$y$&$y$&$y$\\ \hline

\end{tabular}
\end{center}
\end{example}
\begin{example} $G944$ is a zero-divisor graph (combining $G185$ and \cite[Lemma 3.14]{DGSW} because $b^2\not=0$).
\begin{center} G944
\begin{tabular}{|l|c|r|r|r|r|r|r|}
\hline
$\bullet$&$a$&$b$&$c$&$x$&$y$&$z$&$w$  \\ \hline
$a$&$0$&$a$&$0$&0&0&$0$&$a$  \\ \hline

$b$&$a$&$b$&$a$&$0$&$0$&$a$&$b$  \\ \hline
$c$&$0$&$a$&$0$&$0$&$0$&$0$&$a$  \\ \hline
$x$&0&$0$&$0$&$0$&$x$&$0$&0  \\ \hline
$y$&0&$0$&$0$&$x$&$y$&$x$&$0$  \\ \hline
$z$&$0$&$a$&$0$&$0$&$x$&$0$&$a$  \\ \hline
$w$&$a$&$b$&$a$&$0$&$0$&$a$&$b$\\ \hline

\end{tabular}
\end{center}
\end{example}
\begin{example} $G948$ is a zero-divisor graph (combining $G140$  and \cite[Lemma 3.14]{DGSW} because $x^2\not=0$).
\begin{center} G948
\begin{tabular}{|l|c|r|r|r|r|r|r|}
\hline
$\bullet$&$a$&$b$&$c$&$x$&$y$&$z$&$w$  \\ \hline
$a$&$a$&$a$&$a$&0&0&$0$&$0$  \\ \hline

$b$&$a$&$b$&$a$&0&$0$&$$0$$&$0$  \\ \hline
$c$&$a$&$a$&$c$&$0$&$z$&$z$&$0$  \\ \hline
$x$&0&$0$&$0$&$x$&$x$&$0$&$x$  \\ \hline
$y$&0&$0$&$z$&$x$&$y$&$z$&$x$  \\ \hline
$z$&$0$&$0$&$z$&$0$&$z$&$z$&$0$  \\ \hline
$w$&$0$&$0$&$0$&$x$&$0$&$0$&$x$\\ \hline

\end{tabular}
\end{center}
\end{example}

\begin{example}$G950-G952$ are zero-divisor graphs because they are the refinement of star graphs.
\end{example}
\begin{example} $G957$ is a zero-divisor graph (combining $G168$  and \cite[Lemma 3.14]{DGSW} since $a^2\not=0.$).
\begin{center} G957
\begin{tabular}{|l|c|r|r|r|r|r|r|}
\hline
$\bullet$&$a$&$b$&$c$&$x$&$y$&$z$&$w$  \\ \hline
$a$&$b$&$b$&$b$&0&0&$0$&$b$  \\ \hline
$b$&$b$&$b$&$b$&0&0&$0$&$b$  \\ \hline

$c$&$b$&$b$&$b$&$0$&$0$&$x$&$b$  \\ \hline
$x$&0&$0$&$0$&$0$&$0$&$x$&$0$  \\ \hline
$y$&0&$0$&$0$&$0$&$y$&$y$&$0$  \\ \hline
$z$&$0$&$0$&$x$&$x$&$y$&$z$&$0$  \\ \hline
$w$&$b$&$b$&$b$&$0$&$0$&$0$&$b$\\ \hline

\end{tabular}
\end{center}
\end{example}
\begin{example} $G972$ is a zero-divisor graph.
\begin{center}G972
\begin{tabular}{|l|c|r|r|r|r|r|r|}
\hline
$\bullet$&$1$&$2$&$3$&$4$&$5$&$6$&$7$  \\ \hline
$1$&0&$0$&1&0&0&0&0  \\ \hline
$2$&0&2&0&4&5&5&4  \\ \hline

$3$&1&$0$&3&0&0&1&1  \\ \hline
$4$&0&4&0&4&0&0&4 \\ \hline
$5$&0&5&0&0&5&5&0  \\ \hline

$6$&0&5&1&0&5&5&0 \\ \hline
$7$&0&4&1&4&0&0&4\\ \hline
\end{tabular}
\end{center}

\end{example}
\begin{example}$G975$ is a zero-divisor graph (combining $G166$  and \cite[Lemma 3.14]{DGSW} since $y^2\not=0$).
\begin{center} G975
\begin{tabular}{|l|c|r|r|r|r|r|r|}
\hline
$\bullet$&$a$&$b$&$c$&$x$&$y$&$z$&$w$  \\ \hline
$a$&$a$&$a$&$a$&0&0&$0$&$0$  \\ \hline
$b$&$a$&$a$&$a$&0&$x$&$0$&$x$  \\ \hline

$c$&$a$&$a$&$c$&$0$&$0$&$z$&$0$  \\ \hline
$x$&0&$0$&$0$&$0$&$x$&$0$&$x$  \\ \hline
$y$&0&$x$&$0$&$x$&$y$&$0$&$y$  \\ \hline
$z$&$0$&$0$&$z$&$0$&$0$&$z$&$0$  \\ \hline
$w$&$0$&$x$&$0$&$x$&$y$&$0$&$y$\\ \hline

\end{tabular}
\end{center}
\end{example}
\begin{example} The connected graph $G1007$ which is defined by $N(a)=N(b)=N(c)=\{d,e,f,g\},$ $N(d)=N(e)=N(f)=N(g)=\{a,b,c\}$ is a zero-divisor graph because it is a bipartite graph.
\end{example}
\begin{example} $G1008-G1009$ are zero-divisor graph because they are the refinement of the star graph.
\end{example}
\begin{example} $G1012$ is a zero-divisor graph.
\end{example}
\begin{example} $G1013-G1015$ are zero-divisor graphs because they are the refinement of a star graph with seven vertices.
\end{example}
\begin{example} $G1016$ is a zero-divisor graph.
\begin{center}G1016
\begin{tabular}{|l|c|r|r|r|r|r|r|}
\hline
$\bullet$&$1$&$2$&$3$&$4$&$5$&$6$&$7$  \\ \hline
$1$&1&$0$&3&3&0&0&1  \\ \hline
$2$&0&2&0&2&0&0&0  \\ \hline

$3$&3&$0$&0&0&0&0&3  \\ \hline
$4$&3&2&0&2&0&0&3 \\ \hline
$5$&0&0&0&0&0&0&5  \\ \hline

$6$&0&0&0&0&0&0&5 \\ \hline
$7$&1&0&3&3&5&5&7\\ \hline
\end{tabular}
\end{center}

\end{example}
\begin{example} $G1017$ is a zero-divisor graph (combining $G178$  and \cite[Lemma 3.14]{DGSW} because $y^2\not=0$).
\begin{center} G1017
\begin{tabular}{|l|c|r|r|r|r|r|r|}
\hline
$\bullet$&$a$&$b$&$c$&$x$&$y$&$z$&$w$  \\ \hline
$a$&$0$&$a$&$0$&0&0&$0$&$0$  \\ \hline

$b$&$a$&$b$&$a$&$0$&$a$&$a$&$a$  \\ \hline
$c$&$0$&$a$&$0$&$0$&$0$&$0$&$0$  \\ \hline
$x$&0&$0$&$0$&$x$&$x$&$0$&$x$  \\ \hline
$y$&0&$a$&$0$&$x$&$x$&$0$&$x$  \\ \hline
$z$&$0$&$a$&$0$&$0$&$0$&$0$&$0$  \\ \hline
$w$&$0$&$a$&$0$&$x$&$x$&$0$&$x$\\ \hline

\end{tabular}
\end{center}
\end{example}
\begin{example} $G1018$ is a zero-divisor graph because it is a refinement of a star graph.
\end{example}
\begin{example} $G1019$ is a zero-divisor graph.
\begin{center}G1019
\begin{tabular}{|l|c|r|r|r|r|r|r|}
\hline
$\bullet$&$1$&$2$&$3$&$4$&$5$&$6$&$7$  \\ \hline
$1$&1&$0$&3&4&7&6&7  \\ \hline
$2$&0&2&0&0&2&0&0  \\ \hline

$3$&3&$0$&3&0&0&3&0  \\ \hline
$4$&4&0&0&4&0&4&0 \\ \hline
$5$&7&2&0&0&2&0&0  \\ \hline

$6$&6&0&3&4&0&6&0 \\ \hline
$7$&7&0&0&0&0&0&0\\ \hline
\end{tabular}
\end{center}
\end{example}
\begin{example} $G1020$ is a zero-divisor graph (combining $G157$  and \cite[Lemma 3.14]{DGSW} because $y^2\not=0$).
\begin{center} G1020
\begin{tabular}{|l|c|r|r|r|r|r|r|}
\hline
$\bullet$&$a$&$b$&$c$&$x$&$y$&$z$&$w$  \\ \hline
$a$&$a$&$a$&$a$&0&0&$0$&$0$  \\ \hline
$b$&$a$&$b$&$a$&$x$&$0$&$0$&$0$  \\ \hline

$c$&$a$&$a$&$c$&$0$&$y$&$z$&$y$  \\ \hline
$x$&0&$x$&$0$&$x$&$0$&$0$&$0$  \\ \hline
$y$&0&$0$&$y$&$0$&$y$&$0$&$y$  \\ \hline
$z$&$0$&$0$&$z$&$0$&$0$&$z$&$0$  \\ \hline
$w$&$0$&$0$&$y$&$0$&$y$&$0$&$y$\\ \hline

\end{tabular}
\end{center}
\end{example}
\begin{example}$G1025-G1028$ are zero-divisor graphs because they are the refinement of a star graph with seven vertices.
\end{example}
\begin{example} $G1029$ is a zero-divisor graph.
\begin{center}G1029
\begin{tabular}{|l|c|r|r|r|r|r|r|}
\hline
$\bullet$&$1$&$2$&$3$&$4$&$5$&$6$&$7$  \\ \hline
$1$&1&$0$&1&0&0&0&0  \\ \hline
$2$&0&2&0&0&0&2&2  \\ \hline

$3$&1&$0$&1&0&0&5&4  \\ \hline
$4$&0&0&0&0&0&0&4 \\ \hline
$5$&0&0&0&0&0&5&0  \\ \hline

$6$&0&2&5&0&5&6&2 \\ \hline
$7$&0&2&4&4&0&2&7\\ \hline
\end{tabular}
\end{center}

\end{example}
\begin{example} $G1031-G1032$ are  zero-divisor graphs because they are the refinement of a star graph.
\end{example}
\begin{example} $G1035$ is a zero-divisor graph (combining $G157$ and \cite[Lemma 3.14]{DGSW} since $x^2\not=0$).
\begin{center} G1035
\begin{tabular}{|l|c|r|r|r|r|r|r|}
\hline
$\bullet$&$a$&$b$&$c$&$x$&$y$&$z$&$w$  \\ \hline
$a$&$a$&$a$&$a$&0&0&$0$&$0$  \\ \hline
$b$&$a$&$b$&$a$&$x$&$0$&$0$&$x$  \\ \hline

$c$&$a$&$a$&$c$&$0$&$y$&$z$&$0$  \\ \hline
$x$&0&$x$&$0$&$x$&$0$&$0$&$x$  \\ \hline
$y$&0&$0$&$y$&$0$&$y$&$0$&$0$  \\ \hline
$z$&$0$&$0$&$z$&$0$&$0$&$z$&$0$  \\ \hline
$w$&$0$&$x$&$0$&$x$&$0$&$0$&$x$\\ \hline

\end{tabular}
\end{center}
\end{example}
\begin{example} $G1036$ is a zero-divisor graph.
\begin{center}G1036
\begin{tabular}{|l|c|r|r|r|r|r|r|}
\hline
$\bullet$&$1$&$2$&$3$&$4$&$5$&$6$&$7$  \\ \hline
$1$&0&$0$&0&0&0&4&3  \\ \hline
$2$&0&0&0&0&0&4&3  \\ \hline

$3$&0&0&0&0&0&0&3  \\ \hline
$4$&0&0&0&0&0&4&0 \\ \hline
$5$&0&0&0&0&0&4&3  \\ \hline

$6$&4&4&0&4&4&6&0 \\ \hline
$7$&3&3&3&0&3&0&7\\ \hline
\end{tabular}
\end{center}
\end{example}
\begin{example}
$G1037$ is a zero-divisor graph.
\begin{center}G1037
\begin{tabular}{|l|c|r|r|r|r|r|r|}
\hline
$\bullet$&$1$&$2$&$3$&$4$&$5$&$6$&$7$  \\ \hline
$1$&1&$1$&0&0&5&6&7  \\ \hline
$2$&1&1&0&0&5&6&7  \\ \hline

$3$&0&0&3&4&0&0&0  \\ \hline
$4$&0&0&4&0&0&0&0 \\ \hline
$5$&5&5&0&0&5&0&0  \\ \hline

$6$&6&6&0&0&0&6&0 \\ \hline
$7$&7&7&0&0&0&0&7\\ \hline
\end{tabular}
\end{center}

\end{example}
\begin{example} $G1038-G1042$ are zero-divisor graphs because they are the refinement of $\star$ graph.
\end{example}
\begin{example} $G1045$ is a zero-divisor graph.
\begin{center}G1045
\begin{tabular}{|l|c|r|r|r|r|r|r|}
\hline
$\bullet$&$1$&$2$&$3$&$4$&$5$&$6$&$7$  \\ \hline
$1$&1&$0$&3&0&5&3&3  \\ \hline
$2$&0&2&0&2&0&0&0  \\ \hline

$3$&3&0&3&0&0&3&3  \\ \hline
$4$&0&2&0&2&0&0&0 \\ \hline
$5$&5&0&0&0&5&0&0  \\ \hline

$6$&3&0&3&0&0&3&3 \\ \hline
$7$&3&0&3&0&0&3&3\\ \hline
\end{tabular}
\end{center}

\end{example}
\begin{example}$G1046-G1049$ are zero-divisor graphs because they are the refinement of the star graph with seven vertices.
\end{example}
\begin{example}
$G1050$ is a zero-divisor graph.
\begin{center}G1050
\begin{tabular}{|l|c|r|r|r|r|r|r|}
\hline
$\bullet$&$1$&$2$&$3$&$4$&$5$&$6$&$7$  \\ \hline
$1$&1&$0$&6&6&0&6&0  \\ \hline
$2$&0&2&0&2&0&0&2  \\ \hline

$3$&6&0&0&0&0&0&5  \\ \hline
$4$&6&2&0&2&0&0&2 \\ \hline
$5$&0&0&0&0&0&0&5  \\ \hline

$6$&6&0&0&0&0&0&0 \\ \hline
$7$&0&2&5&2&5&0&7\\ \hline
\end{tabular}
\end{center}
\end{example}
\begin{example}
$G1052$ is a zero-divisor graph.
\begin{center}G1052
\begin{tabular}{|l|c|r|r|r|r|r|r|}
\hline
$\bullet$&$1$&$2$&$3$&$4$&$5$&$6$&$7$  \\ \hline
$1$&1&$2$&3&6&0&6&0  \\ \hline
$2$&2&2&0&6&0&6&0  \\ \hline

$3$&3&0&3&0&0&0&0  \\ \hline
$4$&6&6&0&0&0&0&5 \\ \hline
$5$&0&0&0&0&0&0&5  \\ \hline

$6$&6&6&0&0&0&0&0 \\ \hline
$7$&0&0&0&5&5&0&7\\ \hline
\end{tabular}
\end{center}
\end{example}
\begin{example}
$G1053$ is a zero-divisor graph.
\begin{center}G1053
\begin{tabular}{|l|c|r|r|r|r|r|r|}
\hline
$\bullet$&$1$&$2$&$3$&$4$&$5$&$6$&$7$  \\ \hline
$1$&0&$0$&0&0&0&1&0  \\ \hline
$2$&0&2&0&4&4&0&2  \\ \hline

$3$&0&0&3&0&3&0&0  \\ \hline
$4$&0&4&0&0&0&0&4 \\ \hline
$5$&0&4&3&0&3&0&4  \\ \hline

$6$&1&0&0&0&0&6&1 \\ \hline
$7$&0&2&0&4&4&1&2\\ \hline
\end{tabular}
\end{center}
\end{example}
\begin{example}
$G1056$ is a zero-divisor graph.
\begin{center}G1056
\begin{tabular}{|l|c|r|r|r|r|r|r|}
\hline
$\bullet$&$1$&$2$&$3$&$4$&$5$&$6$&$7$  \\ \hline
$1$&1&$0$&0&4&5&5&4  \\ \hline
$2$&0&2&0&0&0&0&2  \\ \hline

$3$&0&0&3&0&0&3&0  \\ \hline
$4$&4&0&0&4&0&0&4 \\ \hline
$5$&5&0&0&0&5&5&0  \\ \hline

$6$&5&0&3&0&5&6&0 \\ \hline
$7$&4&2&0&4&0&0&7\\ \hline
\end{tabular}
\end{center}

\end{example}
\begin{example} $G1057$ is a zero-divisor graph (combining $G185$  and \cite[Lemma 3.14]{DGSW} because of $b^2\not=0$).
\begin{center} G1057
\begin{tabular}{|l|c|r|r|r|r|r|r|}
\hline
$\bullet$&$a$&$b$&$c$&$x$&$y$&$z$&$w$  \\ \hline
$a$&$a$&$a$&$0$&0&0&$0$&$a$  \\ \hline

$b$&$a$&$a$&$x$&$0$&$0$&$0$&$a$  \\ \hline
$c$&$0$&$x$&$c$&$x$&$0$&$x$&$x$  \\ \hline
$x$&0&$0$&$x$&$0$&$0$&$0$&0  \\ \hline
$y$&0&$0$&$0$&$0$&$y$&$y$&$0$  \\ \hline
$z$&$0$&$0$&$x$&$0$&$y$&$y$&$0$  \\ \hline
$w$&$a$&$a$&$x$&$0$&$0$&$0$&$a$\\ \hline

\end{tabular}
\end{center}
\end{example}
\begin{example} $G1059$ is a zero-divisor graph.
\begin{center}G1059
\begin{tabular}{|l|c|r|r|r|r|r|r|}
\hline
$\bullet$&$1$&$2$&$3$&$4$&$5$&$6$&$7$  \\ \hline
$1$&1&$0$&0&0&0&0&1  \\ \hline
$2$&0&2&0&4&4&2&0  \\ \hline

$3$&0&0&3&0&3&3&0  \\ \hline
$4$&0&4&0&0&0&4&0 \\ \hline
$5$&0&4&3&0&3&5&0  \\ \hline

$6$&0&2&3&4&5&6&0 \\ \hline
$7$&1&0&0&0&0&0&1\\ \hline
\end{tabular}
\end{center}
\end{example} 
\begin{example} $G1062$ is a zero-divisor graph (combining $G141$  and \cite[Lemma 3.14]{DGSW} since $x^2=0$).
\begin{center} G1062
\begin{tabular}{|l|c|r|r|r|r|r|r|}
\hline
$\bullet$&$a$&$b$&$c$&$x$&$y$&$z$&$w$  \\ \hline
$a$&$a$&$a$&$a$&0&0&$0$&$0$  \\ \hline

$b$&$a$&$b$&$b$&0&$0$&$z$&$0$  \\ \hline
$c$&$a$&$b$&$b$&$0$&$x$&$z$&$0$  \\ \hline
$x$&0&$0$&$0$&$0$&$x$&$0$&$0$  \\ \hline
$y$&0&$0$&$x$&$x$&$y$&$0$&$x$  \\ \hline
$z$&$0$&$z$&$z$&$0$&$0$&$z$&$0$  \\ \hline
 $w$&$0$&$0$&$0$&$0$&$x$&$0$&$0$\\ \hline

\end{tabular}
\end{center}
\end{example}
\begin{example} $G1067$ is a zero-divisor graph (combining $G166$  and \cite[lemma 3.14]{DGSW} since $z^2\not=0.$).
\begin{center} G1067
\begin{tabular}{|l|c|r|r|r|r|r|r|}
\hline
$\bullet$&$a$&$b$&$c$&$x$&$y$&$z$&$w$  \\ \hline
$a$&$a$&$a$&$a$&0&0&$0$&$0$  \\ \hline
$b$&$a$&$a$&$a$&0&$x$&$0$&$0$  \\ \hline

$c$&$a$&$a$&$c$&$0$&$0$&$z$&$z$  \\ \hline
$x$&0&$0$&$0$&$0$&$x$&$0$&$0$  \\ \hline
$y$&0&$x$&$0$&$x$&$y$&$0$&$0$  \\ \hline
$z$&$0$&$0$&$z$&$0$&$0$&$z$&$z$  \\ \hline
$w$&$0$&$0$&$z$&$0$&$0$&$z$&$z$\\ \hline

\end{tabular}
\end{center}
\end{example}
\begin{example}
$G1072$ is a zero-divisor graph (combining $G159$  and \cite[lemma 3.14]{DGSW} since $x^2\not=0.$).
\begin{center} G1072
\begin{tabular}{|l|c|r|r|r|r|r|r|}
\hline
$\bullet$&$a$&$b$&$c$&$x$&$y$&$z$&$w$  \\ \hline
$a$&$a$&$a$&$a$&0&0&$0$&$0$  \\ \hline
$b$&$a$&$b$&$a$&0&0&$0$&$0$  \\ \hline

$c$&$a$&$a$&$c$&$0$&$z$&$z$&$0$  \\ \hline
$x$&0&$0$&$0$&$x$&$x$&$0$&$x$  \\ \hline
$y$&0&$0$&$z$&$x$&$x$&$0$&$x$  \\ \hline
$z$&$0$&$0$&$z$&$0$&$0$&$0$&$0$  \\ \hline
$w$&$0$&$0$&$0$&$x$&$x$&$0$&$x$\\ \hline

\end{tabular}
\end{center}
\end{example}
\begin{example} $G1077-G1081$ are the zero-divisor graphs since they are the refinement of a star graph with seven vertices.
\end{example}
\begin{example} $G1085$ is a zero-divisor graph (combining $G119$  and \cite[Lemma 3.14]{DGSW} since $x^2=0$).
\begin{center} G1085
\begin{tabular}{|l|c|r|r|r|r|r|r|}
\hline
$\bullet$&$a$&$b$&$c$&$x$&$y$&$z$&$w$  \\ \hline
$a$&$0$&$a$&$a$&0&0&$0$&$0$  \\ \hline

$b$&$a$&$b$&$b$&0&$0$&$a$&$0$  \\ \hline
$c$&$a$&$b$&$c$&$0$&$0$&$a$&$0$  \\ \hline
$x$&0&$0$&$0$&0&$x$&$0$&$0$  \\ \hline
$y$&0&$0$&$0$&$x$&$y$&$x$&$x$  \\ \hline
$z$&$0$&$a$&$a$&$0$&$x$&$0$&$0$  \\ \hline
$w$&$0$&$0$&$0$&$0$&$x$&$0$&$0$\\ \hline

\end{tabular}
\end{center}
\end{example}
\begin{example} $G1088$ is a zero-divisor graph (applying \cite{DGSW} to $G189$  since $a^2\not=0$).
\begin{center} G1088
\begin{tabular}{|l|c|r|r|r|r|r|r|}
\hline
$\bullet$&$a$&$b$&$c$&$x$&$y$&$z$&$w$  \\ \hline
$a$&$a$&$a$&$a$&0&0&$0$&$a$  \\ \hline
$b$&$a$&$a$&$a$&0&0&$0$&$a$  \\ \hline
$c$&$a$&$a$&$a$&0&0&$0$&$a$  \\ \hline
$x$&0&$0$&$0$&$0$&$x$&$0$&$0$  \\ \hline
$y$&0&$0$&$0$&$x$&$y$&$z$&$0$  \\ \hline
$z$&$0$&$0$&$0$&$0$&$z$&$z$&$0$  \\ \hline
$w$&$a$&$a$&$a$&$0$&$0$&$0$&$a$\\ \hline

\end{tabular}
\end{center}
\end{example}
\begin{example}$G1106$ is a zero-divisor graph (combining $G140$  and \cite[Lemma 3.14]{DGSW} since $y^2\not=0$).
\begin{center} G1106
\begin{tabular}{|l|c|r|r|r|r|r|r|}
\hline
$\bullet$&$a$&$b$&$c$&$x$&$y$&$z$&$w$  \\ \hline
$a$&$0$&$a$&$a$&0&0&$0$&$0$  \\ \hline

$b$&$a$&$b$&$b$&0&$0$&$$a$$&$0$  \\ \hline
$c$&$a$&$b$&$b$&$0$&$0$&$a$&$0$  \\ \hline
$x$&0&$0$&$0$&$0$&$x$&$0$&$x$  \\ \hline
$y$&0&$0$&$0$&$x$&$y$&$0$&$y$  \\ \hline
$z$&$0$&$a$&$a$&$0$&$0$&$0$&$0$  \\ \hline
$w$&$0$&$0$&$0$&$x$&$y$&$0$&$y$\\ \hline

\end{tabular}
\end{center}
\end{example}
\begin{example} $G1108$ is a zero-divisor graph because it is a refinement of a star graph.
\end{example}
\begin{example} $G1109$ is a zero-divisor graph (combining $G157$  and \cite[Lemma 3.14]{DGSW} since $y^2=0$).
\begin{center}G1109
\begin{tabular}{|l|c|r|r|r|r|r|r|}
\hline
$\bullet$&$1$&$2$&$3$&$4$&$5$&$6$&$7$  \\ \hline
$1$&1&$0$&3&4&3&6&6  \\ \hline
$2$&0&2&0&0&2&0&0  \\ \hline

$3$&3&0&3&0&3&0&0  \\ \hline
$4$&4&0&0&4&0&0&0 \\ \hline
$5$&3&2&3&0&5&0&0  \\ \hline

$6$&6&0&0&0&0&0&0 \\ \hline
$7$&6&0&0&0&0&0&0\\ \hline
\end{tabular}
\end{center}

\end{example}
\begin{example}$G1110$ is a zero-divisor graph because it is a refinement of a star graph.
\end{example}
\begin{example}$G1111$ is a zero-divisor graph.
\begin{center}G1111
\begin{tabular}{|l|c|r|r|r|r|r|r|}
\hline
$\bullet$&$1$&$2$&$3$&$4$&$5$&$6$&$7$  \\ \hline
$1$&1&$0$&3&3&3&6&6  \\ \hline
$2$&0&2&0&2&0&0&0  \\ \hline

$3$&3&0&0&0&0&0&0  \\ \hline
$4$&3&2&0&2&0&0&0 \\ \hline
$5$&3&0&0&0&0&0&0  \\ \hline

$6$&6&0&0&0&0&3&3 \\ \hline
$7$&6&0&0&0&0&3&3\\ \hline
\end{tabular}
\end{center}

\end{example}
\begin{example}$G1113-G1114$ are zero-divisor graphs because they are the refinement of a star graph with seven vertices.
\end{example}
\begin{example} $G1115$ is a zero-divisor graph.
\begin{center}G1115
\begin{tabular}{|l|c|r|r|r|r|r|r|}
\hline
$\bullet$&$1$&$2$&$3$&$4$&$5$&$6$&$7$  \\ \hline
$1$&1&$1$&1&0&0&0&0  \\ \hline
$2$&1&2&1&4&0&0&4  \\ \hline

$3$&1&1&3&0&5&6&0  \\ \hline
$4$&0&4&0&0&0&0&0 \\ \hline
$5$&0&0&5&0&5&0&0  \\ \hline

$6$&0&0&6&0&0&6&0 \\ \hline
$7$&0&4&0&0&0&0&0\\ \hline
\end{tabular}
\end{center}

\end{example}
\begin{example} $G1116-G1119$ are zero-divisor graphs because they are the refinement a star graph with seven vertices.
\end{example}
\begin{example}$G1121-G1124$ are zero-divisor graphs because they are the refinement of a star graph with seven vertices.
\end{example}
\begin{example} $G1125$ is a zero-divisor graph.
\begin{center}G1125
\begin{tabular}{|l|c|r|r|r|r|r|r|}
\hline
$\bullet$&$1$&$2$&$3$&$4$&$5$&$6$&$7$  \\ \hline
$1$&1&$0$&0&0&0&0&1  \\ \hline
$2$&0&2&6&4&0&6&0  \\ \hline

$3$&0&6&3&0&5&6&5  \\ \hline
$4$&0&4&0&4&0&0&0 \\ \hline
$5$&0&0&5&0&0&0&0  \\ \hline

$6$&0&6&6&0&0&6&0 \\ \hline
$7$&1&0&5&0&0&0&1\\ \hline
\end{tabular}
\end{center}
\end{example}
\begin{example} $G1126$ is a zero-divisor graph.
\begin{center}G1126
\begin{tabular}{|l|c|r|r|r|r|r|r|}
\hline
$\bullet$&$1$&$2$&$3$&$4$&$5$&$6$&$7$  \\ \hline
$1$&3&$0$&3&2&0&0&3  \\ \hline
$2$&0&0&0&2&0&0&0  \\ \hline

$3$&3&0&3&0&0&0&3  \\ \hline
$4$&2&2&0&4&0&0&0 \\ \hline
$5$&0&0&0&0&0&0&5  \\ \hline

$6$&0&0&0&0&0&0&5 \\ \hline
$7$&3&0&3&0&5&5&7\\ \hline
\end{tabular}
\end{center}

\end{example}
\begin{example} $G1128$ is a zero-divisor graph.
\begin{center}G1128
\begin{tabular}{|l|c|r|r|r|r|r|r|}
\hline
$\bullet$&$1$&$2$&$3$&$4$&$5$&$6$&$7$  \\ \hline
$1$&1&$2$&3&0&0&7&7  \\ \hline
$2$&2&0&0&0&0&3&3  \\ \hline

$3$&3&0&0&0&0&0&0  \\ \hline
$4$&0&0&0&4&4&0&0 \\ \hline
$5$&0&0&0&4&4&0&0  \\ \hline

$6$&7&3&0&0&0&0&0 \\ \hline
$7$&7&3&0&0&0&0&0\\ \hline
\end{tabular}
\end{center}

\end{example} 
\begin{example} $G1129$ is a zero-divisor graph.
\begin{center}G1129
\begin{tabular}{|l|c|r|r|r|r|r|r|}
\hline
$\bullet$&$1$&$2$&$3$&$4$&$5$&$6$&$7$  \\ \hline
$1$&1&$0$&3&0&3&3&0  \\ \hline
$2$&0&2&0&4&4&4&4  \\ \hline

$3$&3&0&0&0&0&0&0  \\ \hline
$4$&0&4&0&0&0&0&0 \\ \hline
$5$&3&4&0&0&0&0&0  \\ \hline

$6$&3&4&0&0&0&0&0 \\ \hline
$7$&0&4&0&0&0&0&0\\ \hline
\end{tabular}
\end{center}

\end{example}
\begin{example} $G1131$ is a zero-divisor graph.
\begin{center}G1131
\begin{tabular}{|l|c|r|r|r|r|r|r|}
\hline
$\bullet$&$1$&$2$&$3$&$4$&$5$&$6$&$7$  \\ \hline
$1$&4&$0$&0&4&4&2&0  \\ \hline
$2$&0&0&0&0&0&2&0  \\ \hline

$3$&0&0&0&0&0&2&0  \\ \hline
$4$&4&0&0&4&4&0&0 \\ \hline
$5$&4&0&0&4&4&0&0  \\ \hline

$6$&2&2&2&0&0&6&2 \\ \hline
$7$&0&0&0&0&0&2&0\\ \hline
\end{tabular}
\end{center}
\end{example}
\begin{example} $G1132$ is a zero-divisor graph because it is a refinement of star graph with six vertices.
\end{example}
\begin{example}$G1134$ is a zero-divisor graph.
\begin{center}G1134
\begin{tabular}{|l|c|r|r|r|r|r|r|}
\hline
$\bullet$&$1$&$2$&$3$&$4$&$5$&$6$&$7$  \\ \hline
$1$&1&$0$&4&4&0&6&6  \\ \hline
$2$&0&0&0&0&2&0&0  \\ \hline

$3$&4&0&0&0&2&0&0  \\ \hline
$4$&4&0&0&0&0&0&0 \\ \hline
$5$&0&2&2&0&5&0&0  \\ \hline

$6$&6&0&0&0&0&6&6 \\ \hline
$7$&6&0&0&0&0&6&6\\ \hline
\end{tabular}
\end{center}

\end{example}
\begin{example}
$G1135$ is a zero-divisor graph.
\begin{center}G1135
\begin{tabular}{|l|c|r|r|r|r|r|r|}
\hline
$\bullet$&$1$&$2$&$3$&$4$&$5$&$6$&$7$  \\ \hline
$1$&4&$0$&0&4&4&3&2  \\ \hline
$2$&0&0&0&0&0&0&2  \\ \hline

$3$&0&0&0&0&0&3&0  \\ \hline
$4$&4&0&0&4&4&0&0 \\ \hline
$5$&4&0&0&4&4&0&0  \\ \hline

$6$&3&0&3&0&0&6&0 \\ \hline
$7$&2&2&0&0&0&0&7\\ \hline
\end{tabular}
\end{center}
\end{example}

\begin{example}
$G1137$ is a zero-divisor graph.
\begin{center}G1137
\begin{tabular}{|l|c|r|r|r|r|r|r|}
\hline
$\bullet$&$1$&$2$&$3$&$4$&$5$&$6$&$7$  \\ \hline
$1$&1&$0$&0&1&0&0&1  \\ \hline
$2$&0&2&0&2&0&2&0  \\ \hline

$3$&0&0&3&0&3&0&0  \\ \hline
$4$&1&2&0&4&0&2&1 \\ \hline
$5$&0&0&3&0&3&0&0  \\ \hline

$6$&0&2&0&2&0&2&0 \\ \hline
$7$&1&0&0&1&0&0&1\\ \hline
\end{tabular}
\end{center}
\end{example}
\begin{example}
$G1138-G1144$ are zero-divisor graphs because they are the refinement of a star graphs with seven vertices.
\end{example}
\begin{example}$G1145$ is a zero-divisor graph.
\begin{center}G1145
\begin{tabular}{|l|c|r|r|r|r|r|r|}
\hline
$\bullet$&$1$&$2$&$3$&$4$&$5$&$6$&$7$  \\ \hline
$1$&0&$0$&0&0&0&0&1  \\ \hline
$2$&0&0&0&0&0&2&0  \\ \hline

$3$&0&0&5&0&5&2&1  \\ \hline
$4$&0&0&0&0&0&2&1 \\ \hline
$5$&0&0&5&0&5&0&0  \\ \hline

$6$&0&2&2&2&0&6&0 \\ \hline
$7$&1&0&1&1&0&0&7\\ \hline
\end{tabular}
\end{center}
\end{example}
\begin{example}$G1147$ is a zero-divisor graph.
\begin{center}G1147
\begin{tabular}{|l|c|r|r|r|r|r|r|}
\hline
$\bullet$&$1$&$2$&$3$&$4$&$5$&$6$&$7$  \\ \hline
$1$&1&$0$&7&7&0&0&7  \\ \hline
$2$&0&2&5&0&5&6&0  \\ \hline

$3$&7&5&5&0&5&0&0  \\ \hline
$4$&7&0&0&0&0&0&0 \\ \hline
$5$&0&5&5&0&5&0&0  \\ \hline

$6$&0&6&0&0&0&6&0 \\ \hline
$7$&7&0&0&0&0&0&0\\ \hline
\end{tabular}
\end{center}

\end{example}
\begin{example}
$G1148$ is a zero-divisor graph.
\begin{center}G1148
\begin{tabular}{|l|c|r|r|r|r|r|r|}
\hline
$\bullet$&$1$&$2$&$3$&$4$&$5$&$6$&$7$  \\ \hline
$1$&1&$1$&0&0&0&0&0  \\ \hline
$2$&1&1&0&0&0&0&0  \\ \hline

$3$&0&0&0&0&0&3&0  \\ \hline
$4$&0&0&0&4&4&0&4 \\ \hline
$5$&0&0&0&4&4&3&4  \\ \hline

$6$&0&0&3&0&3&6&3 \\ \hline
$7$&0&0&0&4&4&3&4\\ \hline
\end{tabular}
\end{center}

\end{example}
\begin{example} $G1149-G1152$ are aero-divisor graphs because they are the refinement of a star graph with seven vertices.
\end{example}
\begin{example} $G1157$ is a zero-divisor graph.
\begin{center}G1157
\begin{tabular}{|l|c|r|r|r|r|r|r|}
\hline
$\bullet$&$1$&$2$&$3$&$4$&$5$&$6$&$7$  \\ \hline
$1$&1&$0$&3&2&0&0&3  \\ \hline
$2$&0&0&0&2&0&0&0  \\ \hline

$3$&3&0&0&0&0&0&0  \\ \hline
$4$&2&2&0&4&0&0&0 \\ \hline
$5$&0&0&0&0&5&5&5  \\ \hline

$6$&0&0&0&0&5&5&5 \\ \hline
$7$&3&0&0&0&5&5&5\\ \hline
\end{tabular}
\end{center}

\end{example}
\begin{example} $G1163$ is a zero-divisor graph.
\begin{center}G1163
\begin{tabular}{|l|c|r|r|r|r|r|r|}
\hline
$\bullet$&$1$&$2$&$3$&$4$&$5$&$6$&$7$  \\ \hline
$1$&1&$0$&0&0&1&0&0  \\ \hline
$2$&0&2&0&0&2&0&0  \\ \hline

$3$&0&0&3&0&0&3&3  \\ \hline
$4$&0&0&0&4&0&4&4 \\ \hline
$5$&1&2&0&0&5&0&0  \\ \hline

$6$&0&0&3&4&0&6&6 \\ \hline
$7$&0&0&3&4&0&6&6\\ \hline
\end{tabular}
\end{center}
\end{example}
\begin{example} $G1169$ is a zero-divisor graph.
\begin{center}G1169
\begin{tabular}{|l|c|r|r|r|r|r|r|}
\hline
$\bullet$&$1$&$2$&$3$&$4$&$5$&$6$&$7$  \\ \hline
$1$&1&$0$&5&0&5&6&0  \\ \hline
$2$&0&2&0&2&0&0&2  \\ \hline

$3$&5&0&5&0&5&0&0  \\ \hline
$4$&0&2&0&2&0&0&2 \\ \hline
$5$&5&0&5&0&5&0&0  \\ \hline

$6$&6&0&0&0&0&6&0 \\ \hline
$7$&0&2&0&2&0&0&2\\ \hline
\end{tabular}
\end{center}
\end{example}

\begin{example} $G1173$ is a zero-divisor graph because it is a refinement of a star graph with seven vertices.
\end{example}
\begin{example}
$G1174$ is a zero-divisor graph.
\begin{center}G1174
\begin{tabular}{|l|c|r|r|r|r|r|r|}
\hline
$\bullet$&$1$&$2$&$3$&$4$&$5$&$6$&$7$  \\ \hline
$1$&0&$0$&0&0&0&0&1  \\ \hline
$2$&0&5&0&0&5&0&1  \\ \hline

$3$&0&0&0&0&0&0&1  \\ \hline
$4$&0&0&0&0&0&0&1 \\ \hline
$5$&0&5&0&0&5&0&0  \\ \hline

$6$&0&0&0&0&0&0&1 \\ \hline
$7$&1&1&1&1&0&1&7\\ \hline
\end{tabular}
\end{center}
\end{example}
\begin{example} $G1175-G1176$ are zero-divisor graphs because they are the refinement of a star graph with seven vertices.
\end{example}
\begin{example} $G1178-G1180$ are zero-divisor graphs because they are the refinement of a star graph with seven vertices.
\end{example}
\begin{example} $G1181$ is a zero-divisor graph.
\begin{center}G1181
\begin{tabular}{|l|c|r|r|r|r|r|r|}
\hline
$\bullet$&$1$&$2$&$3$&$4$&$5$&$6$&$7$  \\ \hline
$1$&1&$0$&0&4&4&4&4  \\ \hline
$2$&0&2&0&0&0&0&2  \\ \hline

$3$&0&0&3&0&0&3&0  \\ \hline
$4$&4&0&0&0&0&0&0 \\ \hline
$5$&4&0&0&0&0&0&0  \\ \hline

$6$&4&0&3&0&0&3&0 \\ \hline
$7$&4&2&0&0&0&0&2\\ \hline
\end{tabular}
\end{center}

\end{example}

\begin{example} $G1182$ is a zero-divisor graph.
\begin{center}G1182
\begin{tabular}{|l|c|r|r|r|r|r|r|}
\hline
$\bullet$&$1$&$2$&$3$&$4$&$5$&$6$&$7$  \\ \hline
$1$&1&$0$&3&3&3&3&0  \\ \hline
$2$&0&2&0&0&7&0&7  \\ \hline

$3$&3&0&0&0&0&0&0  \\ \hline
$4$&3&0&0&0&0&0&0 \\ \hline
$5$&3&7&0&0&0&0&0  \\ \hline

$6$&3&0&0&0&0&0&0 \\ \hline
$7$&0&7&0&0&0&0&0\\ \hline
\end{tabular}
\end{center}

\end{example}

\begin{example} $G1183$ is a zero-divisor graph.
\begin{center}G1183
\begin{tabular}{|l|c|r|r|r|r|r|r|}
\hline
$\bullet$&$1$&$2$&$3$&$4$&$5$&$6$&$7$  \\ \hline
$1$&1&$2$&0&0&5&6&2  \\ \hline
$2$&2&2&0&0&0&0&2  \\ \hline

$3$&0&0&0&3&0&0&0  \\ \hline
$4$&0&0&3&4&0&0&0 \\ \hline
$5$&5&0&0&0&5&0&0  \\ \hline

$6$&6&0&0&0&0&6&0 \\ \hline
$7$&2&2&0&0&0&0&2\\ \hline
\end{tabular}
\end{center}

\end{example}
\begin{example}$G1184-G1193$ are zero-divisor graphs because they are the refinement of a star graph with seven vertices.
\end{example}
\begin{example} $G1194$ is a zero-divisor graph.

\begin{center}G1194
\begin{tabular}{|l|c|r|r|r|r|r|r|}
\hline
$\bullet$&$1$&$2$&$3$&$4$&$5$&$6$&$7$  \\ \hline
$1$&1&$0$&3&3&0&3&0  \\ \hline
$2$&0&2&0&0&0&2&0  \\ \hline

$3$&3&0&0&0&0&0&0  \\ \hline
$4$&3&0&0&0&0&0&0 \\ \hline
$5$&0&0&0&0&5&7&7  \\ \hline

$6$&3&2&0&0&7&2&0 \\ \hline
$7$&0&0&0&0&7&0&0\\ \hline
\end{tabular}
\end{center}

\end{example}
\begin{example} $G1195$ is a zero-divisor graph.
\begin{center}G1195
\begin{tabular}{|l|c|r|r|r|r|r|r|}
\hline
$\bullet$&$1$&$2$&$3$&$4$&$5$&$6$&$7$  \\ \hline
$1$&1&$0$&0&0&0&0&1  \\ \hline
$2$&0&2&0&0&2&2&0  \\ \hline

$3$&0&0&3&0&0&3&0  \\ \hline
$4$&0&0&0&4&4&0&0 \\ \hline
$5$&0&2&0&4&5&2&0  \\ \hline

$6$&0&2&3&0&2&6&0 \\ \hline
$7$&1&0&0&0&0&0&1\\ \hline
\end{tabular}
\end{center}

\end{example}

\begin{example} $G1196$ is a zero-divisor graph.
\begin{center}G1196
\begin{tabular}{|l|c|r|r|r|r|r|r|}
\hline
$\bullet$&$1$&$2$&$3$&$4$&$5$&$6$&$7$  \\ \hline
$1$&0&$0$&0&0&0&4&3  \\ \hline
$2$&0&0&0&0&0&0&3  \\ \hline

$3$&0&0&0&0&0&0&3  \\ \hline
$4$&0&0&0&0&0&4&0 \\ \hline
$5$&0&0&0&0&0&4&0  \\ \hline

$6$&4&0&0&4&4&6&0 \\ \hline
$7$&3&3&3&0&0&0&7\\ \hline
\end{tabular}
\end{center}

\end{example}
\begin{example}$G1197-G1199$ are zero-divisor graphs because they are the refinement of a star graph with seven vertices.
\end{example}
\begin{example}
$G1200$ is a zero-divisor graph.

\begin{center}G1200
\begin{tabular}{|l|c|r|r|r|r|r|r|}
\hline
$\bullet$&$1$&$2$&$3$&$4$&$5$&$6$&$7$  \\ \hline
$1$&1&$0$&0&5&5&0&5  \\ \hline
$2$&0&0&0&0&0&2&0  \\ \hline

$3$&0&0&3&0&0&0&3  \\ \hline
$4$&5&0&0&0&0&2&0 \\ \hline
$5$&5&0&0&0&0&0&0  \\ \hline

$6$&0&2&0&2&0&6&0 \\ \hline
$7$&5&0&3&0&0&0&3\\ \hline
\end{tabular}
\end{center}

\end{example}
\begin{example} $G1202$ is a zero-divisor graph.
\begin{center}G1202
\begin{tabular}{|l|c|r|r|r|r|r|r|}
\hline
$\bullet$&$1$&$2$&$3$&$4$&$5$&$6$&$7$  \\ \hline
$1$&3&$0$&3&3&2&0&0  \\ \hline
$2$&0&0&0&0&2&0&0  \\ \hline

$3$&3&0&3&3&0&0&0  \\ \hline
$4$&3&0&3&3&0&0&0 \\ \hline
$5$&2&2&0&0&5&0&0  \\ \hline

$6$&0&0&0&0&0&6&6 \\ \hline
$7$&0&0&0&0&0&6&6\\ \hline
\end{tabular}
\end{center}

\end{example}

\begin{example} $G1203$ is a zero-divisor graph.
\begin{center}G1203
\begin{tabular}{|l|c|r|r|r|r|r|r|}
\hline
$\bullet$&$1$&$2$&$3$&$4$&$5$&$6$&$7$  \\ \hline
$1$&1&$0$&0&0&1&0&0  \\ \hline
$2$&0&2&0&0&0&6&6  \\ \hline

$3$&0&0&3&3&3&0&0  \\ \hline
$4$&0&0&3&3&3&0&0 \\ \hline
$5$&1&0&3&3&5&0&0  \\ \hline

$6$&0&6&0&0&0&0&0 \\ \hline
$7$&0&6&0&0&0&0&0\\ \hline
\end{tabular}
\end{center}

\end{example}
\begin{example} $G1205$ is a zero-divisor graph.
\begin{center}G1205
\begin{tabular}{|l|c|r|r|r|r|r|r|}
\hline
$\bullet$&$1$&$2$&$3$&$4$&$5$&$6$&$7$  \\ \hline
$1$&1&$0$&0&0&0&6&6  \\ \hline
$2$&0&2&0&0&5&0&5  \\ \hline

$3$&0&0&3&4&0&0&4  \\ \hline
$4$&0&0&4&0&0&0&0 \\ \hline
$5$&0&5&0&0&0&0&0  \\ \hline

$6$&6&0&0&0&0&0&0 \\ \hline
$7$&6&5&4&0&0&0&0\\ \hline
\end{tabular}
\end{center}

\end{example}
\begin{example}
$G1206$ is a zero-divisor graph.
\begin{center}G1206
\begin{tabular}{|l|c|r|r|r|r|r|r|}
\hline
$\bullet$&$1$&$2$&$3$&$4$&$5$&$6$&$7$  \\ \hline
$1$&1&$0$&0&0&5&5&5  \\ \hline
$2$&0&2&2&2&0&0&0  \\ \hline

$3$&0&2&2&2&0&0&0  \\ \hline
$4$&0&2&2&2&0&0&0 \\ \hline
$5$&5&0&0&0&0&0&0  \\ \hline

$6$&5&0&0&0&0&0&0 \\ \hline
$7$&5&0&0&0&0&0&0\\ \hline
\end{tabular}
\end{center}

\end{example}
\begin{example} $G1207-1208$ are zero-divisor graphs because they are the refinement of a star graph with seven vertices.
\end{example}
\begin{example} $G1210$ is a zero-divisor graph.

\begin{center}G1210
\begin{tabular}{|l|c|r|r|r|r|r|r|}
\hline
$\bullet$&$1$&$2$&$3$&$4$&$5$&$6$&$7$  \\ \hline
$1$&1&$1$&0&0&0&0&1  \\ \hline
$2$&1&1&0&0&0&0&1  \\ \hline

$3$&0&0&3&6&0&6&0  \\ \hline
$4$&0&0&6&5&5&0&0 \\ \hline
$5$&0&0&0&5&5&0&0  \\ \hline

$6$&0&0&6&0&0&0&0 \\ \hline
$7$&1&1&0&0&0&0&1\\ \hline
\end{tabular}
\end{center}

\end{example}
\begin{example} $G1213-G1224$ are zero-divisor graphs because they are the refinement of a star graph with seven vertices.
\end{example}
\begin{example} $G1225$ is a zero-divisor graph.

\begin{center}G1225
\begin{tabular}{|l|c|r|r|r|r|r|r|}
\hline
$\bullet$&$1$&$2$&$3$&$4$&$5$&$6$&$7$  \\ \hline
$1$&1&$0$&0&0&6&6&6  \\ \hline
$2$&0&2&0&0&2&0&0  \\ \hline

$3$&0&0&3&3&0&0&0  \\ \hline
$4$&0&0&3&3&0&0&0 \\ \hline
$5$&6&2&0&0&2&0&0  \\ \hline

$6$&6&0&0&0&0&0&0 \\ \hline
$7$&6&0&0&0&0&0&0\\ \hline
\end{tabular}
\end{center}

\end{example}
\begin{example}
$G1226$ is a zero-divisor graph.
\begin{center}G1226
\begin{tabular}{|l|c|r|r|r|r|r|r|}
\hline
$\bullet$&$1$&$2$&$3$&$4$&$5$&$6$&$7$  \\ \hline
$1$&0&$0$&0&1&0&0&0  \\ \hline
$2$&0&0&0&1&0&0&0  \\ \hline

$3$&0&0&3&0&5&5&5  \\ \hline
$4$&1&1&0&4&0&0&0 \\ \hline
$5$&0&0&5&0&0&0&0  \\ \hline

$6$&0&0&5&0&0&0&0 \\ \hline
$7$&0&0&5&0&0&0&0\\ \hline
\end{tabular}
\end{center}

\end{example}
\begin{example}$G1227-G1230$ are zero-divisor graphs because they are the refinement of star graph with seven vertices.
\end{example}
\begin{example}
$G1231$ is a zero-divisor graph.
\begin{center}G1231
\begin{tabular}{|l|c|r|r|r|r|r|r|}
\hline
$\bullet$&$1$&$2$&$3$&$4$&$5$&$6$&$7$  \\ \hline
$1$&1&$0$&0&0&7&0&7  \\ \hline
$2$&0&2&0&0&6&6&0  \\ \hline

$3$&0&0&3&3&0&0&0  \\ \hline
$4$&0&0&3&3&0&0&0 \\ \hline
$5$&7&6&0&0&0&0&0  \\ \hline

$6$&0&6&0&0&0&0&0 \\ \hline
$7$&7&0&0&0&0&0&0\\ \hline
\end{tabular}
\end{center}

\end{example}
\begin{example}
$G1232$ is a zero-divisor graph.
\begin{center}G1232
\begin{tabular}{|l|c|r|r|r|r|r|r|}
\hline
$\bullet$&$1$&$2$&$3$&$4$&$5$&$6$&$7$  \\ \hline
$1$&0&$0$&0&0&1&0&0  \\ \hline
$2$&0&2&0&2&0&0&0  \\ \hline

$3$&0&0&3&0&0&6&6  \\ \hline
$4$&0&2&0&2&1&0&0 \\ \hline
$5$&1&0&0&1&5&0&0  \\ \hline

$6$&0&0&6&0&0&0&0 \\ \hline
$7$&0&0&6&0&0&0&0\\ \hline
\end{tabular}
\end{center}

\end{example}
\begin{example}
$G1233$ is a zero-divisor graph.
\begin{center}G1233
\begin{tabular}{|l|c|r|r|r|r|r|r|}
\hline
$\bullet$&$1$&$2$&$3$&$4$&$5$&$6$&$7$  \\ \hline
$1$&1&$0$&0&0&0&0&1  \\ \hline
$2$&0&2&0&0&0&2&0  \\ \hline

$3$&0&0&3&3&3&0&0  \\ \hline
$4$&0&0&3&3&3&0&0 \\ \hline
$5$&0&0&3&3&3&0&0  \\ \hline

$6$&0&2&0&0&0&2&0 \\ \hline
$7$&1&0&0&0&0&0&1\\ \hline
\end{tabular}
\end{center}

\end{example}
\begin{example} $G1234-G1242$ are zero-divisor graphs because they are the refinement of a star graph with seven vertices.
\end{example}
\begin{example} $G1243$ is a zero-divisor graph.

\begin{center}G1243
\begin{tabular}{|l|c|r|r|r|r|r|r|}
\hline
$\bullet$&$1$&$2$&$3$&$4$&$5$&$6$&$7$  \\ \hline
$1$&1&$0$&0&0&0&0&1  \\ \hline
$2$&0&2&0&0&0&2&0  \\ \hline

$3$&0&0&0&0&3&0&0  \\ \hline
$4$&0&0&0&0&3&0&0 \\ \hline
$5$&0&0&3&3&5&0&0  \\ \hline

$6$&0&2&0&0&0&2&0 \\ \hline
$7$&1&0&0&0&0&0&1\\ \hline
\end{tabular}
\end{center}

\end{example}
\begin{example} $G1244-G1252$ are zero-divisor graphs because they are the refinement of a star graph with seven vertices.
\end{example}

 \section {non zero-divisor graph with seven vertices that satisfies the given condition $\star$}\label{section4}
 
\subsection{a non-zero divisor graph with seven vertices produced by emanating an end from the highest degree vertex of a zero-divisor graph with six vertices produces.}

By Lemma~\ref{highest}, we know that emanating an end from the highest degree vertex of a zero-divisor graph with six vertices always produces a connected graph satisfies the condition $\star$.

\begin{theorem}\label{B} 
 
\begin{enumerate}
\item \label{graph22}$G184:$ $N(a)=\{c,x,y,z\},\, N(b)=\{x,y,z\}, N(c)=\{a,y\}, N(x)=\{a,b,y,z\},N(y)=\{a,b,c,x\}, N(z)=\{a,b,x\}$
\item \label{graph23}$G194:$ $N(a)=N(b)=\{c,x,y,z\},\,  N(c)=\{a,b,x,y\}, N(x)=\{a,b,c,y\},N(y)=\{a,b,c,x\},\, N(z)=\{a,b\}$
\item \label{graph24}$G204:$ $N(a)=N(b)=\{c,x,y,z\},\,  N(c)=N(z)=\{a,b,x,y\}, N(x)=\{a,b,c,z\},N(y)=\{a,b,c,z\}$ 
\end{enumerate}
are zero-divisor graphs which were proved in 24 exceptional cases in \cite{sauer}.  By adding a vertex $w$ and an edge $xw$, the resulting graph is not a zero-divisor graph.
\end{theorem}
\begin{proof}
The proof is given in the following Examples~\ref{eG766},\ref{eG893}, and
\ref{eG1024}.
\end{proof}

\subsection{Emanating an end from the highest degree vertex of a graph with six vertices which satisfies the condition $\star$  but is not a zero-divisor graph}
\begin{theorem}\cite{DGSW}\label{DGSW}Up to isomorphism, there are four connected graphs with six vertices that satisfy condition $\star$ but are not the zero-divisor graphs. They are
 \begin{enumerate}
 \item \label{type1}$G98:$ $N(a)=\{b,d\},$ $N(b)=\{a,c,e\},$ $N(c)=\{b\}$,$N(d)=\{a,e\},$ $N(e)=\{b,d,f\},$ $N(f)=\{e\}.$
 \item \label{type2}$G145:$ $N(a)=\{b,d\},$ $N(b)=\{a,c,e,f\},$ $N(c)=N(f)=\{b,e\},$ $N(d)=\{a,e\},$ $N(e)=\{b,c,d,f\}.$
 \item \label{type3}$G163:$ $N(a)=\{b,f\},$ $N(b)=\{a,c,d,f\}$,$ N(c)=\{b,d\},$ $N(d)=\{b,c,e,f\},$ $N(e)=\{d,f\},$$N(f)=\{a,b,d,e\}.$
 \item \label{type4}$G181:$ $N(a)=\{b,f\},$ $N(b)=\{a,c,e,f\},$ $N(c)=\{b,d,e,f\},$ $N(d)=\{c,e\},$ $N(e)=\{b,c,d,f\},$ $N(f)=\{a,b,c,e\}.$
 \end{enumerate}
 \end{theorem}
 \begin{theorem} By adding a vertex and an edge to the vertex with maximum degree of the above graphs in Theorem~\ref{DGSW}, we will not get any zero-divisor graphs.
 \end{theorem}
 \begin{proof}The proof will be given in Examples~\ref{eG322}, \ref{eG504},\ref{eG617}, and \ref{eG750}.
 \end{proof}
\begin{lemma}\label{emanating} Let $G$ be a connected graph. Let $x\in V(G)$ be a vertex with the highest degree.  Suppose there are no ends emanating from $x$ in $G$.  Let $H$ be a graph defined by emanating an end from vertex $x.$  If $G$ is not a zero divisor graph but satisfies the necessary condition $\star$, then $H$ is not a zero-divisor graph.
\end{lemma}
\begin{proof}
Notice $N(w)=\{x\}.$ Suppose $a$ and $b$ are two dis-connected vertices in $G$.  If one of $a$ or $b$ is not an end in $G$, then $|N(a)\cup N(b)|\geq 2.$ It follows that $N(a)\cup N(b)$ is not a subset of $\overline{N(w)}.$ Hence $ab\ne w.$ If $\deg(a)=\deg(b)=1$, then $a$ and $b$ are the ends in $G$.   By the definition of $G$, $a$ and $b$ do not emanates from $x$.  It follows that $N(a)\cup N(b)$ is not a subset of $\overline{N(w)}.$  Hence $ab\ne w.$

Suppose $H$ is a zero-divisor graph. Then $G$ must be a zero-divisor graph.  This is a contradiction!
\end{proof}

\begin{lemma} Let $G$ be a connected graph which does not satisfies the necessary condition $\star$ of a zero-divisor graph.   Let $x$ be a vertex of $G$ with highest degree.  Let $H$ be a graph defined by adding a new vertex $w$ and an new edge $xw$ to $G$.  Then $H$ is not a zero-divisor graph.
\end{lemma}
\begin{proof} Suppose $a$ and $b$ be two dis-connected vertices in $G$ such that $N(a)\cup N(b)$ is not subset of any $\overline{N(y)}$ in $G$.  Then it is easy to see that $N(a)\cup N(b)$ is not subset of any $\overline{N(y)}$ in $H$.

Hence $H$ is not a zero-divisor graph.
\end{proof}
\subsection{Non zero-divisor graphs with seven vertices that satisfy the condition $\star$  produced by applying  \cite[Lemma 3.14]{DGSW} to any graphs  in Theorem~\ref{DGSW}}
\begin{theorem}
One can not produce any zero-divisor graphs if we apply  \cite[Lemma 3.14]{DGSW} to any graphs in Theorem~\ref{DGSW}.
\end{theorem}
\begin{proof}
The proof will be given in Examples~\ref{eG322},\ref{eG405},\ref{eG635},\ref{eG669}, \ref{eG677},\ref{eG793}, \ref{eG918}, \ref{eG1044}, \ref{eG1060}, and\ref{eG1130}.
\end{proof}

\begin{lemma}If $G$ is graph that violates the necessary condition $\star$ of a zero-divisor graph, then a graph $H$ formed by the method of  \cite[Lemma 3.14]{DGSW} still violates the necessary condition $\star$  of a zero-divisor graph.
\end{lemma}
\begin{proof}
Suppose $N_{G}(a)\cup N_{G}(b)$ is not subset of any $\overline{N_{G}(x)}$ in old graph. Suppose there exists a vertex $y\notin V(H)-V(G)$ such that $N_{H}(a)\cup N_{H}(b)\subset \overline{N_{H}(y)}$. Notice $N_{H}(a)\cap V(G)=N_{G}(a),$ $N_{H}(b)\cap V(G)=N_{G}(b),$ $N_{H}(y)\cap V(G)=N_{G}(y).$ We get a contradiction. If $y\in V(H)-V(G),$ notice $N_H(y)\cap V(G)=N(x)$. This again is a contradiction. 
\end{proof}


\subsection{satisfying $\star$ condition but not zero-divisor graph}
\begin{theorem} The following are all the connected graphs in \cite{RW} which satisfy the $\star$ condition but are not zero-divisor graphs.
\begin{center}
 $G322$, $G383$,$G405$,$G475$,$G482$,$G490$,$G504$,$G600$,$G602$,$G607$, $G617$,$G627$,$G635$,$G669$,$G677$,$G742$,$G750$,$G754$, $G766$, $G772$,$G793$,$G799$, $G803$,$G808$, $G893$, $G899$, $G907$, $G917-G918$, $G928$, $G933$, $G938$, $G953$, $G1024$, $G1030$, $G1034$, $G1043-G1044$, $G1060$, $G1083$, $G1120$,
$G1130$, $G1146$, $G1177$,
\end{center}
\end{theorem}
\begin{proof} The proof is given in the following examples in this section.
\end{proof}
\begin{example}\label{eG322} The connected graph $G322$  with $V(G)=\{a,b, c,d, e,f,g\}$ and $E(G)$ defined by $N(a)=\{b,e\},\, N(b)=\{a,c, d,f\}, N(c)=\{b\}, N(d)=\{b\},N(e)=\{a,f\}, N(f)=\{b,e,g\}, N(g)=\{f\}$ is not a zero-divisor graph. If we remove $d,$ we get graph~(\ref{type1}) in Theorem~\ref{DGSW}.
We have the following multiplication table.
\begin{center}
\begin{tabular}{|l|c|r|r|r|r|r|r|}
\hline
$\bullet$ &$a$&$b$&$c$&$d$&$e$&$f$&$g$  \\ \hline
$a$&&0&&&0&$f$&$f$  \\ \hline

$b$&0&0&0&0&$b$&0&$b$  \\ \hline
c&&0&&&$b$&$f$&  \\ \hline

$d$&&0&&&$b$&$f$&  \\ \hline
$e$&0&$b$&$b$&$b$&&0&  \\ \hline
$f$&$f$&0&$f$&$f$&0&0&0\\  \hline
$g$&$f$&$b$&&&&0&  \\ \hline

\end{tabular}
\end{center}

By Light's Associativity Test, we have the following $a-$table.

\begin{center}
\begin{tabular}{|l|c|r|r|r|r|r|r|}
\hline
$a$&$a^2$&0&$ac$&$ad$&0&$f$&$f$  \\ \hline
$a^2$&&0&&0&0&$f$&$f$  \\ \hline

$0$&0&0&0&0&$0$&0&$0$  \\ \hline
$ac$&&0&&0&0&$f$&$f$  \\ \hline
$ad$&&0&&0&0&$f$&$f$   \\ \hline

$0$&0&$0$&$0$&$0$&0&0&0  \\ \hline
$f$&$f$&0&$f$&$f$&0&0&0\\  \hline
$f$&$f$&0&$f$&$f$&0&0&0\\  \hline

\end{tabular}
\end{center}
By checking each spectrum of vertices of $G,$ we get $a^2=a, \, ac=a,\, ad=a.$ Hence we have the following multiplication Table.
\begin{center}
\begin{tabular}{|l|c|r|r|r|r|r|r|}
\hline
$\bullet$&$a$&$b$&$c$&$d$&$e$&$f$&$g$  \\ \hline
$a$&$a$&0&$a$&$a$&0&$f$&$f$  \\ \hline

$b$&0&0&0&0&$b$&0&$b$  \\ \hline
c&$a$&0&&&$b$&$f$&  \\ \hline

$d$&$a$&0&&&$b$&$f$&  \\ \hline
$e$&0&$b$&$b$&$b$&$e$&0&$e$  \\ \hline
$f$&$f$&0&$f$&$f$&0&0&0\\  \hline
$g$&$f$&$b$&&&$e$&0&  \\ \hline

\end{tabular}
\end{center}
By Light's Associativity Test, we have the following $d-$table.
\begin{center}
\begin{tabular}{|l|c|r|r|r|r|r|r|}
\hline
$d$&$a$&0&$cd$&$d^2$&$b$&$f$&$dg$  \\ \hline
$a$&$a$&0&$a$&$a$&0&$f$&$f$  \\ \hline

$0$&0&0&0&0&$0$&0&$0$  \\ \hline
$cd$&$a$&0&&&0&$f$& \\ \hline
$d^2$&$a$&0&&&0&$f$&   \\ \hline

$b$&0&$0$&$0$&$0$&$b$&0&$b$  \\ \hline
$f$&$f$&0&$f$&$f$&0&0&0\\  \hline
$dg$&$f$&0&&&$b$&0&\\  \hline

\end{tabular}
\end{center}
Since $dg\in\{b,f\},$  we get a contradiction.

Second proof:Notice this is a complete bipartite graph together with three ends emanating from two distinct vertices $c$ and $x$. By \cite[Theorem IV.5]{sauer}, we conclude this graph is not a zero-divisor graph.

\end{example}
\begin{example} $G383$ satisfies the necessary condition $\star$ but is not a zero-divisor graph.
 $G383$ can be defined by $N(1)=\{3,4\},$ $N(2)=\{3\},$ $N(3)=\{1,2,4,6\},$ $N(4)=\{1,3,5,6\},$ $N(5)=\{4\},$ $N(6)=\{3,4,7\},$ $N(7)=\{6\}.$
 
 It folllows that \[N(2)\cup N(4)=N(4)\cup N(7)=\{1,3,5,6\}=N(4)\] and \[N(3)\cup N(5)=N(3)\cup N(7)=\{1,2,4,6\}=N(3)\]
 
 \begin{center}
\begin{tabular}{|l|c|r|r|r|r|r|r|}
\hline
$\bullet$&$1$&$2$&$3$&$4$&$5$&$6$&$7$  \\ \hline
$1$&&&$0$&$0$&&&  \\ \hline

$2$&&&$0$&$4$&&&  \\ \hline
$3$&0&$0$&&$0$&$3$&0&3  \\ \hline
$4$&$0$&$4$&$0$&&$0$&0&$4$  \\ \hline
$5$&&&$3$&$0$&&&  \\ \hline
$6$&&&0&0&&&$0$  \\ \hline
$7$&&&3&$4$&&$0$&  \\ \hline

\end{tabular}
\end{center}
 By checking $7-$ table, we get $5\bullet 7=3.$ It follows that $1\bullet 7=4.$ But then we can't define $2\bullet 7.$
 \end{example}

 \begin{example}\label{eG405}The connected graph $G405$  with $V(G)=\{a,b, c,d, e,f,g\}$ and $E(G)$ defined by $N(a)=\{c,d,e\},\, N(b)=\{c,d,e,g\}, N(c)=\{a,b\}, N(d)=\{a,b,f\},N(e)=\{a,b\}, N(f)=\{d\}, N(g)=\{b\}$ is not a zero-divisor graph. If we remove $e,$ we get graph~(\ref{type1}) in Theorem~\ref{DGSW}.

We have the following multiplication Table.
\begin{center}
\begin{tabular}{|l|c|r|r|r|r|r|r|}
\hline
$\bullet$&$a$&$b$&$c$&$d$&$e$&$f$&$g$  \\ \hline
$a$&&&$0$&0&$0$&&  \\ \hline

$b$&&&$0$&$0$&$0$&&0  \\ \hline
$c$&0&$0$&&&&&  \\ \hline
$d$&$0$&$0$&&&&$0$&  \\ \hline
$e$&$0$&0&&&&&  \\ \hline
$f$&&&&0&&&  \\ \hline
$g$&&0&&&&&  \\ \hline

\end{tabular}
\end{center}
Since \[N(a)\cup N(b)=\{c,d,e,g\}=N(b),\, N(a)\cup N(g)\subset\overline{N(b)}\]
\[N(b)\cup N(f)=N(b),\, N(c)\cup N(d)=N(d),\, N(c)\cup N(f)\subset \overline{N(d)}\]
\[N(d)\cup N(e)=N(d),\,N(e)\cup N(f)\subset \overline{N(d)},\,N(d)\cup N(g)=N(d)\]

we get the following updated Multiplication-table.
\begin{center}
\begin{tabular}{|l|c|r|r|r|r|r|r|}
\hline
$\bullet$&$a$&$b$&$c$&$d$&$e$&$f$&$g$  \\ \hline
$a$&&$b$&$0$&0&$0$&&$b$  \\ \hline

$b$&$b$&0&$0$&$0$&$0$&$b$&0  \\ \hline
$c$&0&$0$&&$d$&&$d$&  \\ \hline
$d$&$0$&$0$&$d$&0&$d$&$0$&$d$  \\ \hline
$e$&$0$&0&&$d$&&$d$&  \\ \hline
$f$&&$b$&$d$&0&$d$&&  \\ \hline
$g$&$b$&0&&$d$&&&  \\ \hline

\end{tabular}
\end{center}
Let us check the $a-$table.
\begin{center}
\begin{tabular}{|l|c|r|r|r|r|r|r|}
\hline
$a$&$a^2$&$b$&$0$&$0$&$0$&$af$&$b$  \\ \hline
$a^2$&&$b$&$0$&0&$0$&&$b$  \\ \hline

$b$&$b$&0&$0$&$0$&0&$b$&0  \\ \hline
$0$&0&$0$&$0$&$0$&$0$&$0$&$0$  \\ \hline
$0$&0&$0$&$0$&$0$&$0$&$0$&$0$  \\ \hline
$0$&0&$0$&$0$&$0$&$0$&$0$&$0$  \\ \hline

$af$&&$b$&$0$&$0$&&&$b$ \\ \hline
$b$&$b$&0&$0$&$0$&0&$b$&0  \\ \hline

\end{tabular}
\end{center}
we get $a^2=a,\, af=a.$ We have the following updated multiplication table.
\begin{center}
\begin{tabular}{|l|c|r|r|r|r|r|r|}
\hline
$\bullet$&$a$&$b$&$c$&$d$&$e$&$f$&$g$  \\ \hline
$a$&$a$&$b$&$0$&0&$0$&$a$&$b$  \\ \hline

$b$&$b$&0&$0$&$0$&$0$&$b$&0  \\ \hline
$c$&0&$0$&&$d$&&$d$&  \\ \hline
$d$&$0$&$0$&$d$&0&$d$&$0$&$d$  \\ \hline
$e$&$0$&0&&$d$&&$d$&  \\ \hline
$f$&$a$&$b$&$d$&0&$d$&&  \\ \hline
$g$&$b$&0&&$d$&&&  \\ \hline

\end{tabular}
\end{center}
Let's check the $f-$table.
\begin{center}
\begin{tabular}{|l|c|r|r|r|r|r|r|}
\hline
$f$&$a$&$b$&$d$&$0$&$d$&$f^2$&$fg$  \\ \hline
$a$&$a$&$b$&$0$&0&$0$&$a$&$b$  \\ \hline

$b$&$b$&0&$0$&$0$&0&$b$&0  \\ \hline
$d$&$0$&$0$&$d$&0&$d$&$0$&$d$  \\ \hline

$0$&0&$0$&$0$&$0$&$0$&$0$&$0$  \\ \hline
$d$&$0$&$0$&$d$&0&$d$&$0$&$d$  \\ \hline

$f^2$&$a$&$b$&$0$&$0$&$0$&&  \\ \hline
$fg$&$b$&$0$&$d$&$0$&$d$&&  \\ \hline

\end{tabular}
\end{center}
Then we can not define $fg$ since in the Multiplication-table, no spectrum of $a,b,c,d,e,f,g$  match the spectrum of $fg$ in the  $f-$table.

Second proof:$G405$ is not a zero-divisor graph since $G405$ is a complete bipartite graph together with ends emanating from two distinct the  vertices \cite[Theorem IV.5]{sauer}.
\end{example}
\begin{example}\label{Sauer Theorem IV3, page 26}The connected graph $G475$  with $V(G)=\{a,b, c,d, e,f,g\}$ and $E(G)$ defined by $N(a)=\{b,c,d,e\},\, N(b)=\{a,c,d\}, N(c)=\{a,b,d,g\}, N(d)=\{a,b,c,f\},N(e)=\{a\}, N(f)=\{d\}, N(g)=\{c\}$ is not a zero-divisor graph. If we remove $f$ from the graph, we get a zero-divisor graph with six vertices. This is a complete graph on four vertices together with three ends emanating from three different vertices.  By \cite[Theorem IV.3]{sauer}, we have that $G$ is not a zero-divisor graph.
\end{example}
 \begin{example}
 $G482$ is not a zero-divisor graph although it satisfies the necessary condition $\star.$
 $G482$ can be defined by $N(1)=\{2,3,6,7\},$ $N(2)=\{1,3\},$ $N(3)=\{1,2,4,6\},$  $N(4)=\{3\},$ $N(5)=\{6\},$ $N(6)=\{1,3,5,7\},$ $N(7)=\{1,6\}.$  It is easy to check that $G482$ satisfies the necessary condition $\star.$
 
 We have the following,
 \[N(1)\cup N(4)=N(1)\cup N(5)=\{2,3,6,7\}=N(1)\]
\[N(3)\cup N(5)=N(3)\cup N(7)=\{1,2,4,6\}=N(3)\]
\[N(6)\cup N(2)=N(6)\cup N(4)=\{1,3,5,7\}=N(6)\]
 
 \begin{center}
\begin{tabular}{|l|c|r|r|r|r|r|r|}
\hline
$\bullet$&$1$&$2$&$3$&$4$&$5$&$6$&$7$  \\ \hline
$1$&&0&$0$&$1$&1&0&0  \\ \hline

$2$&0&&$0$&&&6&  \\ \hline
$3$&0&$0$&&$0$&$3$&0&3  \\ \hline
$4$&$1$&&$0$&&&6&  \\ \hline
$5$&1&&$3$&&&0&  \\ \hline
$6$&0&6&0&6&0&&$0$  \\ \hline
$7$&0&&3&&&$0$&  \\ \hline

\end{tabular}
\end{center}

Checking the $4-$ table, we get $4\bullet 4=4$ and $4\bullet 5=1.$ Hence $1\bullet 1=1.$ It follows that $4\bullet 7=6$ and $6\bullet 6=0.$ Hence $2\bullet 4=2.$ But then we get $2\bullet 5=0.$ This is a contradiction.
 \end{example}
 \begin{example}
 $G490$ is not a zero-divisor graph although it satisfies the necessary condition $\star.$
 $G490$ can be defined by $N(1)=\{2\},$ $N(2)=\{1,3,5,7\},$ $N(3)=\{2,4\},$  $N(4)=\{3,5,7\},$ $N(5)=\{2,4,6,7\},$ $N(6)=\{5\},$ $N(7)=\{2,4,5\}.$  It is easy to check that $G490$ satisfies the necessary condition $\star.$
 
 We have the following,
 \[N(2)\cup N(4)=N(2)\cup N(6)=\{1,3,5,7\}=N(2)\]
\[N(1)\cup N(5)=N(3)\cup N(5)=\{2,4,6,7\}=N(5)\]
\[N(1)\cup N(4)=\{2,3,5,7\}=\overline{N(2)}\]
 
 \begin{center}
\begin{tabular}{|l|c|r|r|r|r|r|r|}
\hline
$\bullet$&$1$&$2$&$3$&$4$&$5$&$6$&$7$  \\ \hline
$1$&&0&&$2$&5&&  \\ \hline

$2$&0&0&$0$&2&0&2&0  \\ \hline
$3$&&$0$&&$0$&$5$&&  \\ \hline
$4$&$2$&2&$0$&&0&&0  \\ \hline
$5$&5&0&$5$&0&&0&0  \\ \hline
$6$&&2&&&0&&  \\ \hline
$7$&&0&&0&0&&  \\ \hline

\end{tabular}
\end{center}
 Checking the $4-$ table, we get $4\bullet 4=4$ and $4\bullet 6=4.$  Checking the $6-$ table, we get $1\bullet 6=2.$  It follows that $N(6\bullet 7)\supseteq \{1,2,4,5\}.$  Hence $6\bullet 7$ is undefined.
 \end{example}

\begin{example} \label{eG504}Let $H=G504$ be graph  with seven vertices $a,b,c,x,y,z, w$.  Let $N(a)=N(b)=\{x,z\},$  $N(c)=\{x,y\},$  $N(x)=\{a,b,c,z,w\}$, $N(y)=\{c,z\},$ $N(z)=\{a,b,x,y\}.$ $N(w)=\{x\}.$ 

Notice $|N(x)-\{w\}|=4.$  If we remove $w$ from $H$, we get a  graph~(\ref{type2}) in Theorem~\ref{DGSW}  which is not a zero-divisor graph with six vertices. Hence by Lemma~\ref{emanating},$H$ is not a zero-divisor graph.
\end{example} 
\begin{example}\label{first graph, page 99}The connected graph $G600$  with $V(G)=\{a,b, c,x, y,z,w\}$ and $E(G)$ defined by $N(a)=\{x,y,z\},\, N(b)=\{y,z\}, N(c)=\{x\}, N(x)=\{a,c,y,z\},N(y)=\{a,b,x,z,w\}, N(z)=\{a,b,x,y\}, N(w)=\{y\}$ is not a zero-divisor graph. 
Since \[N(b)\cup N(x)=N(w)\cup N(x)=\{a,c,y,z\}=N(x),\] and \[N(c)\cup N(y)=N(y)\]
we have the following multiplication table.
\begin{center}
\begin{tabular}{|l|c|r|r|r|r|r|r|}
\hline
$\bullet$&$a$&$b$&$c$&$x$&$y$&$z$&$w$  \\ \hline
$a$&&&&0&0&0&  \\ \hline
$b$&&&&$x$&0&0&  \\ \hline

$c$&&&&0&$y$&&  \\ \hline
$x$&0&$x$&$0$&&0&0&$x$  \\ \hline
$y$&$0$&$0$&$y$&$0$&&$0$&0  \\ \hline
$z$&0&0&&0&$0$&& \\ \hline
$w$&&&$x$&$0$&&&\\ \hline
\end{tabular}
\end{center}

Claim: $cz=z.$

Since \[N(c)\cup N(z)=\{a,b, x,y\}\subset N(z)\cap \overline{N(y)}\]
we have $cz=y$ or $cz=z.$

Let's check the $z-$ table.
If $cz=y,$ then $N(zw)\supset \{a,b,c,x,y\}.$   It follows that we can not define $zw.$  Hence $cz=z.$
\vskip .5cm
Claim: $cw=z,$ $ac=y, bc=y, $ $y^2=0.$

Since \[N(c)\cup N(w)\subset\overline{N(a)}\cap \overline{N(x)}\cap \overline{N(y)}\cap \overline{N(z)}\]
we have $cw\in \{a,x,y,z\}.$

Let's check the $c-$ table. If $cw=a,$ then $N(cw)\supset \{x,y,z\}.$ Hence $N(z)\supset \{a,b,x,y,w\}.$ Contradiction!
If $cw=x,$ then  $N(cw)\supset \{a,c,y,z\}\cup \{x,y\}=\{a,c,x,y,z\}.$ Hence $N(z)\supset \{a,b,x,y,w\}.$ Contradiction!
If $cw=y,$ then $N(cw)\supset \{a,b,x,y,z,w\}$  Hence $N(z)\supset \{a,b,x,y,w\}.$  Contradiction!
Hence $cw=z.$  It follows that $N(ac)\supset \{x,y,z,w\}.$ Hence $ac=y,\, y^2=0.$  Similarly $bc=y.$  
\vskip .5cm
Let's Check $a-$ table. Notice $N(ac)\supset \{c,x,y,z\}$ and $N(aw)\supset\{c,x,y,z\}.$  It follows that  $ab=x,\, aw=x,$ and $a^2=0.$  
\vskip .5cm
Let's check the  $b-$ table. Notice $N(b^2)\supset \{c,y,z\}$ and $N(bw)\supset\{c,y,z\}.$ Hence $b^2=x$ and $bw=x.$  But $b^2a=x\ne 0=xa$ and $bwa=x\ne 0=xa.$  It follows that  $b^2$ and $bw$
are undefined.

This argument shows that the above graph is not a zero-divisor graph.
\end{example}
\begin{example}
$G602$ is not a zero-divisor graph.

$G602$ can be defined by $N(1)=\{2\},$ $N(2)=\{1,3,5,7\},$ $N(3)=\{2,4,5,7\},$ $N(4)=\{3\},$ $N(5)=\{2,3,6,7\},$ $N(6)=\{5,7\},$ $N(7)=\{2,3,5,6\}.$
It is easy to check that $G602$ satisfies the $\star$ condition.

We have the following table.
\begin{center}
\begin{tabular}{|l|c|r|r|r|r|r|r|}
\hline
$\bullet$&$1$&$2$&$3$&$4$&$5$&$6$&$7$  \\ \hline
$1$&&$0$&$3$&&\{5,7\}&&$\{5,7\}$  \\ \hline

$2$&$0$&&$0$&2&0&2&$0$  \\ \hline
$3$&3&$0$&&$0$&$0$&&0  \\ \hline
$4$&&$2$&$0$&&\{5,7\}&&\{5,7\}  \\ \hline
$5$&$\{5,7\}$&$0$&$0$&$\{5,7\}$&&$0$&0  \\ \hline
$6$&&2&&&$0$&&$0$ \\ \hline
$7$&$\{5,7\}$&$0$&$0$&$\{5,7\}$&$0$&$0$&\\ \hline
\end{tabular}
\end{center}
Checking the $4-$table, we get $4\times 6=2$ and $2\times 2=2.$  Then we have $N(1\times 4)\supseteq \{2,3,6\}.$  Hence $1\times 4=5$ or $1\times 4=7.$
But $1\times 4\ne 5$ since $(1\times 4)\times 7\ne 0.$  Similarly $1\time 4\ne 7$ since $(1\times 4)\times 5\ne 0.$
\end{example}
\begin{example}
$G607$ is not a zero-divisor graph. 

$G607$ can be defined by $N(1)=\{2\},$ $N(2)=\{1,3,5,6\},$ $N(3)=\{2,4,6\},$ $N(4)=\{3,5,6\},$ $N(5)=\{2,4,6,7\},$ $N(6)=\{2,3,4,5\},$ $N(7)=\{5\}.$ 
We have the following table.
\begin{center}
\begin{tabular}{|l|c|r|r|r|r|r|r|}
\hline
$\bullet$&$1$&$2$&$3$&$4$&$5$&$6$&$7$  \\ \hline
$1$&&$0$&&&5&6&$6$  \\ \hline

$2$&$0$&&$0$&2&0&0&$2$  \\ \hline
$3$&&$0$&&$0$&$5$&0&  \\ \hline
$4$&&$2$&$0$&&0&0&  \\ \hline
$5$&$5$&$0$&$5$&$0$&&$0$&0  \\ \hline
$6$&6&0&0&0&$0$&&$6$ \\ \hline
$7$&$6$&$2$&&&$0$&$6$&\\ \hline
\end{tabular}
\end{center}
Checking $1-$table, we get $6\times 6=6.$  It follows that $1\times 4=5$ and $5\times 5=0.$  But then we can't define $1\times 3.$
\end{example}
\begin{example}\label{eG617} Let $H=G617$ be graph  with seven vertices $a,b,c,x,y,z, w$.  Let $N(a)=\{x,y\},$ $N(b)=\{x,z\},$  $N(c)=\{y,z\},$  $N(x)=\{a,c,y,z,w\}$, $N(y)=\{a,b,x,z\},$ $N(z)=\{a,c,x,y\}.$ $N(w)=\{x\}.$ 

Notice $|N(x)-\{w\}|=4.$  If we remove $w$ from $H$, we get a  graph~(\ref{type3}) in Theorem~\ref{DGSW} that is not a zero-divisor graph which is proved in \cite{DGSW}. Hence by Lemma~\ref{emanating},$H$ is not a zero-divisor graph.
\end{example} 
\begin{example}
$G627$ is not a zero-divisor graph.  $G627$ can be defined by $N(1)=\{2\},$ $N(2)=\{1,3,5,7\},$ $N(3)=\{2,4,5\},$ $N(4)=\{3,5\},$ $N(5)=\{2,3,4,6,7\},$ $N(6)=\{5,7\},$ $N(7)=\{2,5,6\}.$ It is easy to check that $G627$ satisfies the $\star$ condition.

We have the following table.
\begin{center}
\begin{tabular}{|l|c|r|r|r|r|r|r|}
\hline
$\bullet$&$1$&$2$&$3$&$4$&$5$&$6$&$7$  \\ \hline
$1$&&$0$&\{3,5\}&\{2,3,5\}&5&\{2,5,7\}&\{5,7\}  \\ \hline

$2$&$0$&&$0$&2&0&2&$0$  \\ \hline
$3$&\{3,5\}&$0$&&$0$&$0$&5&5  \\ \hline
$4$&\{2,3,5\}&$2$&$0$&&0&\{2,5\}&5  \\ \hline
$5$&$5$&$0$&$0$&$0$&0&$0$&0  \\ \hline
$6$&\{2,5,7\}&2&5&\{2,5\}&$0$&&0 \\ \hline
$7$&\{5,7\}&$0$&5&5&$0$&$0$&\\ \hline
\end{tabular}
\end{center}
Checking $3-$table, we get $1\times 3=3.$
Checking  $6-$table, we get $4\times 6=2,$ $6\times 6=2,$ and $2\times 2=2.$ But then we can't define $1\times 6.$
\end{example}
\begin{example}\label{eG635}The connected graph $G635$  with $V(G)=\{a,b, c,d, e,f,g\}$ and $E(G)$ defined by $N(a)=\{b,d,g\},\, N(b)=\{a,c,e,g\}, N(c)=\{b,g\}, N(d)=\{a,e\},N(e)=\{b,d,f,g\}, N(f)=\{e\}, N(g)=\{a,b,c,e\}$ is not a zero-divisor graph. If we remove $g,$ we get graph~(\ref{type1}) in Theorem~\ref{DGSW}.

Notice \[N(a)\cup N(e)=N(c)\cup N(e)=\{b,d,f,g\}=N(e)\]
and \[N(a)\cup N(f)=\{b,d,e,g\}\subset \overline{N(e)}\]
we have the following multiplication table.
\begin{center}
\begin{tabular}{|l|c|r|r|r|r|r|r|}
\hline
$\bullet$&$a$&$b$&$c$&$d$&$e$&$f$&$g$  \\ \hline
$a$&&$0$&&$0$&$e$&$e$&0  \\ \hline

$b$&$0$&&$0$&&0&&0  \\ \hline
$c$&&$0$&&&$e$&&$0$  \\ \hline
$d$&$0$&&&&$0$&&  \\ \hline
$e$&$e$&0&$e$&$0$&0&0&0  \\ \hline
$f$&$e$&&&&0&&  \\ \hline
$g$&$0$&0&$0$&&0&&  \\ \hline

\end{tabular}
\end{center}
By checking $a$-table, we get $ac=a^2=a.$ It follows that by checking $c-$ table, we get $cf=e.$
Since $N(c)\cup N(d)=\{a,b,e,g\}$ is a subset of $\overline{N(b)}$ or a subset of $\overline{N(g)}$, we have $cd=b$ or $cd=g.$  

If $cd=g,$ then $fg=f(cd)=(fc)d=ed=0.$ This is a contradiction.

If $cd=b,$ then $fb=f(cd)=(fc)d=ed=0.$  Contradiction.

\end{example}
\begin{example}\label{eG669}The connected graph $G669$  with $V(G)=\{a,b, c,d, e,f,g\}$ and $E(G)$ defined by $N(a)=\{b,c,d,e,f\},\, N(b)=\{a, d,e,f,g\}, N(c)=\{a,g\}, N(d)=\{a,b\},N(e)=\{a,b\}, N(f)=\{a,b\}, N(g)=\{b,c\}$ is not a zero-divisor graph. If we remove $f,$ we get graph~(\ref{type2}) in Theorem~\ref{DGSW}.

We have the following multiplication table.
\begin{center}
\begin{tabular}{|l|c|r|r|r|r|r|r|}
\hline
$\bullet$&$a$&$b$&$c$&$d$&$e$&$f$&$g$  \\ \hline
$a$&&0&0&0&0&0&  \\ \hline

$b$&0&&&0&0&0&0  \\ \hline
c&0&&&&&&0  \\ \hline

$d$&0&0&&&&&  \\ \hline
$e$&0&0&&&&&  \\ \hline
$f$&0&0&&&&&\\  \hline
$g$&&0&0&&&&  \\ \hline

\end{tabular}
\end{center}
Since $N(c)\cup N(d)=\{a,b,g\}\subset \overline{N(b)}.$  Hence $b^2=0.$
Since $N(d)\cup N(g)=\{a,b,c\}\subset \overline{N(a)}.$  Hence $a^2=0$ and we have a updated multiplication Table.
\begin{center}
\begin{tabular}{|l|c|r|r|r|r|r|r|}
\hline
$\bullet$&$a$&$b$&$c$&$d$&$e$&$f$&$g$  \\ \hline
$a$&0&0&0&0&0&0&  \\ \hline

$b$&0&0&&0&0&0&0  \\ \hline
c&0&&&&&&0  \\ \hline

$d$&0&0&&&&&  \\ \hline
$e$&0&0&&&&&  \\ \hline
$f$&0&0&&&&&\\  \hline
$g$&&0&0&&&&  \\ \hline
\end{tabular}
\end{center}

Since $N(b)\cup N(c)=\{a, d,e, f,g\},$   we get
 $bc=b.$ We get the following updated multiplication Table.
\begin{center}
\begin{tabular}{|l|c|r|r|r|r|r|r|}
\hline
$\bullet$&$a$&$b$&$c$&$d$&$e$&$f$&$g$  \\ \hline
$a$&0&0&0&0&0&0&  \\ \hline

$b$&0&0&b&0&0&0&0  \\ \hline
c&0&b&&&&&0  \\ \hline

$d$&0&0&&&&&  \\ \hline
$e$&0&0&&&&&  \\ \hline
$f$&0&0&&&&&\\  \hline
$g$&&0&0&&&&  \\ \hline
\end{tabular}
\end{center}
Since $N(a)\cup N(g)=\{b,c,d,e,f\},$ we have
$ag=a.$  We have the following updated Multiplication Table.
\begin{center}
\begin{tabular}{|l|c|r|r|r|r|r|r|}
\hline
$\bullet$&$a$&$b$&$c$&$d$&$e$&$f$&$g$  \\ \hline
$a$&0&0&0&0&0&0&$a$  \\ \hline

$b$&0&0&$b$&0&0&0&0  \\ \hline
c&0&$b$&&&&&0  \\ \hline

$d$&0&0&&&&&  \\ \hline
$e$&0&0&&&&&  \\ \hline
$f$&0&0&&&&&\\  \hline
$g$&$a$&0&0&&&&  \\ \hline
\end{tabular}
\end{center}
Since \[N(c)\cup N(d)=N(c)\cup N(e)=N(c)\cup N(f)=\{a,b,g\}\subset\overline{N(b)}\] we have $cd=ce=cf=b$ and the following multiplication Table.
\begin{center}
\begin{tabular}{|l|c|r|r|r|r|r|r|}
\hline
$\bullet$&$a$&$b$&$c$&$d$&$e$&$f$&$g$  \\ \hline
$a$&0&0&0&0&0&0&$a$  \\ \hline

$b$&0&0&$b$&0&0&0&0  \\ \hline
c&0&$b$&&$b$&$b$&$b$&0  \\ \hline

$d$&0&0&$b$&&&&  \\ \hline
$e$&0&0&$b$&&&&  \\ \hline
$f$&0&0&$b$&&&&\\  \hline
$g$&$a$&0&0&&&&  \\ \hline
\end{tabular}
\end{center}
Since \[N(g)\cup N(d)=N(g)\cup N(e)=N(g)\cup N(f)=\{a,b,c\}\subset \overline{N(a)}\]  Hence $gd=eg=gf=a.$  We have the following multiplication Table.
\begin{center}
\begin{tabular}{|l|c|r|r|r|r|r|r|}
\hline
$\bullet$&$a$&$b$&$c$&$d$&$e$&$f$&$g$  \\ \hline
$a$&0&0&0&0&0&0&$a$  \\ \hline

$b$&0&0&$b$&0&0&0&0  \\ \hline
c&0&$b$&&$b$&$b$&$b$&0  \\ \hline

$d$&0&0&$b$&&&&$a$  \\ \hline
$e$&0&0&$b$&&&&$a$  \\ \hline
$f$&0&0&$b$&&&&$a$\\  \hline
$g$&$a$&0&0&$a$&$a$&$a$&  \\ \hline
\end{tabular}
\end{center}

By checking $g-$ table, we get $g^2=g.$  By checking $c-$ table, we get $c^2=c.$  By checking $d-$ table, we showed that there is no way to define $de$ or $df.$

Therefore $G$ is not a zero-divisor graph.
\end{example}
\begin{example}\label{eG677}The connected graph $G677$  with $V(G)=\{a,b, c,d, e,f,g\}$ and $E(G)$ defined by $N(a)=\{b,d\},\, N(b)=\{a, c,e,f,g\}, N(c)=N(f)=\{b,e\}, N(d)=\{a,e,g\},N(e)=\{b,c,d,f\}, N(g)=\{b,d\}$ is not a zero-divisor graph. If we remove $g,$ we get graph~(\ref{type2}) in Theorem~\ref{DGSW}.

Notice \[N(a)\cup N(c)=N(a)\cup N(f)=\{b,d,e\}\subset \overline{N(e)}\]
\[N(g)\cup N(c)=N(g)\cup N(f)=\{b,d,e\}\subset \overline{N(e)}\]
\[N(a)\cup N(e)=N(g)\cup N(e)=\{b,c,d,f\}=N(e)\]
\[N(b)\cup N(d)=\{a,c,e,f,g\}=N(b)\]
we get the following multiplication table.
\begin{center}
\begin{tabular}{|l|c|r|r|r|r|r|r|}
\hline
$\bullet$&$a$&$b$&$c$&$d$&$e$&$f$&$g$  \\ \hline
$a$&&0&$e$&0&$e$&$e$&$0$  \\ \hline

$b$&0&&0&$b$&0&0&0  \\ \hline
c&$e$&$0$&&&$0$&&$e$  \\ \hline

$d$&0&$b$&&&0&&$0$  \\ \hline
$e$&$e$&0&$0$&0&0&0&$e$  \\ \hline
$f$&$e$&0&&&0&&$e$\\  \hline
$g$&&0&$e$&$0$&$e$&$e$&  \\ \hline
\end{tabular}
\end{center}
By checking $c$-table, we get $cd=cf=b.$  It follows that $b=bd=fcd=fb=0.$ This is a contradiction.
\end{example}
\begin{example}\label{third graph, page 101}The connected graph $G742$  with $V(G)=\{a,b, c,x, y,z,w\}$ and $E(G)$ defined by $N(a)=\{c,x,y,z,w\},\, N(b)=\{x\}, N(c)=\{a,x,y,z\}, N(x)=\{a,b,c,z\},N(y)=\{a,c,z\}, N(z)=\{a,c,x,y\}, N(w)=\{a\}$ is not a zero-divisor graph. 
Since \[N(a)\cup N(b)=\{c,x,y,z,w\}=N(a)\]
and
\[N(x)\cup N(y)=N(x)\cup N(w)=\{a,b,c,z\}=N(x)\]
one gets  $ab=a$ , $xy=xw=x$ and the following multiplication table.
\begin{center}
\begin{tabular}{|l|c|r|r|r|r|r|r|}
\hline
$\bullet$&$a$&$b$&$c$&$x$&$y$&$z$&$w$  \\ \hline
$a$&&$a$&0&0&0&0&$a$  \\ \hline
$b$&$a$&&&0&&&  \\ \hline

$c$&0&&&$0$&$0$&$0$&  \\ \hline
$x$&0&$0$&$0$&&$x$&0&$x$  \\ \hline
$y$&$0$&&$0$&$x$&&$0$&  \\ \hline
$z$&0&&0&$0$&$0$&& \\ \hline
$w$&$0$&&&$x$&&&\\ \hline
\end{tabular}
\end{center}
Since \[N(b)\cup N(y)=\{a,c,x,z\},\] one gets that $by\in \{a,c,z,x\}.$ By checking $y-$ table, one gets $y^2\in \{x,y\}.$

If $y^2=y,$ by checking $y-$ table, one gets $by=x$ and $yw=y.$  By checking $b-$ table, one gets that $bz=a$ and $bw=x.$  Then one can't define $b^2.$

Suppose $y^2=x.$  By checking $y-$ table, we get $x^2=x.$  It follows that $N(by)\supset \overline {N(a)}$ or $N(by)\supset \overline{N(c)}.$  Hence one has $by\in \{a,c\}$ and $yw\in \{x,y\}.$ If $by=a$ and $yw=y,$ then $(wy)b=a\ne 0=w(yb).$ If $by=c$ and $yw=x,$ one gets $cw=byw=bx=0.$  It follows that  $by=a, \, yw=x$ or $by=c,\, yw=y.$ If $by=c,$ then $cw=c.$

Case I: $by=a$ and $yw=x.$

Suppose $by=a$ and $yw=x.$  By checking $b-$ table, one gets $b^2=b$ and $N(bw)\supset\{a,x,y\}.$  It follows that $bw\in \{a,c,z\}$
Suppose $by=a$, $yw=x$, and $bw=a.$  By checking $b-$ table, one gets $w^2=w$ and $N(cw)\supset \{a,b,x,y,w\}.$ Hence one can't define $cw.$
Suppose $by=a$, $yw=x$, and $bw=c.$  By checking $b-$ table, one gets $N(wz)\supset \{a,b,x,y\}.$  Hence one can't define $wz.$
Suppose $by=a$, $yw=x$, and $bw=z.$  By checking $b-$ table, one gets $N(cw)\supset \{a,b,x,y\}.$  Hence one can't define $cw.$

Case II:  $by=c$ and $yw=y.$

We get $cw=c$ and $w^2=w.$  By checking $w-$ table, one gets $N(wz)\supset\{a,c,x,y\}.$  Hence $wz\in \{a,c,z\}.$
If $wz=a,$ then $(zw)w=0\ne z(w^2).$  If $wz=c,$ then $z^2=0.$  By checking $w-$ table and the spectrum of each vertex of $G$, one can't define $bw$. Suppose $wz=z.$ By checking $w-$ table and the spectrum of each vertex of $G,$ then one can't define $bw$.

\end{example}
\begin{example}\label{eG750} Let $H=G750$ be graph  with seven vertices $a,b,c,x,y,z, w$.  Let $N(a)=\{b,x,y,z\},$ $N(b)=\{a,y\},$  $N(c)=\{x,z\},$  $N(x)=\{a,c,y,z,w\}$, $N(y)=\{a,b,x,z\},$ $N(z)=\{a,c,x,y\}.$ $N(w)=\{x\}.$ 

Notice $|N(x)-\{w\}|=4.$  If we remove $w$ from $H$, we get a graph~(\ref{type4}) in Theorem~\ref{DGSW} that is not a zero-divisor graph which is proved in \cite{DGSW}. Hence by Lemma~\ref{emanating},$H$ is not a zero-divisor graph.
\end{example} 
\begin{example}
$G754$ is not a zero-divisor graph.

$G754$ can be defined by $N(1)=\{2\},$ $N(2)=\{1,3,5,6,7\},$ $N(3)=\{2,4,6\},$ $N(4)=\{3,5,6\},$ $N(5)=\{2,4,6\},$ $N(6)=\{2,3,4,5,7\},$ $N(7)=\{2,6\}$.
We have the following table.

\begin{center}
\begin{tabular}{|l|c|r|r|r|r|r|r|}
\hline
$\bullet$&$1$&$2$&$3$&$4$&$5$&$6$&$7$  \\ \hline
$1$&&$0$&&\{2,6\}&&6&\{2,3,5,6,7\}  \\ \hline

$2$&$0$&&$0$&2&0&0&$0$  \\ \hline
$3$&&$0$&&0&&0&  \\ \hline
$4$&\{2,6\}&$2$&$0$&&0&0&\{2,6\} \\ \hline
$5$&&$0$&&$0$&&$0$&  \\ \hline
$6$&6&0&0&0&$0$&&0 \\ \hline
$7$&\{2,3,5,6,7\}&0&&\{2,6\}&&$0$&\\ \hline
\end{tabular}
\end{center}
Let $1\times 4=2.$ Then $2\times 2=0.$ Checking $1-$ table, we get $1\times 1=6. $ Then we can't define $1\times 7.$

Let $1\times 4=6.$ Then $6\times 6=0.$ Checking $4-$ table, we get $4\times 7=2.$  Then we can't define $4\times 4.$
\end{example}
\begin{example}\label{eG766}The connected graph $G766$  with $V(G)=\{a,b, c,x, y,z,w\}$ and $E(G)$ defined by $N(a)=\{c,x,y,z\},\, N(b)=\{x,y,z\}, N(c)=\{a,y\}, N(x)=\{a,b,y,z,w\},N(y)=\{a,b,c,x\}, N(z)=\{a,b,x\}, N(w)=\{x\}$ is not a zero-divisor graph. If we remove $w$ from the graph, we get the  graph~(\ref{graph22}) in Theorem~\ref{B} which is a zero-divisor graph with six vertices.
Since \[N(a)\cup N(b)=N(a)\cup N(w)=\{c,x,y,z\}=N(a)\]
\[N(y)\cup N(z)=N(y)\cup N(w)=\{a,b,c,x\}=N(y)\]
\[N(x)\cup N(c)=\{a,b,y,z,w\}=N(x)\]
we get $ab=aw=a,\, yz=yw=y$ and $xc=x.$
we have the following multiplication table.
\begin{center}
\begin{tabular}{|l|c|r|r|r|r|r|r|}
\hline
$\bullet$&$a$&$b$&$c$&$x$&$y$&$z$&$w$  \\ \hline
$a$&&$a$&0&0&0&0&$a$  \\ \hline
$b$&$a$&&&0&0&0&  \\ \hline

$c$&0&&&$x$&$0$&$y$&  \\ \hline
$x$&0&$0$&$x$&&0&0&$0$  \\ \hline
$y$&$0$&$0$&$0$&$0$&&$y$&$y$  \\ \hline
$z$&0&0&&$0$&$y$&& \\ \hline
$w$&$a$&&&$0$&$y$&&\\ \hline
\end{tabular}
\end{center}
Since \[N(b)\cup N(c)=\{a,x,y,z\}\]
and \[N(c)\cup N(z)=\{a,b,x,y\}\]
we get \[bc=\begin{cases} a \quad \text{if} \quad a^2=0\\
x \quad\text{if} \quad x^2=0\\
\end{cases}\]
and
\[cz=\begin{cases} x \quad \text{if} \quad x^2=0\\
y \quad\text{if} \quad y^2=0\\
\end{cases}\]

Case I: $cz=x$ and $bc=a$.  By checking $c-$table, we can't define $c^2.$

Case II: $cz=x$ and $bc=x$. By checking $c-$table, we get $cw=x.$  Then checking $w-$ table, we can't define $w^2.$

Case III: $cz=y$ and $bc=a$.  By checking $c-$ table, we get $c^2=x.$ It follows that we can't define $cw.$

Case IV: $cz=y$ and $bc=x.$ By checking $c-$ table, we can't define $c^2.$ 

This argument shows that the graph is not a zero-divisor graph.

\end{example}
\begin{example} $G772$ is not a zero-divisor graph.
$G772$ can be defined by $N(1)=\{2,4\},$ $N(2)=\{1,3,5,6\},$ $N(3)=\{2,4,5\},$ $N(4)=\{1,3,5,6\},$ $N(5)=\{2,3,4,6\},$ $N(6)=\{2,4,5,7\},$ $N(7)=\{6\}.$

We have the following table.
\begin{center}
\begin{tabular}{|l|c|r|r|r|r|r|r|}
\hline
$\bullet$&$1$&$2$&$3$&$4$&$5$&$6$&$7$  \\ \hline
$1$&&$0$&&0&5&6& \{5,6\} \\ \hline

$2$&$0$&&$0$&\{2,4\}&0&0&$\{2,4\}$  \\ \hline
$3$&&$0$&&0&0&6&\{5,6\}  \\ \hline
$4$&0&$\{2,4\}$&$0$&&0&0&\{2,4\} \\ \hline
$5$&&$0$&$0$&$0$&&$0$&5  \\ \hline
$6$&6&0&6&0&$0$&&0 \\ \hline
$7$&\{5,6\}&$\{2,4\}$&\{5,6\}&\{2,4\}&$5$&$0$&\\ \hline
\end{tabular}
\end{center}
Checking $7-$table, we get $7\times 7=7.$  Let $3\times 7=5.$ Then we can't define $1\times 7.$ Let $3\times 7=6.$ Then we get $3\times 7=0.$ This is a contradiction.
\end{example}
\begin{example}\label{eG793} The connected graph $G793$ which is defined by  $N(a)=\{b,f\},\, N(b)=\{a,c,d,f,g\},$ $N(c)=\{b,d\},$ $N(d)=\{b,c,e,f\},N(e)=\{d,f\}, N(f)=\{a,b,d,e,g\}, N(g)=\{b,f\}$ is not a zero-divisor graph. If we remove $g,$ we get graph~(\ref{type3}) in Theorem~\ref{DGSW}.
Notice \[N(a)\cup N(d)=N(g)\cup N(d)=\{b,c,e,f\}=N(d)\]
\[N(b)\cup N(e)=\{a,c,d,f,g\}=N(b)\]
\[N(c)\cup N(f)=\{a,b,d,e,g\}=N(f)\]
We have the following multiplication table.
\begin{center}
\begin{tabular}{|l|c|r|r|r|r|r|r|}
\hline
$\bullet$&$a$&$b$&$c$&$d$&$e$&$f$&$g$  \\ \hline
$a$&&$0$&&$d$&&0&  \\ \hline

$b$&$0$&&$0$&$0$&$b$&0&$0$  \\ \hline
$c$&&$0$&&$0$&&$f$&  \\ \hline
$d$&$d$&$0$&$0$&&$0$&0&$d$  \\ \hline
$e$&&$b$&&$0$&&$0$&  \\ \hline
$f$&$0$&0&$f$&0&$0$&&$0$  \\ \hline
$g$&&$0$&&$d$&&$0$&  \\ \hline

\end{tabular}
\end{center}
Notice \[N(a)\cup N(c)=N(a)\cup N(e)=\{b,d,f\}=N(c)\cup N(e)\subset \overline{N(b)}\cap \overline{N(d)}\cap \overline{N(f)}\]
Hence $ac\in \{b,d,f\}$ , $ae\in \{b,d,f\}$ and $ce\in \{b,d,f\}.$

If $ac=b,$ then $b^2=0.$
Checking $a-$ table, we get that if $ac=b,$ then \\$ce\notin \{b,d,f\}.$

If $ac=d,$ then $d^2=0.$ Checking $c-$ table, we get $ce\notin \{b,d\}.$  If $ce=f,$ then $f^2=0$ and $N(c^2)\supset\{a,b,d\}.$  It follows that $c^2=b$ or $c^2=f.$  Checking the spectrum of $b$ and $f,$ we have $c^2\not= b$ or $c^2\not= f.$

If $ac=f$ and $ce=b,$ then by checking $c-$ table, we get $N(c^2)\supset \{b,d,e\}.$  Hence $c^2\in \{d,f\}.$ Checking the spectrum of $d$ and $f$, we get $c^2\notin \{d,f\}.$

If $ac=f,$ then by checking $c-$ table, we get $ce\not=d.$

If $ce=f,$ then by checking $e-$ table, we get $N(ae)\supset \{b,c,d,f\}.$ Hence $ae\in \{b,d\}.$  If $ae=b$ or $ae=d,$ by checking $e-$ table, we have that $e^2$ is undefined. 
\end{example}
\begin{example} $G799$ is not a zero-divisor graph.  $G799$ can be defined by $N(1)=\{2,3,6\},$ $N(2)=\{1,3,6\},$ $N(3)=\{1,2,4,6,7\},$ $N(4)=\{3,5\},$ $N(5)=\{4,6\},$ $N(6)=\{1,2,3,5,7\},$ $N(7)=\{3,6\}.$
We have the following table.
\begin{center}
\begin{tabular}{|l|c|r|r|r|r|r|r|}
\hline
$\bullet$&$1$&$2$&$3$&$4$&$5$&$6$&$7$  \\ \hline
$1$&&$0$&0&6&3&0& \{1,2,3,6\} \\ \hline

$2$&$0$&&$0$&6&3&0&  \\ \hline
$3$&0&$0$&0&0&3&0&0  \\ \hline
$4$&6&$6$&$0$&&0&6&6 \\ \hline
$5$&3&$3$&$3$&$0$&&$0$&3  \\ \hline
$6$&0&0&0&6&$0$&0&0 \\ \hline
$7$&\{1,2,3,6\}&&0&6&$3$&$0$&\\ \hline
\end{tabular}
\end{center}
Checking $1-$table, we can't define $1\times 7.$ 
\end{example}
\begin{example} $G803$ is not a zero-divisor graph. $G803$ can be defined by $N(1)=\{2,3\},$ $N(2)=\{1,3,5,6,7\},$ $N(3)=\{1,2,4,6\},$ $N(4)=\{3,5\},$ $N(5)=\{2,4,6\},$ $N(6)=\{2,3,5,7\},$ $N(7)=\{2,6\}.$  We have the following table.
\begin{center}
\begin{tabular}{|l|c|r|r|r|r|r|r|}
\hline
$\bullet$&$1$&$2$&$3$&$4$&$5$&$6$&$7$  \\ \hline
$1$&&$0$&0&&3&&  \\ \hline

$2$&$0$&&$0$&2&0&0&$0$  \\ \hline
$3$&0&$0$&0&0&3&0&3  \\ \hline
$4$&&$4$&$0$&&0&& \\ \hline
$5$&$3$&$0$&$3$&$0$&&$0$&  \\ \hline
$6$&&0&0&&$0$&&0 \\ \hline
$7$&&$0$&3&&&$0$&\\ \hline
\end{tabular}
\end{center}
Checking $5-$ table, we get $5\times 5=5$ and $5\times 7=5.$  Checking $7-$ table, we get $1\times 7=3.$ It follows $4\times 7=2$ and $2\times 2=0.$  But then $0=2\times 2=2\times 4\times 7=4\times 7=2.$ This is a contradiction.
\end{example}
\begin{example} $G808$ is not a zero-divisor graph.
$G808$ can be defined by  $N(1)=\{2,4,5,6,7\},$ $N(2)=\{1,3\},$ $N(3)=\{2,5,6\},$ $N(4)=\{1,5\},$ $N(5)=\{1,3,4,6\},$ $N(6)=\{1,3,5,7\},$ $N(7)=\{1,6\}.$ We have the following table.
\begin{center}
\begin{tabular}{|l|c|r|r|r|r|r|r|}
\hline
$\bullet$&$1$&$2$&$3$&$4$&$5$&$6$&$7$  \\ \hline
$1$&&$0$&1&0&0&0&0  \\ \hline

$2$&$0$&&$0$&&5&6&  \\ \hline
$3$&3&$0$&&1&0&0&1  \\ \hline
$4$&0&&$1$&&0&6& \\ \hline
$5$&$0$&$5$&$0$&$0$&&$0$&5  \\ \hline
$6$&0&6&0&6&$0$&&0 \\ \hline
$7$&0&&1&&5&$0$&\\ \hline
\end{tabular}
\end{center}
Checking $2-$table, we get $2\times 2=2.$ It follows that $2\times 4=6$ and $2\times 7=5.$ Checking  $7-$table,  we get $N(4\times 7)\supseteq \{1,2,3,5,6\}.$  Hence we can't define $4\times 7$.
\end{example}
\begin{example}\label{eG893} $G893$ is not a zero-divisor graph. $G893$ can be defined by $N(1)=\{2\},$ $N(2)=\{1,3,4,6,7\},$ $N(3)=\{2,4,6,7\},$ $N(4)=\{2,3,5,7\},$ $N(5)=\{4,6\},$ $N(6)=\{2,3,5,7\},$ $N(7)=\{2,3,4,6\}.$ We have the following table.
\begin{center}
\begin{tabular}{|l|c|r|r|r|r|r|r|}
\hline
$\bullet$&$1$&$2$&$3$&$4$&$5$&$6$&$7$  \\ \hline
$1$&&$0$&&\{4,6\}&\{2,3,7\}&\{4,6\}&  \\ \hline
$2$&0&&0&0&2&0&0  \\ \hline

$3$&&$0$&&0&&0&0  \\ \hline
$4$&\{4,6\}&0&$0$&&0&\{4,6\}&0 \\ \hline
$5$&\{2,3,7\}&$2$&&0&&0&  \\ \hline

$6$&\{4,6\}&0&0&\{4,6\}&$0$&&0 \\ \hline
$7$&&0&0&0&&$0$&\\ \hline
\end{tabular}
\end{center}
Checking $1$-table.  If $1\times 5=2$ or $1\times 5=3,$ then $N(1\times 7)\supseteq \{2,3,4,5,6\}.$  It follows that we can't define $1\times 7.$  Hence $1\times 5=7.$  But then $N(1\times 3)\supseteq \{2,4,5,6,7\}.$ Hence we can't define $1\times 3.$
\end{example}
\begin{example} $G899$ is not a zero-divisor graph.
$G899$ can be defined by $N(1)=\{2,3,5\},$ $N(2)=\{1,3,6\},$ $N(3)=\{1,2,4,5,6\},$ $N(4)=\{3,5,6\},$ $N(5)=\{1,3,4,6\},$ $N(6)=\{2,3,4,5,7\},$ $N(7)=\{6\}.$ We have the following table.
\begin{center}
\begin{tabular}{|l|c|r|r|r|r|r|r|}
\hline
$\bullet$&$1$&$2$&$3$&$4$&$5$&$6$&$7$  \\ \hline
$1$&&$0$&0&&0&6&  \\ \hline
$2$&0&&0&&&0&  \\ \hline

$3$&0&$0$&&0&0&0&3  \\ \hline
$4$&&&$0$&&0&0& \\ \hline
$5$&0&&0&0&&0&  \\ \hline

$6$&6&0&0&0&$0$&&0 \\ \hline
$7$&&&3&&&$0$&\\ \hline
\end{tabular}
\end{center}
We know that $1\times 7=3$ or $1\times 7=6.$

We have $1\times 4=3$ or $1\times 4=6.$ If $1\times 4=3,$ then $3\times 3=0.$   But then we get $0=(1\times 7)\times 4=3.$  This is a contradiction. Hence we can't define $1\times 7.$

If $1\times 4=6,$ then $6\times 6=0.$ Hence $1\times 1=1.$  Let $1\times 7=3.$ Then it follows that $1\times 7=0.$ This is a contradiction. Hence $1\times 7=6.$

Checking $7-$table, we get $7\times 7=3.$ But then $N(4\times 7)\supseteq \{1,3,6,7\}.$  Hence we can't define $4\times 7.$
\end{example}
\begin{example}\label{eG907}The connected graph $G907$  with $V(G)=\{a,b, c,x, y,z,w\}$ and $E(G)$ defined by $N(a)=\{c,x,y,z,w\},$ $N(b)=\{c,x,y,z\},\,  N(c)=\{a,b,x,y\}, N(x)=\{a,b,c,z\},N(y)=\{a,b,c\},\, N(z)=\{a,b,x\},\, N(w)=\{a\}$ is not a zero-divisor graph. 
\[N(a)\cup N(b)=\{c,x,y,z,w\}\]
\[N(b)\cup N(w)=\{a,c,x,y,z\}\]
\[N(c)\cup N(z)=N(c)\cup N(w)=\{a,b,x,y\}\]
\[N(x)\cup N(y)=N(c)\cup N(w)=\{a,b,c,z\}\]
we have the following multiplication table.
\begin{center}
\begin{tabular}{|l|c|r|r|r|r|r|r|}
\hline
$\bullet$&$a$&$b$&$c$&$x$&$y$&$z$&$w$  \\ \hline
$a$&0&$a$&0&0&0&0&0  \\ \hline
$b$&$a$&&0&0&0&0&$a$  \\ \hline

$c$&0&0&&$0$&$0$&$c$&$c$  \\ \hline
$x$&0&$0$&0&&$x$&$0$&$x$  \\ \hline
$y$&$0$&0&$0$&$x$&&&  \\ \hline
$z$&0&0&$c$&$0$&&& \\ \hline
$w$&0&$a$&$c$&$x$&&&\\ \hline
\end{tabular}
\end{center}
Checking the $w-$ table, we get that $w^2$ is undefined.
\end{example}
\begin{example} $G917$ is  not a zero-divisor graph.
$G917$ can be defined by $N(1)=\{2,7\},$ $N(2)=\{1,3,5,7\},$ $N(3)=\{2,4,5,7\},$ $N(4)=\{3,5\},$ $N(5)=\{2,3,4,6,7\},$ $N(6)=\{5,7\},$ $N(7)=\{1,2,3,5,6\}.$
We have the following table.
\begin{center}
\begin{tabular}{|l|c|r|r|r|r|r|r|}
\hline
$\bullet$&$1$&$2$&$3$&$4$&$5$&$6$&$7$  \\ \hline
$1$&&$0$&&&5&&0  \\ \hline
$2$&0&&0&\{2,7\}&0&\{2,7\}&0  \\ \hline

$3$&&$0$&&0&0&&0  \\ \hline
$4$&&\{2,7\}&0&&0&&4 \\ \hline
$5$&5&0&0&0&&0&0  \\ \hline

$6$&&\{2,7\}&&&$0$&&0 \\ \hline
$7$&0&0&0&4&0&$0$&\\ \hline
\end{tabular}
\end{center}
Suppose $2\times 4=7.$ Checking $2-$table, we get $N(2\times 6)\supseteq \{1,3,4,5,7\}.$  Hence we can't define $2\times 6.$  It follows that $2\times 4=2.$ Checking $4-$table, we have $0=2\times 7=2\times 4\times 7=2\times 4=2.$ This is a contradiction.

\end{example}
\begin{example}\label{eG918} The connected graph $G918$ which is defined by $N(a)=\{b,f\},\, N(b)=\{a,c,e,f,g\},$ $N(c)=\{b,d,e,f\},$ $N(d)=\{c,e\},N(e)=\{b,c,d,f\},N(f)=\{a,b,c,e,g\} , N(g)=\{b,f\}$ is not a zero-divisor graph. If we remove $g,$ we get graph~(\ref{type4}) in Theorem~\ref{DGSW}.

We claim that if $a^2=g,$ then $g^2\not=a.$

Suppose $a^2=g$ and $g^2=a.$   Then $(ag)^2=agag=a^2g^2=ga=ag.$  It follows that $a(ag)=g^2=a$.

Notice $N(a)\cup N(g)=\{b,f\}\subset N(a)\cap N(g)\cap \overline{N(b)}\cap  N(c) \cap N(e) \cap \overline{N(f)}.$ 

If $ag=a,$ then $a(ag)=a^2=g$.  This is a contradiction.   If $ag=g,$ then $ag=a(ag)=a.$ Another contradiction. 
Hence  $ag\not=a$ or $ag\not=g.$

If $ag=b,$  then $a=a(ag)=ab=0.$  This is a contradiction.

Since $N(a)\cup N(c)=\{b,d,e,f\}\subset N(c)\cup \overline{N(e)}$, $ac\in \{c,e\}.$   It follows that $ac\not= a.$  Hence $ag\not= c.$ 

Since $N(a)\cup N(e)=\{b,c,d,f\}\subset \overline{N(c)}\cap N(e)$, $ae\in \{c,e\}.$  It follows that $ae\not=a.$ Hence $ag\not=e.$

If $ag=f,$ then $a=a(ag)=af=0.$  Contradiction. 

 It is easy to check that $N(x)\cup N(y)$ is not a subset of $\overline{N(a)}$ or $\overline{N(g)}$ unless $x=a,\, y=g$ or $x=g,\, y=a.$  
 
 Suppose $G$ is a zero-divisor graph, then by the above argument, we get the graph formed by removing $a$ or $g$ form $G$ is a zero-divisor graph.  But this is a contradiction to the fact showed in \cite{DGSW}.

\end{example}
\begin{example}
$G928$ is not a zero-divisor graph. $G928$ can be defined by $N(1)=\{2,5\},$ $N(2)=\{1,3,5,6,7\},$ $N(3)=\{2,4,6\},$ $N(4)=\{3,5\},$ $N(5)=\{1,2,4,6,7\}$,$N(6)=\{2,3,5,7\},$ $N(7)=\{2,5,6\}.$ We have the following table.
\begin{center}
\begin{tabular}{|l|c|r|r|r|r|r|r|}
\hline
$\bullet$&$1$&$2$&$3$&$4$&$5$&$6$&$7$  \\ \hline
$1$&&$0$&5&&0&&  \\ \hline
$2$&0&&0&2&0&0&0  \\ \hline

$3$&5&$0$&3&0&5&0&5  \\ \hline
$4$&&2&0&&0&& \\ \hline
$5$&0&0&5&0&0&0&0  \\ \hline

$6$&&0&0&&0&&0 \\ \hline
$7$&&0&5&&0&$0$&\\ \hline
\end{tabular}
\end{center}
We know $1\times 6=2$ or $1\times 6=6.$   Suppose $1\times 6=2.$ Checking $6-$table, we get $4\times 6=6.$ Then checking $4-$table, we get $4\times 4=4.$  It follows that we can't define $1\times 4.$

Suppose $1\times 6=6.$ Checking $1-$table, we get $1\times 1=6.$ It follows that $1\times 7=2.$  But then we can't define $1\times 4.$
\end{example}
\begin{example} $G933$ is not a zero-divisor graph.  $G933$ can be defined by $N(1)=\{2,5,6,7\},$ $N(2)=\{1,3,5,6,7\},$ $N(3)=\{2,4\},$ $N(4)=\{3,5\},$ $N(5)=\{1,2,4,6,7\},$ $N(6)=\{1,2,5\},$ $N(7)=\{1,2,5\}.$ We have the following table. 
\begin{center}
\begin{tabular}{|l|c|r|r|r|r|r|r|}
\hline
$\bullet$&$1$&$2$&$3$&$4$&$5$&$6$&$7$  \\ \hline
$1$&&$0$&5&2&0&0&0  \\ \hline
$2$&0&0&0&2&0&0&0  \\ \hline

$3$&5&$0$&&0&5&5&5  \\ \hline
$4$&2&2&0&&0&2&2 \\ \hline
$5$&0&0&5&0&0&0&0  \\ \hline

$6$&0&0&5&2&0&& \\ \hline
$7$&0&0&5&2&0&&\\ \hline
\end{tabular}
\end{center}
Checking $6-$table, we get $N(6\times 7)\supseteq \{1,2,3,4,5\}.$ Hence we can't define $6\times 7.$
\end{example}
\begin{example} $G938$ is not a zero-divisor graph. $G938$ can be defined by $N(1)=\{2,6,7\},$ $N(2)=\{1,3,5,6\},$ $N(3)=\{2,4,6,7\},$ $N(4)=\{3,7\},$ $N(5)=\{2,7\},$ $N(6)=\{1,2,3,7\},$ $N(7)=\{1,3,4,5,6\}.$ We have the following table.
\begin{center}
\begin{tabular}{|l|c|r|r|r|r|r|r|}
\hline
$\bullet$&$1$&$2$&$3$&$4$&$5$&$6$&$7$  \\ \hline
$1$&&$0$&3&&&0&0  \\ \hline
$2$&0&&0&7&0&0&7  \\ \hline

$3$&3&$0$&&0&3&0&0  \\ \hline
$4$&&7&0&&&6&0 \\ \hline
$5$&&0&3&&&6&0  \\ \hline

$6$&0&0&0&6&6&&0 \\ \hline
$7$&0&7&0&0&0&0&0\\ \hline
\end{tabular}
\end{center}
Checking $4-$table, we get $4\times 4=4\times 5=6.$  Then we can't define $1\times 4$ because $N(1\times 4)\supseteq \{2,3,4,5,6,7\}.$

\end{example}
\begin{example}$G953$ is not a zero-divisor graph. $G953$ can be defined by $N(1)=\{2,5\},$ $N(2)=\{1,3,7\},$ $N(3)=\{2,4,5,6,7\},$ $N(4)=\{3,5,6\},$ $N(5)=\{1,3,4,6,7\},$ $N(6)=\{3,4,5\},$ $N(7)=\{2,3,5\}.$ We have the following table.

\begin{center}
\begin{tabular}{|l|c|r|r|r|r|r|r|}
\hline
$\bullet$&$1$&$2$&$3$&$4$&$5$&$6$&$7$  \\ \hline
$1$&&$0$&3&3&0&3&  \\ \hline
$2$&0&&0&5&5&5&0  \\ \hline

$3$&3&$0$&0&0&0&0&0  \\ \hline
$4$&3&5&0&&0&0&3 \\ \hline
$5$&0&5&0&0&0&0&0  \\ \hline

$6$&3&5&0&0&0&&3 \\ \hline
$7$&&0&0&3&0&3&\\ \hline
\end{tabular}
\end{center}
Checking $1-$table, we get $1\times 1=1$ and $1\times 7=3.$  But then $0=3\times 4=1\times 7\times 4=1\times 3=3.$ This is a contradiction.
\end{example}
\begin{example}\label{eG1024}The connected graph $G1024$  with $V(G)=\{a,b, c,x, y,z,w\}$ and $E(G)$ defined by $N(a)=N(b)=\{c,x,y,z\},\,  N(c)=N(z)=\{a,b,x,y\}, N(x)=\{a,b,c,z,w\},N(y)=\{a,b,c,z\}, N(w)=\{x\}$ is not a zero-divisor graph. If we remove $w$ from the graph, we get the  graph~(\ref{graph24}) in Theorem~\ref{B} which is a zero-divisor graph with six vertices.

Since \[N(x)\cup N(y)=\{a,b,c, z,w\}\] and
\[N(y)\cup N(w)=\{a,b,c, x,z\}\]
one gets that $xy=x$, $yw=x$ and $x^2=0.$  Hence we have the following multiplication table.
\begin{center}
\begin{tabular}{|l|c|r|r|r|r|r|r|}
\hline
$\bullet$&$a$&$b$&$c$&$x$&$y$&$z$&$w$  \\ \hline
$a$&&&0&0&0&0&  \\ \hline
$b$&&&0&0&0&0&  \\ \hline

$c$&0&0&&$0$&$0$&&  \\ \hline
$x$&0&$0$&$0$&0&$x$&0&$0$  \\ \hline
$y$&$0$&0&$0$&$x$&&$0$&$x$  \\ \hline
$z$&0&0&&$0$&$0$&& \\ \hline
$w$&&&&$0$&$x$&&\\ \hline
\end{tabular}
\end{center}
Since \[N(a)\cup N(w)=N(b)\cup N(w)=\{c,x,y,z\}\]
and \[N(c)\cup N(w)=N(z)\cup N(w)=\{a,b,x,y\}\]
one gets $aw\in\{a,b\},$ $bw\in \{a,b\},$ $cw\in \{c,z\},$ and $zw\in \{c,z\}.$

By checking $w-$ table and the spectrum of each vertex in $G$, it follows that $w^2$ is undefined.
\end{example}
\begin{example} $G1030$ is not a zero-divisor graph. $G1030$ can be defined by $N(1)=\{2,5,6\},$ $N(2)=\{1,3,4,5,6\},$ $N(3)=\{2,4,5,6\},$ $N(4)=\{2,3\},$ $N(5)=\{1,2,3,6,7\},$ $N(6)=\{1,2,3,5,7\},$ and $N(7)=\{5,6\}.$ We have the following table.
\begin{center}
\begin{tabular}{|l|c|r|r|r|r|r|r|}
\hline
$\bullet$&$1$&$2$&$3$&$4$&$5$&$6$&$7$  \\ \hline
$1$&&$0$&&&0&0&  \\ \hline
$2$&0&&0&0&0&0&2  \\ \hline

$3$&&&0&0&0&0&  \\ \hline
$4$&&0&0&&&& \\ \hline
$5$&0&0&0&&&0&0  \\ \hline

$6$&0&0&0&&0&&0 \\ \hline
$7$&&2&&&0&0&\\ \hline
\end{tabular}
\end{center}
We know that $3\times 7=2$ or $3\times 7=3.$ If $3\times 7=2,$ then $N(1\times 3)\supseteq \{2,4,5,6,7\}.$  Hence we can't define $1\times 3.$  It follows that $3\times 7=3.$  
 
Checking $4-$table, we get $7\times 7=7$ and $4\times 7=2$ or $4\times 7=3.$  We know that $4\times 5\in \{5,6\}$ and $4\times 6=\{5,6\}.$  Checking $4-$table, we get $N(4\times 4\supseteq\{2,3,7\}.$  Hence $4\times 4\in \{0, 5,6\}.$  It follows that we can't define $4\times 4.$ 
\end{example}
\begin{example}
$G1034$ is not a zero divisor graph. $G1034$ can be defined by $N(1)=\{2,3,4,5\},$ $N(2)=\{1,3,4,5\},$ $N(3)=\{1,2,4,6\},$ $N(4)=\{1,2,3,5,7\},$ $N(5)=\{1,2,4,6,7\},$ $N(6)=\{3,5\}$, and $N(7)=\{4,5\}.$  We have the following table.

\begin{center}
\begin{tabular}{|l|c|r|r|r|r|r|r|}
\hline
$\bullet$&$1$&$2$&$3$&$4$&$5$&$6$&$7$  \\ \hline
$1$&&$0$&0&0&0&&  \\ \hline
$2$&0&&0&0&0&&  \\ \hline

$3$&0&0&&0&5&0&5  \\ \hline
$4$&0&0&0&&0&4&0 \\ \hline
$5$&0&0&5&0&0&0&0  \\ \hline

$6$&&&0&4&0&& \\ \hline
$7$&&&5&0&0&&\\ \hline
\end{tabular}
\end{center}
We know that $6\times 7\in\{1,2,4\}.$   Let's Check $7-$table.  If $6\times 7=1$ or $6\times 7=4,$ then we can't define $2\times 7$ since $N(2\times 7)\supseteq \{3,4,5,6\}.$   If $6\times 7=2,$ then $N(1\times 7)\supseteq \{3, 4,5,6\}.$  Hence we can't define $1\times 7.$
\end{example}
\begin{example}
$G1043$ is not a zero-divisor graph.  $G1043$ can be defined by $N(1)=\{2,3\},$ $N(2)=\{1,3,4,5,6\},$ $N(3)=\{1,2,4,5,7\},$ $N(4)=\{2,3,5\},$ $N(5)=\{2,3,4,6,7\},$ $N(6)=\{2,5,7\},$ and $N(7)=\{3,5,6\}.$ We have the following table.
\begin{center}
\begin{tabular}{|l|c|r|r|r|r|r|r|}
\hline
$\bullet$&$1$&$2$&$3$&$4$&$5$&$6$&$7$  \\ \hline
$1$&&$0$&0&&5&&  \\ \hline
$2$&0&&0&0&0&0&2  \\ \hline

$3$&0&0&&0&0&3&0  \\ \hline
$4$&&0&0&&0&& \\ \hline
$5$&5&0&0&0&&0&0  \\ \hline

$6$&&0&3&&0&&0 \\ \hline
$7$&&2&0&&0&0&\\ \hline
\end{tabular}
\end{center}
We know $1\times 7\in \{2,5\}.$ Suppose $1\times 7=2.$ Checking $1-$table, we get $1\times 1=5$ and $1\times 4=3.$  Then we can't define $1\times 6$ since $N(1\times 6)\supseteq \{1,2,3,5,7\}$ and $(1\times 6)\times 4=3.$ Suppose $1\times 7=5.$ Checking $7-$table, we get $7\times 7=2.$ Then we can't define $4\times 7$ since $N(4\times 7)\supseteq\{1,2,3,5,6,7\}.$

\end{example}
\begin{example}\label{eG1044}The connected graph $G1044$  with $V(G)=\{a,b, c,d, e,f,g\}$ and $E(G)$ defined by $N(a)=\{b,d,g\},\, N(b)=\{a, c,e,f,g\}, N(c)=N(f)=\{b,e,g\}, N(d)=\{a,e\},N(e)=\{b,c,d,f,g\}, N(g)=\{a,b,c,e,f\}$ is not a zero-divisor graph. If we remove $g,$ we get graph~(\ref{type2}) in Theorem~\ref{DGSW}.

Notice \[N(a)\cup N(c)=N(a)\cup N(f)=\{b,d,e,g\}\subset \overline{N(e)}\]
and \[N(a)\cup N(e)=\{b,c,d,f,g\}=N(e)\]
we have the following multiplication table.
\begin{center}
\begin{tabular}{|l|c|r|r|r|r|r|r|}
\hline
$\bullet$&$a$&$b$&$c$&$d$&$e$&$f$&$g$  \\ \hline
$a$&&0&$e$&0&$e$&$e$&$0$  \\ \hline

$b$&0&&0&&0&0&0  \\ \hline
c&$e$&$0$&&&$0$&&$0$  \\ \hline

$d$&0&&&&0&&  \\ \hline
$e$&$e$&0&$0$&0&0&0&$0$  \\ \hline
$f$&$e$&0&&&0&&$0$\\  \hline
$g$&0&0&$0$&&$0$&$0$&  \\ \hline
\end{tabular}
\end{center}
Notice $N(c)\cup N(d)=\{a,b,e,g\}$ is a subset of $\overline{N(b)}$ or a subset of $\overline{N(g)}.$ Hence $cd=b$ or $cd=g.$

If  $cd=b,$ then $dcf=bf=0.$  If $cd=g,$ then $dcf=gf=0.$  In either cases, we have $d\in N(cf).$  By checking $c-$ table, we have $N(cf)\supset \{a,b,d,e,g\}.$  Hence $cf$ is undefined.
\end{example}
\begin{example}\label{eG1060}
$G1060$ is not a zero-divisor graph. $G1060$ can be defined by $N(1)=\{2,3,4,5\},$ $N(2)=\{1,3,6\},$ $N(3)=\{1,2,4,6,7\},$ $N(4)=\{1,3,5,6,7\},$ $N(5)=\{1,4,6\},$ $N(6)=\{2,3,4,5\},$ $N(7)=\{3,4\}.$  We have the following table.
\begin{center}
\begin{tabular}{|l|c|r|r|r|r|r|r|}
\hline
$\bullet$&$1$&$2$&$3$&$4$&$5$&$6$&$7$  \\ \hline
$1$&&$0$&0&0&0&&  \\ \hline
$2$&0&&0&4&&0&  \\ \hline

$3$&0&0&&0&3&0&0  \\ \hline
$4$&0&4&0&&0&0&0 \\ \hline
$5$&0&&3&0&&0&  \\ \hline

$6$&&0&0&0&0&& \\ \hline
$7$&&&0&0&&&\\ \hline
\end{tabular}
\end{center}
We know that $2\times 5\in \{3,4\}.$  Suppose $2\times 5=3.$ Checking $7-$table, we get $2\times 7=4.$ Then we can't define $2\times 2$ since $N(2\times 2\supseteq \{1,3,5,6\}$ and $(2\times 2)\times 7=4.$ Suppose $2\times 5=4.$ Checking $5-$table, we get $5\times 5=3.$ Then we can't define $5\times 7$ since $N(5\times 7)\supseteq \{1,2,3,4,5,6\}.$
\end{example}
\begin{example} $G1083$ is not a zero-divisor graph. $G1083$ can be defined by $N(1)=\{2,5,6,7\},$ $N(2)=\{1,3,6\},$ $N(3)=\{2,6,7\},$ $N(4)=\{5,6,7\},$ $N(5)=\{1,4,7\},$ $N(6)=\{1,2,3,4,7\},$ $N(7)=\{1,3,4,5,6\}.$ We have the following table.
\begin{center}
\begin{tabular}{|l|c|r|r|r|r|r|r|}
\hline
$\bullet$&$1$&$2$&$3$&$4$&$5$&$6$&$7$  \\ \hline
$1$&&$0$&1&1&0&0&0  \\ \hline
$2$&0&&0&7&&0&7  \\ \hline

$3$&1&0&&1&6&0&0  \\ \hline
$4$&1&7&1&&0&0&0 \\ \hline
$5$&0&&6&0&&6&0  \\ \hline

$6$&0&0&0&0&6&0&0 \\ \hline
$7$&0&7&0&0&0&0&0\\ \hline
\end{tabular}
\end{center}
We know $2\times 5\in \{6,7\}.$ Suppose $2\times 5=6.$ Checking $2-$ table, we can't define $2\times 2$ since $N(2\times 2)\supseteq\{1,3,5,6\}$ and $(2\times 2)\times 4=7.$  Suppose $2\times 5=7.$ Checking $5-$table, we can't define $5\times 5$ since $N(5\times 5)\supseteq \{1,2,4,7\}$ and $(5\times 5)\times 3=6.$
\end{example}
\begin{example} $G1120$ is not a zero-divisor graph. $G1120$ can be defined by $N(1)=\{2,3,4,5,6\},$ $N(2)=\{1,4,5\},$ $N(3)=\{1,4,6\},$ $N(4)=\{1,2,3,5,6\},$ $N(5)=\{1,2,4,6,7\},$ $N(6)=\{1,3,4,5,7\},$ $N(7)=\{5,6\}.$  We have the following table.
\begin{center}
\begin{tabular}{|l|c|r|r|r|r|r|r|}
\hline
$\bullet$&$1$&$2$&$3$&$4$&$5$&$6$&$7$  \\ \hline
$1$&&$0$&0&0&0&0&  \\ \hline
$2$&0&&&0&0&6&  \\ \hline

$3$&0&&&0&5&0&  \\ \hline
$4$&0&0&0&&0&0& \\ \hline
$5$&0&0&5&0&&0&0  \\ \hline

$6$&0&6&0&0&0&&0 \\ \hline
$7$&&&&&0&0&\\ \hline
\end{tabular}
\end{center}
We know $2\times 3\in \{1,4,5,6\}$. Suppose $2\times 3=1.$  Checking the $2-$table, we get $2\times 2=6.$  It follows that $N(2\times 7)\supseteq \{1,2, 4,5,6\}.$   Hence $2\times 7\in \{1, 4,5\}.$  But $2\times 7\not=1$ otherwise $7\times 1=7\times (2\times 3)=(7\times 2)\times 3=1\times 3=0.$ This is a contradiction. If $2\times 7=4,$ then $7\times 1=7\times (2\times 3)=(7\times 2)\times 3=4\times 3=0.$ Contradiction again.  If $2\times 7=5,$ then $7\times 1=7\times (2\times 3)=(7\times 2)\times 3=5\times 3=5.$ But $1\times 7\not=5$ since $N(1)\cup N(7)=N(1)$ and $N(1)$ is not a subset of $\overline{N(5)}.$

Suppose $2\times 3=4.$ Checking $2-$table, we get $2\times 2=6.$ It follows that $2\times 7=5.$ But then $7\times 4=7\times (2\times 3)=(7\times 2)\times 3=5\times 3=5.$  This is not true because $N(4)\cup N(7)=N(4)$ and $N(4)$ is not a subset of $\overline{N(5)}$.

 Suppose $2\times 3=5.$  Checking $3-$ table, we get $3\times 3=3.$ It follows that $3\times 7=5.$ Checking $7-$ table, one gets $N(7\times 7)\supseteq \{3,5,6\}.$ It follows that $7\times 7\in \{0, 1,4,6\}.$ If $7\times 7=0$ or $7\times 7=1,$ we get $4\times 7=0.$ This is a contradiction. If $7\times 7=4 $ or $7\times 7=6$, we get $N(1\times 7)\supseteq \{2,3,4,5,6,7\}.$ Hence we can't define $1\times 7.$
 
 Suppose $2\times 3=6.$ Checking $2-$table, we get $2\times 2=2.$ It follows that $2\times 7=6.$  Checking $7-$ table, we get $7\times 7\in \{0, 1,4,5\}.$ If $7\times 7\in \{0, 4,5\},$ then we can't define $1\times 7$ since $N(1\times 7)\supseteq \{2,3,4,5,6,7\}.$ If $7\times 7=1,$ then we can't define $4\times 7$ since $N(4\times 7)\supseteq \{1,2,3,5,6,7\}.$  
\end{example}
\begin{example}\label{eG1130} The connected graph $G1130$ which is defined by $N(a)=\{b,d,e,f\},\, N(b)=\{a,c,d,e,g\},$ $N(c)=\{b,d\},$ $N(d)=\{a,b,c,e,g\},N(e)=\{a,b,d,f,g\}, N(f)=\{a,e,g\},N(g)=\{b,d,e,f\} $ is not a zero-divisor graph. If we remove $g,$ we get graph~(\ref{type4}) in Theorem~\ref{DGSW}.

We have $ce=e.$  It follows $c^2\ne a.$  Otherwise $c^2e=ae=0\ne e=c(ce).$ Similarly, $c^2\ne g.$ 
If $c^2=c,$ then $cf=e.$  ($cf\ne b.$ Otherwise $c^2f=cf=b\ne 0=c(cf).$  $cf\ne d.$ Otherwise $c^2f=cf=d\ne 0=c(cf).$)
If $cf=e,$ then $f^2\ne f.$ Otherwise, $cf^2=cf=e\ne 0=(cf)f.$

\begin{center}
\begin{tabular}{|l|c|r|r|r|r|r|r|}
\hline
$\bullet$&$a$&$b$&$c$&$d$&$e$&$f$&$g$  \\ \hline
$a$&&$0$&&$0$&$0$&0&  \\ \hline

$b$&$0$&&$0$&$0$&$0$&&$0$  \\ \hline
$c$&&$0$&&$0$&$e$&&  \\ \hline
$d$&$0$&$0$&$0$&&$0$&&$0$  \\ \hline
$e$&$0$&$0$&$e$&$0$&&$0$&$0$  \\ \hline
$f$&$0$&&&&$0$&&$0$  \\ \hline
$g$&&$0$&&$0$&$0$&$0$&  \\ \hline

\end{tabular}
\end{center}
\begin{center}
\begin{tabular}{|l|c|r|r|r|r|r|r|}
\hline
$c$&$ac$&$0$&$c^2$&$0$&$e$&$cf$&$cg$  \\ \hline
$ac$&&$0$&&$0$&$0$&0&  \\ \hline

$0$&$0$&0&$0$&$0$&$0$&0&$0$  \\ \hline
$c^2$&&$0$&&$0$&$e$&&  \\ \hline
$0$&$0$&$0$&$0$&0&$0$&0&$0$  \\ \hline
$e$&$0$&$0$&$e$&$0$&&$0$&$0$  \\ \hline
$cf$&$0$&0&&0&$0$&&$0$  \\ \hline
$cg$&&$0$&&$0$&$0$&$0$&  \\ \hline

\end{tabular}
\end{center}
We get $c^2=c,$ or $c^2=e.$ 

If $c^2=e,$ then
\begin{center}
\begin{tabular}{|l|c|r|r|r|r|r|r|}
\hline
$c$&$ac$&$0$&$e$&$0$&$e$&$cf$&$cg$  \\ \hline
$ac$&&$0$&0&$0$&$0$&0&  \\ \hline

$0$&$0$&0&$0$&$0$&$0$&0&$0$  \\ \hline
$e$&0&$0$&$e$&$0$&$e$&0&0  \\ \hline
$0$&$0$&$0$&$0$&0&$0$&0&$0$  \\ \hline
$e$&$0$&$0$&$e$&$0$&$e$&$0$&$0$  \\ \hline
$cf$&$0$&0&0&0&$0$&&$0$  \\ \hline
$cg$&&$0$&0&$0$&$0$&$0$&  \\ \hline

\end{tabular}
\end{center}
It follows that $N(ac)\supset \{b,c,d,e,f\}.$  Hence we can not define $ac.$  We get $c^2=c.$ It follows $cf=e,\, e^2=0.$
\begin{center}
\begin{tabular}{|l|c|r|r|r|r|r|r|}
\hline
$f$&$0$&$bf$&$e$&$df$&$0$&$f^2$&$0$  \\ \hline
$0$&0&$0$&0&$0$&$0$&0&0  \\ \hline

$bf$&$0$&&$0$&$0$&$0$&&$0$  \\ \hline
$e$&0&$0$&$e$&$0$&$0$&0&0  \\ \hline
$df$&$0$&$0$&$0$&&$0$&&$0$  \\ \hline
$0$&$0$&$0$&$0$&$0$&$0$&$0$&$0$  \\ \hline
$f^2$&$0$&&0&&$0$&&$0$  \\ \hline
$0$&0&$0$&0&$0$&$0$&$0$&0  \\ \hline

\end{tabular}
\end{center}
We know $f^2\in\{0,b,d\}.$ 
If $f^2=0,$ then $N(bf)\supset \{a,c,d,e,f,g\},$  we can't define $bf.$ 
If $f^2=b,$ then $N(df)\supset \{a,b,c,e,f,g\},$ can't define $df.$ If $f^2=d,$ then $N(bf)\supset \{a,c,d,e,f,g\},$  we can't define $bf.$

\end{example} 
\begin{example} $G1146$ is not a zero-divisor graph. $G1146$ can be defined by $N(1)=\{2,3,4,5\},$ $N(2)=\{1,3,7\},$ $N(3)=\{1,2,4,6,7\},$ $N(4)=\{1,3,5,6,7\},$ $N(5)=\{1,4,7\},$ $N(6)=\{3,4,7\},$ $N(7)=\{2,3,4,5,6\}.$ We have the following table.
\begin{center}
\begin{tabular}{|l|c|r|r|r|r|r|r|}
\hline
$\bullet$&$1$&$2$&$3$&$4$&$5$&$6$&$7$  \\ \hline
$1$&&$0$&0&0&0&&7  \\ \hline
$2$&0&&0&4&&&0  \\ \hline

$3$&0&0&&0&3&0&0  \\ \hline
$4$&0&4&0&&0&0&0 \\ \hline
$5$&0&&3&0&&&0  \\ \hline

$6$&&&0&0&&&0 \\ \hline
$7$&7&0&0&0&0&0&0\\ \hline
\end{tabular}
\end{center}
We know $2\times 5\in \{3,4\}.$  Suppose $2\times 5=3.$ Checking $2-$table, we get $2\times 6=4$ since $N(2\times 6)\supseteq \{1,3,4,5,6,7\}.$ Since $(2\times 2)\times 5=2\times (2\times 5)=2\times 3=0$ and $(2\times 2)\times 6=2\times (2\times 6)=2\times 4=4.$ It follows that $2\times 2\not=2$ or $2\times 2\not=4.$ Then we can't define $2\times 2.$ 

Suppose $2\times 5=4.$  Checking $5-$table, we get $5\times 6=3$ since $N(5\times 6)\supseteq \{1,2,3,4,7\}.$  We also get $3\times 3=(5\times 6)\times 3=5\times (6\times 3)=5\times 0=0.$ Since $(5\times 5)\times 3=(5\times 3)\times 5=3\times 5=3$ and $(5\times 5)\times 2=(5\times 2)\times 5=4\times 5=0$, we have $5\times 5\not=3$ or $5\times 5\not=5.$  Hence we can't define $5\times 5.$
\end{example}
\begin{example} $G1177$ is not a zero-divisor graph. $G1177$ can be defined by $N(1)=\{2,3,6\},$ $N(2)=\{1,3,4,5,6\},$ $N(3)=\{1,2,4,5,6\},$ $N(4)=\{2,3,5,6,7\},$ $N(5)=\{2,3,4,6,7\},$ $N(6)=\{1,2,3,4,5\},$ $N(7)=\{4,5\}.$ We know $1\times 4\in \{4,5\},$ $1\times 5\in \{4,5\},$ $2\times 7\in \{2,3,6\},$ $2\times 3\in \{2,3,6\},$ and $6\times 7\{2,3,6\}.$ 

Checking $1-$table, we get $1\times 7\in\{4,5\}$ otherwise we can't define $1\times 1.$

Checking $7-$table, we can't define $7\times 7.$
\end{example}
\section{Non zero-divisor graphs which are either dis-connected or connected but not satisfying the $\star$ condition}
\subsection{Dis-connected graphs are not zero-divisor graphs}
\begin{center}
$G209-G269$,$G275$, $G277$,$G281-G283$,$G285$,$G287-G313$,$G323$,$G330$, $G335$,$G345-G347$,$G352$,$G354-G378$,$G387$,$G397$,$G407$,$G417-G418$,$G451-G472$,$G496$, $G502$,$G582-G597$,$G611$,$G731-G739$,$G745$,$G879-G883$,
$G1010-G1011$,$G1107$, $G1172$.
\end{center}

\subsection{connected but not satisfying the $\star$ condition}
\begin{center}
$G273-G274$,$G276$,$G278-G280$,$G284$,$G286$,$G318$,$G320-G321$,$G324-G329$,$G331-G334$,$G336-G344$,$G348-G351$,$G353$,$G385-G386$, $G389$,$G391$, $G394-G396$,$G398-G404$,$G406$,$G408-G410$,$G412-G416$,$G419-G450$,$G484$,$G487-G489$,$G491-G492$,$G494-G495$, $G497-G501$, $G505-G506$,$G508-G512$,$G514-G521$,$G523-G524$,$G526-G550$,$G552-G581$,$G605$,$G608-G610$,$G615$, $G621-G623$,$G625$,$G628$,$G630$,$G632$, $G634$,$G636-G638$,$G640-G666$, $G673-G676$,$G679-730$,$G744$,$G756$,$G760-G761$,$G763$,$G765$,$G768-771$,$G773-G774$,$G776-779$,$G781-G785$,$G787-G789$,$G797$, $G802$, $G804$,$G806-G807$,$G809-G811$,$G816-G831$,$G833-G871$,$G873-G878$ ,$G892$,$G895$, $G900-G901$,$G903-G905$,$G910-G912$,$G926$,$G931$,$G935-G937$,$G940-G943$,$G945-G947$,$G949$,$G954-G956$,$G958-G971$,$G973-G974$,$G976-G1006$,
$G1021-G1023$,$G1033$,$G1038-G1042$,$G1051$,$G1054-G1055$,$G1058$,$G1061$,$G1063-G1066$,$G1068-G1071$,$G1073-G1076$,$G1082$, $G1084$,$G1086-1087$, $G1089-G1105$,$G1112$,$G1127$, $G1133$,$G1136$,$G1153-G1156$,$G1158-G1162$,$G1164-G1168$,$G1170-G1171$,$G1201$,$G1204$,$G1209$,$G1211-G1212$,
\end{center}

\end{document}